\documentclass[11pt]{article}
\usepackage{a4wide}

\usepackage[a4paper, margin=2.3cm]{geometry}
\usepackage{amsmath,amssymb,amsthm, comment, enumitem, array}
\usepackage{color}
\usepackage{parskip}
\usepackage{hyperref}
\usepackage{parskip}
\usepackage{setspace}
\usepackage{graphicx}
\usepackage{mathtools}
\usepackage[thinc]{esdiff}
\usepackage[euler]{textgreek}
\usepackage{subfig}
\usepackage{xcolor}

\newtheorem{theorem}{Theorem}[section]

\newtheorem{remark}{Remark}[section]
\newtheorem{example}[theorem]{Example}
\newtheorem{corollary}[theorem]{Corollary}
\newtheorem{lemma}[theorem]{Lemma}
\newtheorem{question}[theorem]{Question}

\newtheorem{conjecture}[theorem]{Conjecture}
\newtheorem*{definition*}{Definition}

\hypersetup{
    colorlinks=true,
    linkcolor = blue,
    citecolor=black,
    urlcolor=cyan,
}

\usepackage[capitalise, nameinlink]{cleveref}
\crefname{equation}{}{}
\crefname{figure}{{\sc Figure}}{{\sc Figure}}
\crefname{subsection}{Subsection}{Subsections}

\def\bC{\mathbb{C}}
\def\cN{\mathcal{N}}
\def\bZ{\mathbb{Z}}
\def\bF{\mathbb{F}}

\def\cP{\mathcal{P}}
\def\cS{\mathcal{S}}
\def\cM{\mathcal{M}}
\def\cE{\mathcal{E}}

\def\fN{\mathfrak{N}}

\def\bfo{\mathbf{0}}
\def\bfm{\mathbf{m}}
\def\bfl{\mathbf{l}}
\def\bfw{\mathbf{w}}
\def\bfx{\mathbf{x}}
\def\bfy{\mathbf{y}}
\def\bfz{\mathbf{z}}
\def\bfu{\mathbf{u}}
\def\bfv{\mathbf{v}}
\def\mod{\textsf{mod }}
\def\ord{\textsf{ord}}
\def\bfb{\mathbf{b}}
\def\bfc{\mathbf{c}}

\def\i{\mathsf{i}}

\begin{document}
\title{On the distance problem over finite p-adic rings}
\author{Thang Pham \thanks{Institute of Mathematics and Interdisciplinary Sciences, Xidian University. \newline
\hspace*{0.45cm} Email: {\tt thangphammath@xidian.edu.cn}} \and Boqing Xue \thanks{Institute of Mathematical Sciences, ShanghaiTech University. ~Email: {\tt xuebq@shanghaitech.edu.cn}}}
\date{}
\maketitle
\begin{abstract}
In this paper, we study the distance problem in the setting of finite p-adic rings. In odd dimensions, our results are essentially sharp. In even dimensions, we clarify the conjecture and provide examples to support it. Surprisingly, compared to the finite field case, in this setting, we are able to provide a large family of sets such that the distance conjecture holds. By developing new restriction type estimates associated to circles and orbits, with a group theoretic argument, we will prove the $4/3$-parallel result in the two dimensions. This answers a question raised by Alex Iosevich. In a more general scenario, the existence/distribution of geometric/graph configurations will be also considered in this paper. The main results present improvements and extensions of the recent work due to Ben Lichtin (2019, 2023). In comparison with Lichtin's method, our approach is much simpler and flexible, which is also one of the novelties in this paper.

%We also study

%This paper is on  the Erd\H{o}s-Falconer distance problem in finite p-adic rings. The first purpose is to study the distance problem in a general setting when the distance function is replaced by a diagonal polynomial. The second purpose is to study the existence/distribution of geometric/graph configurations. Our results are extensions/improvements of earlier results due to  Ben Lichtin (2019, 2023), and others over finite fields. Compared to the finite field setting, the main challenge in finite p-adic rings is to find an effective approach that leads to a uniform density. In a special case of the usual distance function, compared to Ben Lichtin's method, one of the novelties of this paper is to present a much simpler and flexible approach. Our results are essentially sharp in odd dimensions, and a conjecture in even dimensions will be proposed and discussed carefully.
\end{abstract}
\tableofcontents
\section{Introduction}
Let $p$ be a prime, $r$ be a positive integer, and $\mathbb{Z}/p^r\mathbb{Z}$ be a finite $p$-adic ring. For a postive integer $n$, the {\color{black}distance function} between two points $\bfx=(x_1, \ldots, x_n)$ and $\bfy=(y_1, \ldots, y_n)$ in $(\mathbb{Z}/p^r\mathbb{Z})^n$, {\color{black}denoted by $||\bfx-\bfy||:=F(\bfx,\bfy)$, is given by 
\[
F(\bfx,\bfy)=(x_1-y_1)^2+\cdots+(x_n-y_n)^2\quad (\mod p^r).\]}
Given subsets $E_1,E_2\subset (\mathbb{Z}/p^r\mathbb{Z})^n$, the distance set determined by points in $E_1\times E_2$ is denoted by $\Delta_{n,r}(E_1,E_2)$, i.e.
%\[
%\Delta_{n,r}(E_1,E_2) = \{\|\bfx-%\bfy\|:\, \bfx\in E_1,\,\bfy\in E_2\}.
%\]
\[\color{black}
\Delta_{n,r}(E_1,E_2) = \{F(\bfx,\bfy):\, \bfx\in E_1,\,\bfy\in E_2\}.
\]
For simplicity, we write $\Delta_{n,r}(E)=\Delta_{n,r}(E,E)$. We also denote the density by $\delta_{E_1,E_2}:=\frac{\sqrt{|E_1||E_2|}}{p^{rn}}$ and $\delta_E:=\frac{|E|}{p^{rn}}$.

In the setting of finite $p$-adic rings, the Erd\H{o}s--Falconer distance problem is stated as follows.

\begin{question}\label{mainquestion}
What is the smallest density threshold $\delta\in (0, 1)$ independent of $r$ such that $|\Delta_{n,r}(E)|\gg p^r$ whenever $E$ is a subset of $(\mathbb{Z}/p^r\mathbb{Z})^n$ with $\delta_E\ge \delta$?
\end{question}
This problem was initially studied {\color{black}with the distance function} in the case $r=1$, i.e. over finite fields, due to Iosevich and Rudnev \cite{IR}. More precisely, they proved that if $|E|\geq C p^{\frac{n+1}{2}}$ for some sufficient large constant $C$, then the distance set $\Delta_{n,1}(E)$ covers the whole field. Hart, Iosevich, Koh, and Rudnev \cite{hart} indicated that the exponent $\frac{n+1}{2}$ is sharp in odd dimensions. In even dimensions, it is conjectured that the right exponent should be $n/2$. In two dimensions, Chapman, Erdogan, Hart, Iosevich, and Koh \cite{chapman} proved the exponent $4/3$ by using an extension theorem associated to circles in the plane. This result was recently improved to $5/4$ by Murphy, Petridis, the first listed author, Rudnev, and Stevens in \cite{Murphy} by using algebraic methods and results from incidence geometry.

When $r>1$, by extending the techniques from finite fields, Covert, Iosevich and Pakianathan \cite{covert} proved that $U_r\subset \Delta_{n,r}(E)$ whenever {\color{black}$\delta_E\gg r(r+1)p^{-\frac{n-1}{2}}$}, where $U_r:= (\mathbb{Z}/p^r\mathbb{Z})^\ast$ is the set of units. This result is only non-trivial when $r$ is bounded, and does not offer a uniform density independent of $r$, to Question \ref{mainquestion}. By using a different and sophisticated approach, namely, a combination of $p$-adic analysis and estimates for a class of exponential sums mod $p$, Lichtin \cite{BL} proved that $U_r\subset \Delta_{n,r}(E)$ if $\delta_E\gg p^{-\frac{n-1}{2}}$. As mentioned in his paper, the main advantage of his approach is that the argument detects nontrivial cancellations within certain exponential sums mod $p^r$ which were not used in the work of Covert et al. \cite{covert}.

In this paper, the first purpose is to study the distance problem in a general setting when the distance function is replaced by {\color{black} a more general polynomial}. In the case of the usual distance function, one of the novelties of this paper is to present a simpler and flexible approach than that of Lichtin. Our results are essentially sharp in odd dimensions. The second purpose of this paper is to clarify the conjecture in even dimensions and provide examples to support it. In two dimensions, by developing new restriction type estimates associated to circles and orbits, with a group theoretic argument, we will prove the $4/3$-parallel result. Surprisingly, in comparison with the finite field case, in this setting, we are able to provide a large family of sets such that the distance conjecture holds. The third purpose is to study the existence/distribution of geometric/graph configurations. To be precise, we are interested in the following configuration-type questions.

\begin{question}Let $H$ be a given graph, $E$ be any subset of $(\mathbb{Z}/p^r\mathbb{Z})^n$ and $j$ be any element of $U_r$. What is the smallest density threshold $\delta\in (0, 1)$ independent of $r$ such that $E$ contains a copy of $H$ at distance $j$ whenever $\delta_E\gg \delta$?
\end{question}

%To explain this question, we provide some examples. If $H$ is a graph with two vertices $\{v_1, v_2\}$ and one edge $\{v_1v_2\}$, then the statement $E$ contains a copy of $H$ at distance $j$ is equivalent to say that there are two points $x, y$ in $E$ such that the distance between $x$ and $y$ is $j$. If $H$ is $k$-chain, i.e. a graph with $k+1$ vertices $\{v_1, \ldots, v_{k+1}\}$ and $k$ edges $\{v_1v_2, v_2v_3, \ldots, v_{k}v_{k+1}\}$, then $E$ contains a copy of $H$ at distance $j$ if there are points $x_1, \ldots, x_{k+1}$ in $E$ such that $||x_i-x_{i+1}||=j$ for all $1\le i\le k$.

%In this paper, the following graphs will be considered: trees, cycles, and chains. In the last case, our result improves the main result in a recent paper due to Lichtin \cite{BL3}.

\begin{question}
Let $\mathcal{C}$ be a given geometric configuration and $E$ be any set in $(\mathbb{Z}/p^r\mathbb{Z})^n$. What is the smallest density threshold $\delta\in (0, 1)$ not depending on $r$ such that $E$ contains a copy of $\mathcal{C}$ whenever $\delta_E\gg \delta$?
\end{question}
This paper addresses the graphs of being cycles, chains, and trees, and a geometric configuration of rectangles. In contrast to initial methods/results over finite fields, the main challenge one has to deal with in this setting is to find an effective approach to obtain a uniform density which does not depend on $r$. Lichtin's method in \cite{BL} is based on a combination of $p$-adic analysis and estimates for a class of exponential sums mod $p$. To find cancellations in exponential sums, he investigated paths of points (descending family of neighbourhoods), and his arguments are heavily relied on computing Hessian matrix of all levels. The approach we introduce in this paper is simpler, which involves polynomial congruences and Fourier analysis in finite rings. Our arguments only rely on the {\color{black}first order Jacobian matrix}, and avoid the {\color{black}second order Hessian matrix}.

\subsection{Results on the generalized distance sets}

For a polynomial $F(\bfx)\in \bZ[\bfx]$ in $n$ variables, we denote $F_i(\bfx)=\frac{\partial F}{\partial x_i}(\bfx)$ for all $1\leq i\leq n$, and denote $(\nabla F)(\bfx)=(F_1(\bfx),F_2(\bfx),\ldots,F_n(\bfx))$. Our first result reads as follows.

\begin{theorem} \label{main}
Let $F(\bfx)$ be a given polynomial in $\bZ[\bfx]$ in $n\geq 2$ variables. Suppose that the following conditions hold for some positive constants $c_1,c_2,c_3$ with $c_1<1$, for some prime $p$, and for some $j\in \bZ$ with $(j,p)=1$:

(\romannumeral1) $(\nabla F)(\bfx)\not\equiv \bfo \,(\mod p)$ when $F(\bfx)\equiv j\, (\mod p)$;

(\romannumeral2) $\big|\#\{\bfx\,(\mod p):\, F(\bfx)\equiv j\,(\mod p)\}-p^{n-1}\big| \leq c_1 p^{n-1}$;

(\romannumeral3) When $\bfm \not\equiv \bfo\, (\mod p)$,
\[
\Big|\sum\limits_{\bfx\,(\mod p)\atop F(\bfx)\equiv j \,(\mod p)} \,\,e_p\left(-\bfm \cdot \bfx\right)\Big|\leq c_2 p^{\frac{n-1}{2}};
\]

(\romannumeral4) When $\bfm \not\equiv \bfo\, (\mod p)$,
\begin{align*}
\#\big\lbrace\bfx(\mod p):\, &F(\bfx)\equiv j\, (\mod p),\, \exists~ 1\leq t\leq n,\,\, s.t. ~F_t(\bfx)\not\equiv 0\,(\mod p), \\
&m_t\not\equiv 0\,(\mod p), m_i F_t(\bfx)\equiv m_t F_i(\bfx)\, (\mod p)\, (1\leq i\leq n)\big\rbrace\leq c_3 p^{\frac{n-1}{2}}.
\end{align*}
Let $r$ be any positive integer and $E_1$, $E_2$ be subsets of $(\bZ/p^r\bZ)^n$. Assume that
\[
\delta_{E_1,E_2}\,>\, {\color{black}C p^{-\frac{n-1}{2}}},
\]where $C=(1-c_1)^{-1}\max\{c_2,c_3\}$. Then
\[
\#\{(\bfx,\bfy)\in E_1\times E_2:\, F(\bfx-\bfy)\equiv j\,(\mod p^r)\}>0.
\]
\end{theorem}
\begin{remark}
    In applications, the constants $c_1,c_2,c_3$ depend only on the polynomial $F$ (degree, coefficients, number of variables), and the prime $p$ needs to be sufficiently large compared to the degree and the coefficients of $F$. So the result is uniform in all natural numbers $r\geq 1$ and all integers $j$ with $(j,p)=1$. Here and throughout, by depending on $F$ we mean in terms of the degree, coefficients, and number of variables.
\end{remark}
\begin{corollary} \label{cor_diag_hom}
Let $F(\bfx)=\sum\nolimits_{i=1}^n a_ix_i^k$ be a polynomial in $\bZ[\bfx]$ with $n\geq 2$, $k\geq 2$, and $a_i\neq 0$ for all $1\leq i\leq n$. Then, for any sufficiently large prime $p$ and the density threshold at least $Cp^{-\frac{n-1}{2}}$, the conclusion of Theorem \ref{main} holds for all integers $r,j$ with $r\geq 1$ and $(j,p)=1$, and with the constant $C$ depending only on $F$.
\end{corollary}

It also worth noting that when $k=3$ and $r=1$, Corollary \ref{cor_diag_hom} recovers a result established by Iosevich and Koh in \cite{IK}. Moreover, Example \ref{ex1} in the last section will show that this distance result is sharp in odd dimensions.

We mention that condition (\romannumeral3) in Theorem \ref{main} is equivalent to
\[
\Big|\sum\limits_{\bfx\,(\mod p)}\sum\limits_{s\not\equiv 0\, (\mod p)}  \,\,e_p\left(-\bfm \cdot \bfx+s F(\bfx)-sj\right)\Big|\ll p^{\frac{n+1}{2}}
\]
for $\bfm \not\equiv \bfo\, (\mod p)$. However, the above bound is not applicable to polynomials of the form $F(\bfx)=\sum\nolimits_{j=1}^n a_jx_j^{k_j}$ with distinct exponents. In \cite{KS}, Koh and Shen worked with a weaker exponential sum
\[
\Big|\sum\limits_{\bfx\,(\mod p)} \,\,e_p\left(-\bfm \cdot \bfx+s F(\bfx)\right)\Big|\ll p^{\frac{n}{2}}
\]
for $\bfm \not\equiv \bfo\, (\mod p)$ and $s\not\equiv 0\,(\mod p)$. And they proved that $|\Delta_{n,1}(E_1,E_2)|\gg p$ whenever $E_1,E_2$ are subsets of $\bF_p^n$ satisfying
\begin{equation} \label{eq_main2_r=1}
\delta_{E_1,E_2} \gg p^{-(n-1)/2}.
\end{equation}
We observe that the approach in \cite{KS} will not lead to 
result which is uniform in $r$. With the method developed in this paper, we are able to provide a generalization of the above result in the setting of finite $p$-adic rings.

\begin{theorem} \label{main2}
Let $F(\bfx)=\sum\nolimits_{i=1}^n a_ix_i^{k_i}$ be a polynomial in $\bZ[\bfx]$ with $n\geq 2$, $k_i\geq 2$ and $a_i\neq 0$ for all $1\leq i\leq n$. Denote $k_\ast = \min\nolimits_{1\leq i\leq n}k_i$. Then, for any sufficiently large prime $p$, for any natural number $r\geq 1$, and any subsets $E_1,E_2\subseteq \bZ/p^r\bZ$, we have
\[
|\Delta_{n,r} (E_1 ,E_2 )| \gg \min\left\{p^r,\,\frac{|E_1||E_2|}{p^{r(2n-2/k_\ast)-n+1}}\right\}.
\]
In particular, one has $|\Delta_{n,r} (E_1 ,E_2)|\gg p^r$ if
\begin{equation} \label{eq_main2_rgeq2}
\delta_{E_1,E_2} \gg p^{r\big(\frac{1}{2}-\frac{1}{k_\ast}\big)-\frac{n-1}{2}}.
\end{equation}
Here, the implied constants depend only on $F$.
\end{theorem}
If $k_\ast =2$, then we can see that the lower bound of the density is independent of $r$. However, if $k_\ast \geq 3$, then the formula \eqref{eq_main2_rgeq2} is meaningful only when $r\leq \frac{k_\ast(n-1)}{k_\ast-2}$.

\subsection{The conjecture in even dimensions}
In even dimensions, one might ask about the conjectured densities to guarantee that $U_r\subset \Delta_{n, r}(E)$ or $|\Delta_{n, r}(E)|\gg p^r$. In this paper, we propose the following.
\begin{conjecture}\label{con-even}
    Let $F(\bfx)=\sum\nolimits_{i=1}^n a_ix_i^k$ be a polynomial in $\bZ[\bfx]$ with $n\geq 2$ even, $k\geq 2$, and $a_i\neq 0$ for all $1\leq i\leq n$. Then, for any sufficiently large prime $p$, if {\color{black}$r$ is a positive integer and $E_1,E_2$ are subsets of $(\bZ/p^r\bZ)^n$ with} the density $\delta_{E_1, E_2}\gg p^{-\frac{n}{2}}$, we have $|\Delta_{n, r}(E_1, E_2)|\gg p^r$.
\end{conjecture}
Example \ref{ex2} in the last section will support this conjecture for the case $k=2$.

In the setting of finite fields, i.e. $r=1$, as mentioned above, Chapman, Erdogan, Hart, Iosevich, and Koh \cite{chapman} proved the exponent $4/3$ by using an extension theorem associated to circles in the plane. Another proof by using geometric properties of rigid-motions in the plane can be found in \cite{HLR}. Alex Iosevich asked in several conferences/workshops if the $4/3$-parallel result exists
in the setting of finite p-adic rings. In this paper, we give an affirmative answer to his question.

\begin{theorem}\label{un-conditional}
{\color{black}Let $F$ be the distance function, $p$ be an odd prime, and $r\ge 1$ be an integer.} Let $E\subset (\mathbb{Z}/p^r\mathbb{Z})^2$ with $\delta_E \gg p^{-\frac{2}{3}}$. Then $|\Delta_{2, r}(E)|\gg p^r$. Here the implied constant is independent of $r$.
\end{theorem}
Compared to Theorem \ref{main}, we have a smaller threshold density, namely, of $p^{-\frac{2}{3}}$. Regarding the proof of Theorem \ref{un-conditional}, it is challenging if we just want to follow the methods in \cite{chapman} or in \cite{HLR}. On the one hand, in a recent exposition, Liao showed that the method in \cite{HLR} implies too many degenerate cases, even with $r=2$, which are very hard to deal with. On the other hand, a direct computation shows that the same happens with the approach developed in \cite{chapman}, namely, the main difficulty arises when finding the explicit form of the sum $\sum_{j\in \mathbb{Z}/p^r\mathbb{Z}}\widehat{1_{C_{r, j}}}(\bfm)\widehat{1_{C_{r, j}}}(\bfm')$ for $\bfm$ and $\bfm'$ in $(\mathbb{Z}/p^r\mathbb{Z})^n$, where $C_{r, j}$ is the circle centered at the origin of radius $j$.

In this paper, to prove Theorem \ref{un-conditional}, we first use a group theoretic argument, which we learned from \cite[Appendix]{MMM}, to reduce the theorem to an extension type question for circles, then the rest is devoted to study such type estimate in the p-adic setting. The extension theorems, we obtain in Section 4, are of independent interest and are expected to have many other applications.

Unlike the finite field case, in finite p-adic rings, we are able to provide a large family of sets $E$ such that Conjecture \ref{con-even} is true, namely, the family of sets $E$ satisfying the property that the density of the fibre of the natural projection over each element in $(\mathbb{Z}/p\mathbb{Z})^2$ is not too large in $E$. The precise statement reads as follows.
\begin{theorem}\label{conditional}
{\color{black}Let $F$ be the distance function, $p$ be an odd prime, and $r\ge 2$ be an integer.} Let $E\subset (\mathbb{Z}/p^r\mathbb{Z})^2$ with $\delta_E\gg p^{-1}$. Assume that
\[\color{black}
\#\{(\bfx_1,\bfx_2)\in E^2:\, \bfx_1\equiv \bfx_2\,(\mod p)\}\ll p^{2r-\frac{7}{3}}|E|.
\]
Then $|\Delta_{2, r}(E)|\gg p^r$.
\end{theorem}
\begin{corollary}\label{conditiona2}
{\color{black}Let $F$ be the distance function, $p$ be an odd prime, and $r\ge 2$ be an integer.} Let $E\subset (\mathbb{Z}/p^r\mathbb{Z})^2$ with $\delta_E\gg p^{-1}$. Assume that
\[\#\{\bfx'\in E\colon \bfx'\equiv\bfx\,(\mod p)\}\ll p^{2r-\frac{7}{3}}\]
for each {\color{black}$\bfx\in E$}. Then $|\Delta_{2, r}(E)|\gg p^r$.
\end{corollary}
The implied constants in {\color{black}Theorem \ref{conditional} and Corollary \ref{conditiona2}} are independent of $r$.

In addition to the restriction-theoretic approach developed above, the paper also introduces
 an $\mathbb F_p$--to--$\mathbb Z/p^r\mathbb Z$ transfer
principle for the standard quadratic distance in arbitrary dimensions. This principle allows us to lift results on the number of isosceles triangles from $\mathbb{F}_p^n$ to $(\mathbb{Z}/p^r\mathbb{Z})^n$. As a consequence, we obtain the following results by using the currently available estimates in dimensions two and four from the literature. \footnote{This paper has been under review for almost two years. To keep
the manuscript up to date and to provide a more complete account of the problem, we add the following two theorems, which incorporate recent developments obtained after the original submission. Consequently, the present arXiv version may differ from the version eventually published in the journal.}

\begin{theorem}\label{un-conditional-26}
{\color{black}Let $F$ be the distance function, $p$ be an odd prime, and $r\ge 1$ be an integer.} Let $E\subset (\mathbb{Z}/p^r\mathbb{Z})^2$ with $\delta_E \gg p^{-\frac{3}{4}}$. Then $|\Delta_{2, r}(E)|\gg p^r$. Here, the implied constant is independent of $r$.
\end{theorem}
\begin{theorem}\label{un-conditional-27}
{\color{black}Let $F$ be the distance function, $p$ be an odd prime, and $r\ge 1$ be an integer.} Let $E\subset (\mathbb{Z}/p^r\mathbb{Z})^4$ with $\delta_E \gg p^{-\frac{56}{37}+\varepsilon}$ for any $\varepsilon>0$. Then $|\Delta_{4, r}(E)|\gg p^r$. Here, the implied constant is independent of $r$.
\end{theorem}
\subsection{Results on geometric/graph configurations}
Let $F(\bfx)$ be a given polynomial in $\bZ[\bfx]$ in $n\geq 2$ variables. We call $F(\bfx)$ \textit{good} if it satisfies the conditions of Theorem \ref{main}. In the following, we prove several extensions of the distance result in the setting of geometric/graph configurations in which the polynomials $F(\bfx)$ are assumed to be good.
%We call the function $F$ satisfying the conditions of Theorem \ref{main} \textit{good}. In the following, we present a number of applications.

The first result is on the existence of rectangles of given side-length $j\in U_r$, which is an extension of an earlier result due to Lyall and Magyar in \cite{LM}.

%{\color{gray}Let $F$ be a good polynomial.
%For given $0<\delta'<\delta<1$ and an integer $j$ with $(j, p)=1$, there exists $\epsilon>0$ such that the following holds. If $E\subset (\mathbb{Z}/p^r\mathbb{Z})^{2n}$, with $p^{-\frac{2n-1}{2}}\le \epsilon$, satisfies $\delta_E\ge \delta$, then $E$ contains at least $\delta'p^{4rn-2}$ tuples
%\[(\bfu_1, \bfv_1),~(\bfu_1, \bfv_2), ~(\bfu_2, \bfv_1), ~(\bfu_2, \bfv_2)\]
%that form rectangles of side-length $j$, i.e. $F(\bfu_1-\bfu_2)=F(\bfv_1-%\bfv_2)\equiv j\,(\mod p^r)$.}

%Let $F$ be a good polynomial and $0<\delta'<\delta<1$ be given numbers. There exists $\epsilon>0$ such that the following holds. For any odd prime $p$, any positive integer $r$， and any integer $j$ with $(j,p)=1$, if $E\subset (\mathbb{Z}/p^r\mathbb{Z})^{2n}$, with $p^{-\frac{2n-1}{2}}\le \epsilon$, satisfies $\delta_E\ge \delta$, then $E$ contains at least $\delta' p^{4rn-2}$ tuples
%\[(\bfu_1, \bfv_1),~(\bfu_1, \bfv_2), ~(\bfu_2, \bfv_1), ~(\bfu_2, \bfv_2)\]
%that form rectangles of side-length $j$, i.e. $F(\bfu_1-\bfu_2)\equiv F(\bfv_1-\bfv_2)\equiv j\,(\mod p^r)$.

\begin{theorem} \label{rectangle}
Let $F$ be a good polynomial and $0<\delta'<\delta<1$ be given numbers. There exists $\epsilon>0$ such that the following holds. For any odd prime $p$, any positive integer $r$, and any integer $j$ with $(j,p)=1$, if $E\subset (\mathbb{Z}/p^r\mathbb{Z})^{2n}$, with $p^{-\frac{2n-1}{2}}\le \epsilon$, satisfies $\delta_E\ge \delta$, then $E$ contains at least $\delta' p^{4rn-2}$ tuples
\[(\bfu_1, \bfv_1),~(\bfu_1, \bfv_2), ~(\bfu_2, \bfv_1), ~(\bfu_2, \bfv_2)\]
that form rectangles of side-length $j$, i.e. $F(\bfu_1-\bfu_2)\equiv F(\bfv_1-\bfv_2)\equiv j\,(\mod p^r)$.
\end{theorem}

%Given an integer $k\le n$, a $k$-simplex is a tuple of $k+1$ points in $(\mathbb{Z}/p^r\mathbb{Z})^n$. For each $k$-simplex, say, $(u_1, u_2, \ldots, u_{k+1})$, we denote the vector of F-distances determined by pairs of points by
%\[\left(F(u_1-u_2), F(u_1-u_3), \ldots, F(u_{k}-u_{k+1}) \right)\in (\mathbb{Z}/p^r\mathbb{Z})^{\binom{k+1}{2}}.\]

%Two $k$-simplex $(u_1, \ldots, u_{k+1})$ and $(v_1, \ldots, v_{k+1})$ are called in the same equivalence class if there exists a bijective mapping $\psi$
%\[\psi\colon \{u_1, \ldots, u_{k+1}\}\to \{v_1, \ldots, v_{k+1}\}\]
%such that the vectors of $F$-distances of $(u_1, \ldots, u_{k+1})$ and $(\psi(u_1), \ldots, \psi(u_{k+1}))$ are the same.

%\begin{theorem}
%    Given $E\subset (\mathbb{Z}/p^r\mathbb{Z})^{n}$. Assume that $|E|\gg p^{(n-1)(r-\frac{1}{2})+\frac{rk}{2}}$
%\end{theorem}

If we only want to count the number of quadruples $(\bfu_1, \bfu_2, \bfu_3, \bfu_4)$ in $E$ that form a cycle, i.e., $F(\bfu_1-\bfu_2)=F(\bfu_2-\bfu_3)=F(\bfu_3-\bfu_4)=F(\bfu_4-\bfu_1)\equiv j\,(\mod p^r)$, then a weaker condition is sufficient.

\begin{theorem}\label{cycle}
%{\color{gray}Let $F$ be a good polynomial, and $j$ be an integer with $(j, p)=1$. There is some $C>0$ such that, if $E$ is a subset of $(\mathbb{Z}/p^r\mathbb{Z})^n$ with $\delta_E\geq Cp^{-\frac{n-1}{2}}$, then $E$ contains cycles of length $4$ with distinct vertices and of side-length $j$.}
{\color{black}Let $F$ be a good polynomial and $p$ be an odd prime. There is some $C>0$ such that, if $r$ is a positive integer, $j$ is an integer with $(j, p)=1$, and $E$ is a subset of $(\mathbb{Z}/p^r\mathbb{Z})^n$ with $\delta_E\geq Cp^{-\frac{n-1}{2}}$, then $E$ contains cycles of length $4$ with distinct vertices and of side-length $j$.}
\end{theorem}

The next result is on the existence of $k$-chains, i.e. a graph of $k+1$ vertices $\bfu_1, \ldots, \bfu_{k+1}$ such that $F(\bfu_i-\bfu_{i+1})\equiv j\,(\mod p^r)$ for all $1\le i\le k-1$.
\begin{theorem}\label{chain}
%{\color{gray}Let $F$ be a good polynomial, $k\ge 1$ be an integer, and $j$ be an integer with $(j, p)=1$. There is some $C>0$ such that, if $E$ is a subset of $(\mathbb{Z}/p^r\mathbb{Z})^n$ with $\delta_E\geq Cp^{-\frac{n-1}{2}}$, then $E$ contains $k$-chains of side-length $j$.}
{\color{black}Let $F$ be a good polynomial, $p$ be an odd prime, and $k\ge 1$ be an integer. There is some $C>0$ such that, if $r$ is a positive integer, $j$ is an integer with $(j, p)=1$, and $E$ is a subset of $(\mathbb{Z}/p^r\mathbb{Z})^n$ with $\delta_E\geq Cp^{-\frac{n-1}{2}}$, then $E$ contains $k$-chains of side-length $j$.}
\end{theorem}

We note that this theorem improves the density of $p^{-\frac{n-1}{3}}$ due to Lichtin in a recent paper \cite{BL2}.

Our last result is on the distribution of pinned trees. We first need to introduce some notations.

Let $T$ be an arbitrary tree with $k+1$ vertices and $k$ edges. Assume that $V(T)=\{\bfv_1, \ldots, \bfv_{k+1}\}\subset E$, then the edge set of $T$ can be ordered as follows:
\[\mathcal{E}(T)=\{ (\bfv_{i_{1}}, \bfv_{i_{2}}), (\bfv_{i_3}, \bfv_{i_4}), \ldots, (\bfv_{i_{2k-1}}, \bfv_{i_{2k}})\},\]
where $i_1\le i_3\le \cdots\le i_{2k-1}$, and $i_{2s}< i_{2t}$ if both $s<t$ and $i_{2s-1}=i_{2t-1}$.

For such a tree $T$, the vector
\[\left(|\bfv_{i_1}-\bfv_{i_2}|, |\bfv_{i_3}-\bfv_{i_4}|, \ldots, |\bfv_{i_{2k-1}}-\bfv_{i_{2k}}|\right) \in (\mathbb{Z}/p^r\mathbb{Z})^k\]
is called the edge-length vector of $T$. Two trees $T$ and $T'$ are called distinct if the corresponding edge-length vectors are not the same. We also recall that two trees $T$ and $T'$ are isotropic if there exists a bijective map $\varphi$ from $V(T)$ to $V(T')$ such that the edges are preserved. For $\bfx\in V(T')$ and $\bfv\in V(T)$, we say $(T', \bfx)$ is isomorphic to $(T, \bfv)$ if $T$ is isomorphic to $T'$ under $\varphi$ and $\varphi(\bfx)=\bfv$.

\begin{theorem}\label{tree}
{\color{black}Let $F$ be a good polynomial, $p$ be an odd prime, and $r$ be a positive integer.} Let $E\subset (\mathbb{Z}/p^r\mathbb{Z})^n$, $k\ge 1$ be an integer, and $(T, \bfv)$ be a given tree with $k+1$ vertices and $k$ edges. If $\delta_E\gg_k p^{-\frac{n-1}{2}}$, then there exists $\bfx\in E$ such that the number of distinct trees $(T', \bfx)$ with vertices in $E$ and isotropic to $(T, \bfv)$ is {\color{black}$\gg p^{rk}$}.
\end{theorem}

%In the rest of the paper, in some places, we may denote $q=p^r$ for simplicity if it causes no harm to the argument.

%{\color{black}And we may use both the notation $x\in \bZ/p^r\bZ$ and $x \, (\mod p^r)$.}

\section{Notations and Lemmas}
In this paper, the letter $p$ always denote a given prime {\color{black}greater than $2$}. The letter $n$ denotes the dimension of the space under consideration, and $r$ is always a positive integer. {\color{black}The letters $m,k,l,i,j$ always denote integers. The imaginary unit is denoted by $\i=\sqrt{-1}$.} The cardinality of a finite set $S$ is denoted by either $|S|$ or $\#S$. By $X\ll Y$, it means there exists some positive constant $C$ such that $X\le CY$. If $C=C(k)$, for some parameter $k$, then we write $X\ll_k Y$. 

%$f=o_F(g)$ means that $\frac{f}{g}\rightarrow 0$ in certain convergence, with convergence rate depending on $F$. 

When we write $\bfz\, (\mod p^r)$ for an $n$-dimensional vector $\bfz$, it always means that $\bfz$ is considered as an element in $(\bZ/p^r \bZ)^n$. The expression $\ord_p(z)=u$ means that $z\equiv 0\,(\mod p^u)$ and $z\not\equiv 0\,(\mod p^{u+1})$. We also denote $\ord_p(0) = r$ for $0\in \bZ/p^r\bZ$. When $\bfz=(z_1,z_2,\ldots,z_n)$, we write $v_\bfz=\min\nolimits_{1\leq i\leq n}\{\ord_p(z_i)\}$. Then the vector $\bfz$ can be expressed as $\bfz=p^{v_\bfz}\tilde{\bfz}$, where $\tilde{\bfz}$ is a vector in $(\bZ/p^{r-v_\bfz}\bZ)^n$ such that $v_{\tilde{\bfz}}=0$.

For any $j\in \bZ/p^r\bZ$, we use $C_{n,r,j}$ to denote the sphere in $(\bZ/p^r\bZ)^n$ centered at origin with radius $j$, i.e.,
\[
C_{n,r,j} = \{\bfz\in (\bZ/p^r\bZ)^n:\, \|\bfz\|=j\}.
\]
Similarly, for a given polynomial $F(x)$ in $n$ variables with coefficients in $\bZ$, we define
\[
S_{n,r,j} = \{\bfz\in (\bZ/p^r\bZ)^n:\, F(\bfz)=j\}.
\]
When it makes no confusion, we abbreviate $C_{n,r,j}$ (or $S_{n,r,j}$) as $C_{r,j}$ (or $S_{r,j}$, respectively).

The additive character modulo $p^r$ is denoted by $e_{p^r}(x) = e^{\frac{2\pi \i x}{p^r}}$, $(x\, \mod p^r)$.  For a function $f:\, (\bZ/p^r\bZ)^n\rightarrow \bC$, the Fourier transformation is defined by
\[
\widehat{f}(\bfm) = \frac{1}{p^{rn}} \sum\limits_{\bfx\,(\mod p^r)} f(\bfx)e_{p^r}(-\bfm\cdot \bfx),\quad \big(\bfm\in  (\bZ/p^r\bZ)^n\big),
\]
and the Fourier inverse is given by
\[
f(\bfx) = \sum\limits_{\bfm\,(\mod p^r)} \widehat{f}(\bfm)e_{p^r}(\bfm\cdot \bfx),\quad \big(\bfx\in  (\bZ/p^r\bZ)^n).
\]
The convolution of two functions $f_1$ and $f_2$ is defined by
\[
(f_1\ast f_2)(\bfx) = \frac{1}{p^{rn}}\sum\limits_{\bfy\,(\mod p^r)} f_1(\bfx-\bfy)f_2(\bfy).
\]
We have the property that $\widehat{f_1\ast f_2}=\widehat{f_1}\widehat{f_2}$. The Parseval's identity is
\[
\frac{1}{p^{rn}}\sum\limits_{\bfx\,(\mod p^r)}f_1(\bfx)\overline{f_2(\bfx)} = \sum\limits_{\bfm\,(\mod p^r)}\widehat{f_1}(\bfm)\overline{\widehat{f_2}(\bfm)}.
\]
%For $p\geq 1$, the $l^p$-norm of $f$ is given by
%\[
%\|f\|_p = \left(\bE_{x\in G} |f(x)|^p\right)^{1/p}.
%\]
%If $1/r=1/p+1/q$, then we have the H\"{o}lder inequality $
%\|fg\|_r \leq \|f\|_p\|g\|_q$.
%For a subset $X\subseteq G$, one has $\|1_X\|_1=\|1_X\|_2^2=|X|/|G|$.

The next lemma is one version of Hensel's lemma that will be regularly applied in this paper.

\begin{lemma}[Hensel's lemma] \label{lem_Hensel}
Let
\[
\mathbf{G}(\bfx) = \big(G_1(x_1,\ldots,x_n),\ldots,G_m(x_1,\ldots,x_n)\big)
\]
be a map from $\bZ^n$ to $\bZ^m$, with $G_i$ polynomials with integer coefficients. Let $l$ be a positive integer and $\bfy\in \bZ^n$. Suppose that $\mathbf{G}(\bfy)\equiv \bfo$ $(\mod p^l)$. Let $R$ be the rank of $J({\mathbf{G}})|_\bfy$ modulo $p$, where $J({\mathbf{G}})|_\bfy$ is the Jocobi matrix
\[
J({\mathbf{G}})|_\bfy=\begin{bmatrix}
\frac{\partial G_1}{\partial x_1}(\bfy) &\frac{\partial G_1}{\partial x_2}(\bfy) &\ldots &\frac{\partial G_1}{\partial x_n}(\bfy) \\
\frac{\partial G_2}{\partial x_1}(\bfy) &\frac{\partial G_2}{\partial x_2}(\bfy) &\ldots &\frac{\partial G_2}{\partial x_n}(\bfy) \\
\ldots &\ldots &\ldots &\ldots \\
\frac{\partial G_m}{\partial x_1}(\bfy) &\frac{\partial G_m}{\partial x_2}(\bfy) &\ldots &\frac{\partial G_m}{\partial x_n}(\bfy)
\end{bmatrix}.
\]
Then
%\begin{equation} \label{eq0}
%\#\big\{\bfz\,(\mod p):\, \mathbf{G}(\bfy+p^l\bfz)\equiv \bfo\,(\mod p^{l+1})\big\} \leq p^{n-R},
%\end{equation}
\begin{equation} \label{eq0}
\#\big\{\bfz\,(\mod p^k):\, \mathbf{G}(\bfy+p^l\bfz)\equiv \mathbf{0}\,(\mod p^{l+k})\big\} \leq p^{k(n-R)}
\end{equation}
for any integer $k\geq 1$. When $R=m$, the ``$\leq$'' can be replaced by ``$=$''.
\end{lemma}

\begin{proof}
Firstly, let us consider the case $k=1$. Let $\mathbf{G}(\bfy)= p^l \bfb$, where $\bfb=(b_1,b_2,\ldots,b_m)\in \bZ^m$. Note that
\[
G_i(\bfy+p^l \bfz) \equiv G_i(\bfy)+ p^{l}(\nabla G_i)(\bfy)\cdot \bfz \quad (\mod p^{l+1})
\]
for any $i=1,2,\ldots,m$. Then $G_i(\bfy+p^l\bfz)\equiv 0\, (\mod p^{l+1})$ $(i=1,2,\ldots,m)$ if and only if $b_i+(\nabla G_i)(\bfy)\cdot \bfz \equiv 0\, (\mod p)$ $(i=1,2,\ldots,m)$, if and only if $J(\mathbf{G})|_\bfy\bfz \equiv - \bfb\, (\mod p)$. The number of solutions to such a system of linear equations does not exceed $p^{n-R}$. Thus,
\[
\#\big\{\bfz\,(\mod p):\, \mathbf{G}(\bfy+p^l\bfz)\equiv \bfo\,(\mod p^{l+1})\big\} \leq p^{n-R}.
\]
Moreover, the system is consistent when $R=m$, and has exactly $p^{n-m}$ solutions.

Secondly, let us consider the case $k=2$. For any $\bfz_1 (\mod p)$ with $ \mathbf{G}(\bfy+p^l\bfz_1)\equiv \bfo\,(\mod p^{l+1})$, we write $\bfy_1=\bfy+p^l\bfz_1$. Noting that $J(\mathbf{G})|_{\bfy_1}\equiv J(\mathbf{G})|_{\bfy}\,(\mod p)$, the rank of $J(\mathbf{G})|_{\bfy_1}$ modulo $p$ is also $R$. Hence
\[
\#\big\{\bfz_2\,(\mod p):\, \mathbf{G}(\bfy_1+p^{l+1}\bfz_2)\equiv \bfo\,(\mod p^{l+2})\big\} \leq p^{n-R}.
\]
It follows that
\[
\#\big\{(\bfz_1, \bfz_2)\,(\mod p):\, \mathbf{G}\big(\bfy+p^l(\bfz_1+p\bfz_2)\big)\equiv \bfo\,(\mod p^{l+2})\big\} \leq p^{2(n-R)}.
\]

Finally, the conclusion follows by an induction on $k$.
\end{proof}

In the following, we collect and prove some results involving exponential sums or cardinarlity of varieties.

\begin{lemma} [Weil's theorem, Theorem 5.38, \cite{Lidl}] \label{lem_Weil}
Let $f(x)\in \bF_p[x]$ be a polynomial of degree $k\geq 1$ with $(k,p)=1$. Then
\[
\Big|\sum\limits_{x\in \bF_p} e_p(f(x))\Big| \leq (k-1)p^{1/2}.
\]
\end{lemma}

%{\color{red}Connecting lines and references}

\begin{lemma} [Theorem 2.3, \cite{KS}]  \label{lem_solutions}
Let $F(\bfx)=\sum\nolimits_{i=1}^n a_ix_i^{k_i}$ be a polynomial in $\bF_p[\bfx]$ such that $n\geq 2$, $k_i\geq 1$, and $a_i\neq 0$ for all $1\leq i\leq n$. Let $j\in \bF_p^\ast$. Then
\begin{equation*} %\label{eqo1}
\#\{\bfx\in \bF_p :\, F(\bfx)= j\} = (1+\textit{o}_F(1)) p^{n-1}
\end{equation*}
as $p\rightarrow \infty$. 
\end{lemma}

Here $X=o_F(Y)$, it means $X/Y\rightarrow 0$ with the convergence rate depending on $F$. 

%{\color{red}Connecting lines and references}

\begin{lemma} [Lemma 2.1, \cite{KS}] \label{lem_Fourier_diag_hom}
Let $F(\bfx)=\sum\nolimits_{i=1}^n a_ix_i^k$ be a polynomial in $\bF_p[\bfx]$ with $n\geq 1$, $k\geq 2$, and $a_i\neq 0$ for all $1\leq i\leq n$. Let $j\in \bF_p^\ast$. Then, for any $\bfm\in \bF_p^n \setminus\{\bfo\}$, one has
\[
\sum\limits_{\bfx\in \bF_p^n\atop F(\bfx)=j}e_p(-\bfm\cdot \bfx) \ll_{n,k} p^{\frac{n-1}{2}}
\]
when $p$ is sufficiently large.
\end{lemma}

%{\color{red}Connecting lines and references}

\begin{lemma} \label{lemma_fourier}
Let $p$ be a prime and $j$ be an integer with $(j,p)=1$. Assume that $F(\bfx)$ is a polynomial in $\bZ[\bfx]$ in $n\geq 1$ variables satisfying the conditions (\romannumeral1) and (\romannumeral3) in Theorem \ref{main}, and the following:

(\romannumeral4'):  When $\bfm \not\equiv \bfo\, (\mod p)$,
\begin{align*}
&\#\big\lbrace\bfy(\mod p):\, F(\bfy)\equiv j (\mod p),\, \exists ~1\leq t\leq n,\,\, s.t., F_t(\bfy)\not\equiv 0(\mod p), \\
&\qquad\qquad m_t\not\equiv 0(\mod p), m_i F_t(\bfy)\equiv m_t F_i(\bfy) \,(\mod p)\, \mbox{for all}~1\leq i\leq n\big\rbrace\leq {c_3} p^{\kappa},
\end{align*}
for some $c_3>0$ and $0\leq \kappa\leq \frac{n-1}{2}$.

Then, for any $r\geq 1$, we have
\[
|\widehat{1_{S_{r,j}}}(\bfm')|\leq
\begin{cases}
c_2\, p^{-r-\frac{n-1}{2}},\quad &\text{if }\bfm' \equiv \bfo\,(\mod p^{r-1}) \text{ but }\bfm' \not\equiv \bfo\,(\mod p^{r}),\\
c_3\, p^{-r-n+1+\kappa},  &\text{if }\bfm' \not\equiv \bfo\,(\mod p^{r-1}).
\end{cases}
\]
\end{lemma}

\begin{proof}
%For generality, let us proceed with condition (\romannumeral4'):  When $\bfm \not\equiv \bfo\, (\mod q)$,
%\begin{align*}
%&\#\left\{\bfy(\mod p):\, F(\bfy)\equiv j (\mod p),\, \exists 1\leq t\leq n,\,\, s.t., F_t(\bfy)\not\equiv 0(\mod p), \right.\\
%&\qquad\qquad \left.m_t\not\equiv 0(\mod p), m_i F_t(\bfy)\equiv m_t F_i(\bfy) (\mod p)\, (1\leq i\leq n)\right\}\leq {c_3} p^{\kappa}
%%\end{align*}
%for some $c_3>0$ and $0\leq \kappa\leq \frac{n-1}{2}$. Our aim is to show that
%\[
%%|\widehat{1_{S_{r,j}}}(\bfm)| \leq
%\begin{cases}
%\max\{c_2,c_3\} \cdot p^{-\frac{n+1}{2}},\quad &\text{if }r=1,\\
%\max\{c_2,c_3\}\cdot p^{-r-n+1+\kappa},\quad &\text{if }r\geq 2.
%\end{cases}
%\]

%When $r=1$, it follows by orthogonality that
%\begin{align*}
%\widehat{1_{S_{1,j}}}(\bfm) &=\frac{1}{p^n}\sum\limits_{\bfx (\mod p)\atop F(\bfx)\equiv j (\mod p^r)} e_p(-\bfm\cdot \bfx)\\
%&=\frac{1}{p^{n+1}}\sum\limits_{\bfx\,(\mod p)}{\sum\limits_{y\, (\mod p)}}  \,\,e_p\left(-\bfm \cdot \bfx+y F(\bfx)-yt\right).
%\end{align*}
%The summand over $y\equiv 0\, (\mod p)$ is
%\[
%\frac{1}{p^{n+1}} \sum\limits_{\bfx\, (\mod p)} e_p(-m\cdot x) =0,
%\]
%since $\bfm\not\equiv \bfo\, (\mod p)$. Now, by condition (\romannumeral3), one obtains that
%\begin{align*}
%\widehat{1_{S_{1,j}}}(\bfm) &=\frac{1}{p^{n+1}}\sum\limits_{\bfx\,(\mod p)}{\sum\limits_{y\, (\mod p)}}^\ast  \,\,e_p\left(-\bfm \cdot \bfx+y F(\bfx)-yj\right)\\
%&\leq  \frac{1}{p^{n+1}}\cdot c_2 p^{\frac{n+1}{2}} = c_2 p^{-\frac{n+1}{2}}.
%\end{align*}

%In the following, we deal with the situation that $r\geq 2$.

Recall that $S_{r, j}=\{\bfx\in (\mathbb{Z}/p^r\mathbb{Z})^n\colon F(\bfx)\equiv j\, (\mod p^r)\}$.
For simplicity, we denote $v_{\bfm'}=\nu$. Then $\bfm'=p^\nu\bfm$ with $\bfm\in (\bZ/p^{r-\nu}\bZ)^n$ and $v_\bfm=0$.

%For any $\bfl=(l_1,\ldots,l_n)$, let us denote $\nu(\bfl)=\min\nolimits_{1\leq i\leq n}\{\ord_p(l_i)\}$. Denote $\nu(\bfm')=\nu$. Then $\bfm'=p^\nu\bfm$ with $\nu(\bfm)=0$.
When $r=1$, the conclusion follows from the condition (\romannumeral3). In the following, we assume that $r\geq 2$.

For $1\leq \nu\leq r-1$, by a change of variables $\bfx=\bfy+p^{r-\nu}\bfz$, we have that
\begin{align*}
\widehat{1_{S_{r,j}}}(\bfm') = &\frac{1}{p^{rn}}\sum\limits_{\bfx (\mod p^{r})\atop F(\bfx)\equiv j (\mod p^{r})} e_{p^{r}}(-\bfm'\cdot \bfx)\\
& = \frac{1}{p^{rn}}\sum\limits_{\bfy (\mod p^{r-\nu})\atop F(\bfy)\equiv j (\mod p^{r-\nu})} \sum\limits_{\bfz  (\mod p^\nu)\atop F(\bfy+p^{r-\nu}\bfz)\equiv j (\mod p^{r})} e_{p^{r}}\big(-(p^\nu\bfm)\cdot (\bfy+p^{r-\nu}\bfz)\big)\\
& = \frac{1}{p^{rn}}\sum\limits_{\bfy (\mod p^{r-\nu})\atop F(\bfy)\equiv j (\mod p^{r-\nu})}  e_{p^{r-\nu}}\big(-\bfm\cdot \bfy)\sum\limits_{\bfz  (\mod p^\nu)\atop F(\bfy+p^{r-\nu}\bfz)\equiv j (\mod p^{r})} 1.
\end{align*}

Since $F(\bfy)\equiv j\, (\mod p)$, one has $(\nabla F)(\bfy)\not\equiv \bfo\, (\mod p)$ by the condition (\romannumeral1). By Hensel's lemma, we have
\[
\sum\limits_{\bfz  (\mod p^\nu)\atop F(\bfy+p^{r-\nu}\bfz)\equiv j (\mod p^{r})} 1 = p^{\nu(n-1)}
\]
for each given $\bfy (\mod p^{r-\nu})$. Therefore,
\begin{equation}\label{eq:2}
\widehat{1_{S_{r,j}}}(\bfm') = p^{-rn+\nu(n-1)} \sum\limits_{\bfy (\mod p^{r-\nu})\atop F(\bfy)\equiv j (\mod p^{r-\nu})}e_{p^{r-\nu}}\big(-\bfm\cdot \bfy)
\end{equation}
for $1\leq \nu\leq r-1$. Moreover, when $\nu=0$, the expression (\ref{eq:2}) trivially holds.
%Our aim is to show that
%\[
%\sum\limits_{\bfy (\mod p^{r-\nu})\atop F(\bfy)\equiv j (\mod p^{r-\nu})}e_{p^{r-\nu}}\big(-\bfm\cdot \bfy)\leq c_3 p^{(r-\nu-1)(n-1)+\kappa}
%\]

Let us substitute the parameters $\bfx$, $\gamma$ for $\bfy$, $r-\nu$, respectively. {\color{black}One has $1\leq \gamma\leq r$.} Then, it is sufficient to show that
\[
T_{\gamma}(\bfm):= \sum\limits_{\bfx (\mod p^{\gamma})\atop F(\bfx)\equiv j (\mod p^{\gamma})}e_{p^{\gamma}}\big(-\bfm\cdot \bfx)\leq
\begin{cases}
c_2\, p^{(\gamma-1/2)(n-1)},\quad &\text{if }{\color{black}\gamma=1},\\
c_3\, p^{(\gamma-1)(n-1)+\kappa},  &\text{if }{\color{black}2\leq \gamma\leq r}.
\end{cases}
\]
Here {\color{black}$v_\bfm=0$, i.e.,} $\bfm\not\equiv \bfo\,(\mod p)$.

When $\gamma=1$, the bound follows from the condition (\romannumeral3). In the following, we consider the situation that $\gamma\geq 2$. By a change of variables $\bfx = \bfy + p^{\gamma-1}\bfz$,
\begin{align*}
T_{\gamma}(\bfm) &= \sum\limits_{\bfy (\mod p^{\gamma-1})\atop F(\bfy)\equiv j (\mod p^{\gamma-1})}\sum\limits_{\bfz (\mod p)\atop F(\bfy+p^{\gamma-1}\bfz)\equiv j (\mod p^{\gamma})}e_{p^{\gamma}}\big(-\bfm\cdot (\bfy+p^{\gamma-1}\bfz))\\
&= \sum\limits_{\bfy (\mod p^{\gamma-1})\atop F(\bfy)\equiv j (\mod p^{\gamma-1})}e_{p^{\gamma}}\big(-\bfm\cdot \bfy)\sum\limits_{\bfz (\mod p)\atop F(\bfy+p^{\gamma-1}\bfz)\equiv j (\mod p^{\gamma})}e_p\big(-\bfm\cdot \bfz).
%& \leq\sum\limits_{\bfy (\mod p)\atop F(\bfy)\equiv j (\mod p)}\#\{\bfz(\mod p^{r-1}):\, F(\bfy+p\bfz)\equiv j(\mod p^r)\}.
\end{align*}

Let us denote
\[
T_{\gamma}(\bfm;\bfy) := \sum\limits_{\bfz (\mod p)\atop F(\bfy+p^{\gamma-1}\bfz)\equiv j (\mod p^{\gamma})}e_p\big(-\bfm\cdot \bfz).
\]
%Consider the case that p|mip| m_i for some 1≤i≤n1\leq i\leq n. Without loss of generality, we proceed with p|mnp|m_n.
The condition $F(\bfy+p^{\gamma-1}\bfz)\equiv j (\mod p^{\gamma})$ is equivalent to
\[
F(\bfy) + p^{\gamma-1}(\nabla F)(\bfy)\cdot \bfz \equiv j \,(\mod p^{\gamma}).
\]
Suppose that $F(\bfy)\equiv j+p^{\gamma-1} h_{\gamma}(\bfy)$ $(\mod p^{\gamma})$ for some $h_{\gamma}(\bfy)$ $(\mod p)$. {\color{black}Since $\gamma\geq 2$ and $F(\bfy)\equiv j\,(\mod p^{\gamma-1})$, one has $(\nabla F)(\bfy)\not\equiv \bfo \,(\mod p)$ by the condition (\romannumeral1).} Recall that $F_i(\bfy)=\frac{\partial F}{\partial x_i}(\bfy)$.
Without loss of generality, let us proceed with $F_1(\bfy)\not\equiv 0 (\mod p)$. Then
\[
z_1 \equiv -F_1^{-1}(\bfy)( h_{\gamma}(\bfy)+F_2(\bfy)z_2+\ldots+F_{n}(\bfy)z_{n}) \quad (\mod p).
\]
Here $F_1^{-1}(\bfy)$ is the inverse of $F_1(\bfy)$ $(\mod p)$. So
\[
T_{\gamma}(\bfm;\bfy)= e_p\big(m_1F_1^{-1}(\bfy) h_{\gamma}(\bfy)) \prod_{i=2}^{n} \sum\limits_{z_i (\mod p)} e_p\Big(\big(-m_i+m_1 F_1^{-1}(\bfy)F_i(\bfy)\big) z_i\Big).
\]
It follows that
\[
|T_{\gamma}(\bfm;\bfy)| =
\begin{cases}
p^{n-1},\quad &\text{if }m_iF_1(\bfy) \equiv m_1F_i(\bfy)\,\, (\mod p) \text{ for }1\leq i\leq n,\\
0,\quad &\text{otherwise}.
\end{cases}
\]
%The above expression is non-zero if and only if Fnmi≡FimnF_n m_i\equiv F_im_n (\mod p)(\mod p) for each 1\leq i\leq n-11\leq i\leq n-1. (This formula automatically holds when i=ni=n.) When it is non-zero, the above expression has absolute value p^{n-1}p^{n-1}.
(Note that, for $i=1$, the formula $m_iF_1(\bfy) \equiv m_1F_i(\bfy)\,\, (\mod p)$ trivially holds.)

Moreover, if $m_1\equiv 0$ $(\mod p)$, then the assumption $\bfm\not\equiv \bfo\,(\mod p)$ shows that $m_{i_0}\not\equiv 0 \,(\mod p)$ for some $2\leq i_0\leq n$. Then $m_{i_0} F_1(\bfy) \not\equiv m_1 F_{i_0}(\bfy)\, (\mod p)$, which leads to $T_{\gamma}(\bfm;\bfy)=0$. As a result, if $T(\bfm;\bfy)\neq 0$ and $F_t(\bfy)\not\equiv 0$ $(\mod p)$ for some $1\leq t\leq n$, then $m_t\not\equiv 0$ $(\mod p)$.

Now we conclude that
\[
|T_{\gamma}(\bfm)| \leq p^{n-1}|V_{\gamma-1}(\bfm)|,
\]
where
\begin{align*}
V_{\mu}(\bfm):=&\big\{\bfy(\mod p^\mu):\, F(\bfy)\equiv j (\mod p^\mu),\, \exists~ 1\leq t\leq n,\,\, s.t., F_t(\bfy)\not\equiv 0(\mod p), \\
&\qquad\qquad m_t\not\equiv 0(\mod p), m_i F_t(\bfy)\equiv m_t F_i(\bfy) (\mod p)\,\mbox{for all}~ 1\leq i\leq n\big\}.
\end{align*}
When $\gamma\geq 3$, we further have
\begin{align*}
|V_{\gamma-1}(\bfm)| \leq & \sum\limits_{\bfw\in V_1(\bfm)}\#\{\bfy(\mod p^{\gamma-1}):\, F(\bfy)\equiv j(\mod p^{\gamma-1}),\, \bfy\equiv \bfw\, (\mod p)\}.
\end{align*}
Note that $F(\bfw)\equiv j\,(\mod p)$ for any $\bfw\in V_1(\bfm)$. So $(\nabla F)(\bfw)\not\equiv 0\,(\mod p)$ by the condition (\romannumeral1). Applying Hensel's lemma again, we have
\begin{align*}
&\#\{\bfy(\mod p^{\gamma-1}):\, F(\bfy)\equiv j(\mod p^{\gamma-1}),\, \bfy\equiv \bfw\, (\mod p)\}\\
&= \#\{\bfz\,(\mod p^{\gamma-2}):\, F(\bfw+p\bfz)\equiv j\,(\mod p^{\gamma-1})\}= p^{(\gamma-2)(n-1)}
\end{align*}
for each given $\bfw$. {\color{black}So $|V_{r-1}(\bfm)|\leq p^{(\gamma-2)(n-1)}|V_1(\bfm)|$.} It follows that, for any $\gamma\geq 2$,
\[
|T_{\gamma}(\bfm)| \leq p^{(\gamma-1)(n-1)}|V_1(\bfm)|.
\]
Recalling that $\bfm\not\equiv \bfo\,(\mod p)$, one can apply the condition (\romannumeral4') to get $|V_1(\bfm)|\le c_3p^{\kappa}$. Thus, we conclude that $|T_{\gamma}(\bfm)| \leq c_3 p^{(\gamma-1)(n-1)+\kappa}$. The proof is completed.
\end{proof}

%{\color{red} Why do we have "indeed" in the following remark?}
{\color{black}
\begin{remark}
The definition of $V_\mu(\bfm)$ can be replaced by the collection of elements $\bfy\,(\mod p^\mu)$ such that $F(\bfy)\equiv j\, (\mod p)$, and, for all $1\leq t\leq n$ with $F_t(\bfy)\not\equiv 0\,(\mod p)$, it satisfies that $m_t\not\equiv 0\,(\mod p)$ and $m_i F_t(\bfy)\equiv m_t F_i(\bfy) (\mod p)$ for all $1\leq i\leq n$. So, the set in (\romannumeral4) in Theorem \ref{main} may be replaced by a smaller set.
\end{remark}
}

{\color{black}
\begin{remark}When $n=1$, the cardinality of the set in condition (\romannumeral4') becomes
\[
\#\{y\,(\mod p):\, F(y)\equiv j\,(\mod p)\},
\]
which is no larger than the degree of $F(y)$. So one may take $\kappa =0$ in this case.
\end{remark}
}

%{\color{red}Connecting lines and references}

\begin{lemma} \label{lem_Fourier2}
Let $\kappa$ be a number with $0\leq \kappa\leq n/2$ and $c_1,c_2>0$ be constants. Let $F(\bfx)\in \bZ[\bfx]$ be a polynomial in $n\geq 1$ variables of degree $\geq 2$ such that the upper bounds
\begin{equation} \label{eq_cond0}
\sum\limits_{\bfx\, (\mod p)}e_{p}\left(sF(\bfx)+\bfm\cdot \bfx\right)\leq c_1 p^{n/2}
\end{equation}
and
\begin{equation} \label{eq_cond}
\#\{\bfx\,(\mod p):\,s\,(\nabla F)(\bfx) +\bfm\equiv \bfo\, (\mod p) \} \leq c_2 p^{\kappa}
\end{equation}
both hold when $(\bfm,s)\not\equiv (\bfo,0) \,(\mod p)$. Then, for any $r\geq 1$, we have
\begin{align*}
    &\sum\limits_{\bfx\, (\mod p^{r})}e_{p^{r}}\left(s' F(\bfx)+\bfm'\cdot \bfx\right)\\
    &\le \begin{cases}
c_1 p^{(r-1/2)n},\quad &\text{if }(\bfm',\, s')\equiv (\bfo,0)\, (\mod p^{r-1}) \text{ but } (\bfm',\, s')\not\equiv (\bfo,0)\, (\mod p^{r}),\\
c_2 p^{(r-1)n+\kappa}, &\text{if } (\bfm',\, s')\not\equiv (\bfo,0)\, (\mod p^{r-1}).
\end{cases}
\end{align*}
\end{lemma}

\begin{proof}
%For $(\bfl,t)=(l_1,\ldots,l_n,t)$, denote $\nu(\bfl,t):=\min\{\ord_p(l_1),\ldots,\ord_p(l_n),\ord_p(t)\}$.
For simplicity, let us write $v_{(\bfm',s')}=\nu$. Since $(\bfm',\, s')\not\equiv (\bfo,0)\, (\mod p^{r})$, we have $0\leq \nu\leq r-1$. Write $\bfm'=p^\nu \bfm$ and $s'=p^\nu s$. Then $(\bfm,s)\not\equiv (\bfo, 0)\,(\mod p)$.

When $1\leq \nu\leq r-1$, we obtain by a change of variables $\bfx=\bfy+p^{r-\nu}\bfz$ that
\begin{align*}
G_{r}(\bfm',s'):=&\sum\limits_{\bfx\, (\mod p^{r})}e_{p^{r}}\left(s' F(\bfx)+\bfm'\cdot \bfx\right) \\
&= \sum\limits_{\bfy\, (\mod p^{r-\nu})} \sum\limits_{\bfz\,(\mod p^\nu)} e_{p^{r}}\left(p^\nu s F(\bfy)+p^\nu s(\nabla F)(\bfy)\cdot (p^{r-\nu}\bfz)+p^\nu \bfm\cdot (\bfy+p^{r-\nu}\bfz)\right)\\
& =\sum\limits_{\bfy\, (\mod p^{r-\nu})}  e_{p^{r-\nu}}\left(s F(\bfy)+\bfm\cdot \bfy\right)\sum\limits_{\bfz\,(\mod p^\nu)}1 = p^{\nu n} \cdot G_{r-\nu}(\bfm,s).
\end{align*}
When $\nu=0$, the above equality trivially holds. By writing $\gamma=r-\nu$ {\color{black}with $1\leq \gamma \leq r$}, it is sufficient to prove that
\[
G_{\gamma}(\bfm,s)\leq
\begin{cases}
c_1 p^{(\gamma-1/2)n},\quad &\text{if } {\color{black}\gamma=1},\\
c_2 p^{(\gamma-1)n+\kappa}, &\text{if } {\color{black}2\leq \gamma\leq r},
\end{cases}
\]
where $(\bfm,s)\not\equiv (\bfo,0) \,(\mod p)$.

When $\gamma=1$, the conclusion follows from \eqref{eq_cond0}. In the following, let us assume that $\gamma\geq 2$. By a change of variables $\bfx=\bfy+p^{\gamma-1}\bfz$, we obtain that
\begin{align*}
G_{\gamma}(\bfm,s)&=\sum\limits_{\bfx\, (\mod p^{\gamma})}e_{p^{\gamma}}\left(sF(\bfx)+\bfm\cdot \bfx\right) \\
%&= \sum\limits_{\bfy\, (\mod p^{r-1})} \sum\limits_{\bfz\,(\mod p)} e_{p^r}\left(sF(\bfy+p^{r-1}\bfz)+\bfm\cdot (\bfy+p^{r-1}\bfz)\right)\\
& =\sum\limits_{\bfy\, (\mod p^{\gamma-1})}  e_{p^{\gamma}}\left(sF(\bfy)+\bfm\cdot \bfy\right)\sum\limits_{\bfz\,(\mod p)} e_p\big((s\,(\nabla F)(\bfy) +\bfm)\cdot \bfz\big)\\&= p^n \sum\limits_{\bfy\, (\mod p^{\gamma-1})\atop s(\nabla F)(\bfy)+\bfm \equiv \bfo\, (\mod p)}  e_{p^{\gamma}}\left(sF(\bfy)+\bfm\cdot \bfy\right)\\
&\leq p^n \cdot\#\{\bfy\,(\mod p^{\gamma-1}):\, s\,(\nabla F)(\bfy) +\bfm\equiv \bfo\, (\mod p)\}.
\end{align*}

When $\gamma=2$, it follows from \eqref{eq_cond} that $G_{2}(\bfm,s)\leq c_2 p^{n+\kappa}$. When $\gamma\geq 3$, we further have
\begin{align*}
&\#\{\bfy\,(\mod p^{\gamma-1}):\, s\,(\nabla F)(\bfy)+\bfm\equiv \bfo\, (\mod p)\} \\
& = \#\{(\bfw,\bfz)\in (\bZ/p\bZ)^n\times (\bZ/p^{\gamma-2}\bZ)^n:\, s\,(\nabla F)(\bfw+p\bfz)+\bfm\equiv \bfo\, (\mod p)\} \\
& = p^{(\gamma-2)n}\cdot \#\{\bfw\,(\mod p):\,s\,(\nabla F)(\bfw) +\bfm\equiv \bfo\, (\mod p) \}.
\end{align*}
%The last steps uses condition (\romannumeral1) and Hensel's lemma.
By \eqref{eq_cond} again, one deduces that
\[
G_{\gamma}(\bfm,s)\leq p^n\cdot p^{(\gamma-2)n}\cdot c_2 p^{\kappa} = c_2 p^{(\gamma-1)n+\kappa}.
\]
The proof is completed.
\end{proof}

%{\color{red}Connecting lines and references}

\begin{lemma} \label{lemma_exponentialsum}
Let $a$ be an integer with $(a,p)=1$. Let $k\geq 2$ be a natural number with $(k,p)=1$. Then for any $m\,(\mod p)$,
\begin{equation} \label{eq_onevariable_req1}
\sum\limits_{s \,(\mod p)} \left|\sum\limits_{x\,(\mod p)} e_{p}(-m x+sax^k)\right|^2 \ll_k p^2.
\end{equation}
Moreover, for any natural number $r\geq 2$, it satiesfies that
\begin{equation} \label{eq_onevariable_rgeq2}
\sum\limits_{s \,(\mod p^r)} \left|\sum\limits_{x\,(\mod p^r)} e_{p^r}(-m x+sax^k)\right|^2 \ll_k
\begin{cases}
p^{2r},\quad &\text{if }{\color{black}\ord_p}(m)<r-\lceil \frac{r}{k}\rceil,\\
p^{(3-2/k)r},& \text{if }{\color{black}\ord_p}(m)\geq r-\lceil \frac{r}{k}\rceil,
\end{cases}
\end{equation}
for $m \in \bZ/p^r\bZ$.
\end{lemma}

\begin{proof}
Note that
\begin{align*}
&\sum\limits_{s\,(\mod p^r)} \left|\sum\limits_{x\,(\mod p^r)}e_{p^r}(-m x+sax^k)\right|^2 = \sum\limits_{x,y\,(\mod p^r)} \sum\limits_{s\,(\mod p^r)}e_{p^r}\big(-m(x-y)+sa(x^k-y^k)\big)\\
&\qquad=p^r \sum\limits_{x,y\,(\mod p^r)\atop x^k\equiv y^k\,(\mod p^r)} e_{p^r}\big(-m (x-y)\big) = p^r \sum\limits_{j\,(\mod p^r)}\Big|\sum\limits_{x\,(\mod p^r)\atop x^k\equiv j\,(\mod p^r)}e_{p^r}(-mx)\Big|^2.
\end{align*}
When $j\not\equiv 0\,(\mod p)$, one has
\begin{equation} \label{eq_onevaraible}
\Big|\sum\limits_{x\,(\mod p^r)\atop x^k\equiv j\,(\mod p^r)}e_{p^r}(-mx)\Big| \leq \#\{x\,(\mod p^r):\, x^k\equiv j\,(\mod p^r)\}.
\end{equation}
When $r=1$, it is not hard to see that the above quantity is $\ll_k 1$. By Hensel's lemma, we conclude that \eqref{eq_onevaraible} is $\ll_k 1$ for any $r\geq 1$. Hence
\[
p^r \sum\limits_{j\not\equiv 0\,(\mod p^r)}\Big|\sum\limits_{x\,(\mod p^r)\atop x^k\equiv j\,(\mod p^r)}e_{p^r}(-mx)\Big|^2 \ll p^{2r}.
\]

Next, we consider the case that $j\equiv 0\,(\mod p)$. When $r=1$, it is easy to see that
\[
p\cdot \Big|\sum\limits_{x\,(\mod p)\atop x^k\equiv 0\,(\mod p)}e_p(-mx)\Big|^2 =p.
\]
When $r\geq 2$, let us take $\mu=\lceil \frac{r}{k}\rceil$. Then $x^k\equiv 0\,(\mod p^r)$ if and only if $p^\mu| x$. By writing $x=p^\mu y$, we obtain that
\[
\sum\limits_{x\,(\mod p^r)\atop x^k\equiv 0\,(\mod p^r)}e_{p^r}(-mx) = \sum\limits_{y\,(\mod p^{r-\mu})} e_{p^{r-\mu}}(-my) \ll p^{r-\mu} \ll p^{r(1-1/k)}.
\]
Donote $\nu=\ord_p(m)$ and write $m=p^{\nu}\widetilde{m}$.  When $1\leq \nu<r-\mu$, one deduces by a change of variables $y=w+p^{r-\mu-\nu}z$ that
\[
\sum\limits_{y\,(\mod p^{r-\mu})} e_{p^{r-\mu}}(-my) = \sum\limits_{w\,(\mod p^{r-\mu-\nu})}e_{p^{r-\mu-\nu}}(-\tilde{m}w)\sum\limits_{z\,(\mod p^\nu)}1 =0.
\]
When $\nu=0$, the above result also holds. So the summand over $j\equiv 0\,(\mod p^r)$ is
\[
p^r\cdot \Big|\sum\limits_{x\,(\mod p^r)\atop x^k\equiv 0\,(\mod p^r)}e_{p^r}(-mx)\Big|^2 \ll
\begin{cases}
0,\quad &\text{if }\nu<r-\lceil \frac{r}{k}\rceil,\\
p^{r(3-2/k)},& \text{if }\nu\geq r-\lceil \frac{r}{k}\rceil.
\end{cases}
\]
The lemma then follows.
\end{proof}

\section{Proof of Theorems \ref{main}, \ref{main2}, and Corollary \ref{cor_diag_hom}}

For simplicity, we write $q=p^r$ in this section.

\begin{proof} [Proof of Therorem \ref{main}]
Recall that $S_{r, j}=\{\bfz\in (\mathbb{Z}/q\mathbb{Z})^n\colon F(\bfz)\equiv j\, (\mod q)\}$. First, we have
\begin{align*}
\cN_{r,j}:=&\#\{(\bfx,\bfy)\in E_1\times E_2:\, F(\bfx-\bfy)\equiv j\, (\mod q)\} \\
&=\sum\limits_{\bfx,\bfy\, (\mod q)} 1_{E_1}(\bfx)1_{E_2}(\bfy)1_{S_{r,j}}(\bfx-\bfy) = q^{2n} \cdot \frac{1}{q^n} \sum\limits_{\bfx\, (\mod q)} 1_{E_1}(\bfx)(1_{E_2}\ast 1_{S_{r,j}})(\bfx)\\
&=q^{2n} \sum\limits_{\bfm\, (\mod q)} \overline{\widehat{1_{E_1}}(\bfm)}\widehat{1_{E_1}\ast 1_{S_{r,j}}}(\bfm) = q^{2n} \sum\limits_{\bfm\, (\mod q)} \overline{\widehat{1_{E_1}}(\bfm)}\widehat{1_{E_1}}(\bfm)\widehat{1_{S_{r,j}}}(\bfm) =\cM+\cE,
\end{align*}
where
\[
\cM= q^{-n}|E_1||E_2||{S_{r,j}}|~~\mbox{and}~~ \cE=q^{2n}\sum\limits_{\bfm\not\equiv  \bfo \, (\mod q)}\overline{\widehat{1_{E_1}}(\bfm)}\widehat{1_{E_1}}(\bfm)\widehat{1_{S_{r,j}}}(\bfm).
\]

By the condition (\romannumeral2), one has
\[
|\#S_{1,j}-p^{n-1}| \leq c_1 p^{n-1}.
\]
With the condition (\romannumeral1), we may apply Hensel's lemma to get
\[
\big|\#S_{r,j} - q^{n-1}\big|\leq c_1 q^{n-1}
\]
for all $r,j$ with $r\geq 1$ and $(j,p)=1$. %Then
%\[
%.
%\]

Applying Lemma \ref{lemma_fourier}, one obtains that
\[
\left\vert\widehat{1_{S_{r,j}}}(\bfm)\right\vert \leq \max\{c_2,c_3\}\cdot q^{-1}p^{-\frac{n-1}{2}}
\]
{\color{black}for $\bfm\not\equiv \bfo\,(\mod q)$.}
Hence,
\begin{align*}
|\cE| &\leq q^{2n} \cdot\sup\limits_{\bfm\not\equiv \bfo\, (\mod q)}\big|\widehat{1_{S_{r,j}}}(\bfm)\big|\cdot\left(\sum\limits_{\bfm\, (\mod q)}\big|\overline{\widehat{1_{E_1}}(\bfm)}\big|^2\right)^{1/2}\left(\sum\limits_{\bfm\, (\mod q)}\big|\overline{\widehat{1_{E_2}}(\bfm)}\big|^2\right)^{1/2}\\
&=q^n|E_1|^{1/2}|E_2|^{1/2}\cdot \sup\limits_{\bfm\not\equiv \bfo \, (\mod q)}\big|\widehat{1_{S_{r,j}}}(\bfm)\big| \leq \max\{c_2,c_3\}\cdot  \sqrt{|E_1||E_2|}\cdot q^{n-1}p^{-(n-1)/2}.
\end{align*}

Combining all bounds, we conclude that
\[
\left|\cN_{r,j} -\frac{|E_1||E_2|}{q}\right|\leq  \frac{c_1 |E_1||E_2|}{q} + \max\{c_2,c_3\}\cdot  \sqrt{|E_1||E_2|}\cdot q^{n-1}p^{-(n-1)/2}.
\]
So $\cN_{r,j}>0$ if
\[
\frac{\sqrt{|E_1||E_2|}}{q^n} > \frac{\max\{c_2,c_3\}}{(1-c_1) p^{\frac{n-1}{2}}}.
\]
The proof is completed.
\end{proof}

\begin{proof}[Proof of Corollary \ref{cor_diag_hom}]
Let $p>\max\{k(k-1), |a_1|,\ldots,|a_n|\}$ be sufficiently large. Then $(k,p)=1$ and $(a_i,p)=1$ for all $1\leq i\leq n$.

Let us verify the four conditions in Theorem \ref{main} for {\color{black}$F(\bfx)=\sum\nolimits_{i=1}^n a_ix_i^k$} for all $j\in \bZ$ with $(j,p)=1$. We have
\[
(\nabla F)(\bfx) =  (ka_1x_1^{k-1},\ldots, ka_n x_n^{k-1}) \not\equiv \bfo \, (\mod p)
\]
when $\bfx \not\equiv \bfo\,(\mod p)$. Then the condition (\romannumeral1) holds.  Moreover, the conditions (\romannumeral2) and (\romannumeral3) follow from Lemmas \ref{lem_solutions} and \ref{lem_Fourier_diag_hom}, respectively. In the following, let us verify condition (\romannumeral4).

For $\bfm{\color{black}=(m_1,m_2,\ldots,m_n)}\not\equiv \bfo\,(\mod p)$ and $(j,p)=1$, the cardinality of the set in the condition (\romannumeral4), denoted by $\fN$, satisfies that
\begin{align*}
\fN &\leq \sum\limits_{1\leq t\leq n\atop m_t\not\equiv 0\,(\mod p)} \#\left\{\bfx\,(\mod p):\, \sum\limits_{i=1}^n a_ix_i^k \equiv j\,(\mod p),\right.\\
&\qquad \qquad \left.x_t\not\equiv 0\,(\mod p),\, m_i a_t x_t^{k-1}\equiv m_ta_i x_i^{k-1}\, (\mod p)\, \mbox{for all}~ 1\leq i\leq n\right\}.
\end{align*}
Thus, for the $\bfx$ taken into account, we have $x_i^{k-1}\equiv (m_ta_i)^{-1}m_i a_t x_t^{k-1}\, (\mod p)$.

When $k=2$, it follows that
\[
\fN \leq \sum\limits_{1\leq t\leq n\atop m_t\not\equiv 0\,(\mod p)} \#\left\{x_t\,(\mod p):\, \Big(\sum\limits_{i=1}^n \big((m_ta_i)^{-1}m_ia_t\big)^2 a_i\Big)x_t^2 \equiv j\,(\mod p)\right\}\leq 2n\ll 1.
\]
When $k=3$, we further have
\[
x_i^{3}\equiv x_i^{3+p}  \equiv \big((m_ta_i)^{-1}m_i a_t x_t^2\big)^{\frac{3+p}{2}}\equiv \big((m_ta_i)^{-1}m_i a_t \big)^{\frac{3+p}{2}}x_t^3\, (\mod p).
\]
So
\[
\fN \leq \sum\limits_{1\leq t\leq n\atop m_t\not\equiv 0\,(\mod p)} \#\left\{x_t\,(\mod p):\, \Big(\sum\limits_{i=1}^n \big((m_ta_i)^{-1}m_i a_t \big)^{\frac{3+p}{2}} a_i\Big)x_t^3 \equiv j\,(\mod p)\right\}\leq 3n\ll 1.
\]

When $k\geq 4$, we have
\[
\fN\leq \sum\limits_{1\leq t\leq n\atop m_t\not\equiv 0\,(\mod p)} \#\left\{\bfx\,(\mod p):\, x_i^{k-1}\equiv (m_ta_i)^{-1}m_i a_t x_t^{k-1}\, (\mod p)\,\mbox{for all}~ 1\leq i\leq n\right\}.
\]
The variable $x_t$ can be chosen mod $p$, and each $x_i$ $(i\neq t)$ has at most $k-1$ choices. So
\[
\fN  \leq n\cdot p\cdot k^{n-1} \ll p \ll p^{\frac{n-1}{2}}
\]
when $n\geq 3$.

Finally, when $n=2$, let us deal with $t=1$ without loss of generality. In this situation, one has $m_1,x_1\not\equiv 0\, (\mod p)$. Note that
\[
j\equiv a_1x_1^k+ a_2x_2^k \equiv a_1x_1^k+ a_2(m_1a_2)^{-1}m_2 a_1 x_1^{k-1} x_2\equiv a_1 x_1^{k-1}\big(x_1+m_1^{-1}m_2x_2\big)\, (\mod p).
\]
If $m_2\equiv 0\, (\mod p)$, then the above equivalence has at most $k-1$ solutions in $x_1\, (\mod p)$. Otherwise we have $x_2\equiv m_1m_2^{-1}(ja_1^{-1}x_1^{-(k-1)}-x_1)\, (\mod p)$. Then
\begin{align*}
(m_1a_2)^{-1}m_2 a_1 x_1^{k-1} \equiv x_2^{k-1} \equiv \left(m_1m_2^{-1}(ja_1^{-1}x_1^{-(k-1)}-x_1)\right)^{k-1}\,(\mod p),
\end{align*}
i.e.,
\begin{align*}
&\big(-(m_1a_2)^{-1}m_2a_1+(-1)^{k-1}\big)x_1^{k(k-1)} + (-1)^{k-2}m_1^{k-1}m_2^{-(k-1)}(k-1)(ja_1^{-1})x_1^{k(k-2)}\\
&\qquad +\sum\limits_{i=2}^{k-1} (-1)^{k-1-i} {k-1\choose i}(ja_1^{-1})^i x_1^{k(k-1-i)}\equiv 0\, (\mod p).
\end{align*}
This involves a non-zero polynomial of degree no larger than $k(k-1)$. So there are at most $k(k-1)$ solutions in $x_1\, (\mod p)$. Now we conclude that $\fN \leq 2k(k-1)^2\ll 1$.

Corollary \ref{cor_diag_hom} then follows from Theorem \ref{main}.
\end{proof}

\begin{proof} [Proof of Theorem \ref{main2}]
%Let $p$ be sufficiently large so that $(k_i,p)=1$ and $(a_i,p)=1$ $(1\leq i\leq p)$.
For sufficiently large $p$, we have
\[
(\nabla F)(x) = (k_1 a_1x_1^{k_1-1},\ldots, k_na_nx_n^{k_n-1}) \not\equiv \bfo\, (\mod p)
\]
when $\bfx\not\equiv \bfo\,(\mod p)$. Hensel's lemma can be applied on the roots of $F(\bfx)\equiv j\,(\mod p^l)$ $(l\geq 1)$. Combining Lemma \ref{lem_solutions}, one can obtain that $|S_{r,j}|\ll q^{n-1}$. Then we deduce similarly as in the proof of Theorem \ref{main} that
\[
\cN_{r,j} \ll_F q^{-1}|E_1||E_2|+ q^{2n} \cdot\sum\limits_{\bfm\not\equiv  \bfo \, (\mod q)}\overline{\widehat{1_{E_1}}(\bfm)}\widehat{1_{E_2}}(\bfm)\widehat{1_{S_{r,j}}}(\bfm).
\]
Now
\begin{align*}
\sum\limits_{j\,(\mod q)}|\cN_{r,j}|^2 & \ll \Sigma_1+\Sigma_2,
\end{align*}
where
\[
\Sigma_1 = \sum\limits_{j\,(\mod q)}\left(q^{-1}|E_1 ||E_2 |\right)^2=q^{-1}|E_1 |^2|E_2 |^2,
\]
and
\[
\Sigma_2 = q^{4n}\sum\limits_{\bfm ,\bfl  \not\equiv \bfo \, (\mod q)}\overline{\widehat{1_{E_1 }}(\bfm )}\widehat{1_{E_2 }}(\bfm )\widehat{1_{E_1 }}(\bfl )\overline{\widehat{1_{E_2 }}(\bfl )}\sum\limits_{j\,(\mod q)}\widehat{1_{S _{r,j}}}(\bfm ) \overline{\widehat{1_{S _{r,j}}}(\bfl )}.
\]

Note that
\begin{align*}
\Sigma_3:=&\sum\limits_{j\,(\mod q)}\widehat{1_{S _{r,j}}}(\bfm ) \overline{\widehat{1_{S _{r,j}}}(\bfl )}\\
& = \frac{1}{q^{2n}}\sum\limits_{j\,(\mod q)} \sum\limits_{\bfx \,(\mod q)\atop F(\bfx)\equiv j\,(\mod q)}  \sum\limits_{\bfy \,(\mod q)\atop F(\bfy)\equiv j\,(\mod q)} e_q(-\bfm \cdot \bfx +\bfl \cdot \bfy) \\
& = \frac{1}{q^{2n}} \sum\limits_{\bfx,\bfy \,(\mod q)\atop F(\bfx)\equiv F(\bfy)\,(\mod q)}   e_q(-\bfm \cdot \bfx+\bfl\cdot \bfy)\\
& = \frac{1}{q^{2n+1}} \sum\limits_{s\,(\mod q)}\sum\limits_{\bfx,\bfy \,(\mod q)}   e_q(-\bfm \cdot \bfx+\bfl\cdot \bfy+sF(\bfx)-sF(\bfy))\\
& = \frac{1}{q^{2n+1}} \sum\limits_{s\not\equiv 0\,(\mod q)}\sum\limits_{\bfx,\bfy \,(\mod q)}   e_q(-\bfm \cdot \bfx+\bfl\cdot \bfy+sF(\bfx)-sF(\bfy)),
\end{align*}
where the last two steps follows from the orthogonality and the fact that $\bfm,\bfl \not\equiv \bfo\,(\mod q)$. %Without loss of generality, we assume that $m_1,l_1\not\equiv 0\,(\mod q)$.

Noting that the inner summation is a product of exponential sums in one variable, one has
\[
|\Sigma_3| \leq \frac{1}{q^{2n+1}}\cdot \sum\limits_{s\not\equiv 0 \,(\mod q)} \left|\prod\limits_{i=1}^n \sum\limits_{x_i\,(\mod q)}e_q\big(-m_ix_i+sa_ix_i^{k_i}\big)\cdot \prod\limits_{i=1}^n\sum\limits_{y_i\,(\mod q)} e_q\big(l_iy_i-sa_iy_i^{k_i}\big)\right|.
\]
For $(m_i,s),(l_i,s)\not\equiv (0,0)\,(\mod p)$, the condition \eqref{eq_cond0} in Lemma \ref{lem_Fourier2} follows from Lemma \ref{lem_Weil}, and the condition \eqref{eq_cond} in Lemma \ref{lem_Fourier2} is confirmed by observing that
\[
\#\{x_i(\mod p):\, -m_i+sa_ik_ix_i^{k_i-1}\equiv 0\,(\mod p)\},\,\#\{y_i(\mod p):\,-l_i+sa_ik_iy_i^{k_i-1}\equiv 0\,(\mod p)\}\leq k_i-1.
\]
For $s\not\equiv 0\,(\mod q)$, we deduce from Lemma \ref{lem_Fourier2} that
\begin{equation} \label{eq_onevariable_2n-2}
\left|\sum\limits_{x_i\,(\mod q)}e_q\big(-m_ix_i+sa_ix_i^{k_i}\big)\right|,\,\, \left|\sum\limits_{y_i\,(\mod q)} e_q\big(l_iy_i-sa_iy_i^{k_i}\big)\right|\,\, \ll_k \, qp^{-1/2}
\end{equation}
for any $m_i,l_i\,(\mod q)$ with $1\leq i\leq n$. %Thus,
%\[
%|\Sigma_3| \leq \frac{1}{q^{2n+1}}\cdot \sum\limits_{s\not\equiv 0 \,(\mod q)} \left|\prod\limits_{i=1}^n \sum\limits_{x_i\,(\mod q)}e_q\big(-m_ix_i+sa_ix_i^{k_i}\big)\cdot \prod\limits_{i=1}^n\sum\limits_{y_i\,(\mod q)} %e_q\big(l_iy_i-sa_iy_i^{k_i}\big)\right|.
%\]
Moreover, let us assume without loss of generality that $k_1=k_\ast=\min\nolimits_{1\leq i\leq n}k_i$. Then
\begin{align}
&|\Sigma_3| \ll  q^{-2n-1}\cdot \left(q p^{-1/2}\right)^{2n-2}\cdot \sum\limits_{s\not\equiv 0\,(\mod q)} \left|\sum\limits_{x_1\,(\mod q)}e_q\big(-m_1x_1+sa_1x_1^{k_1}\big)\cdot \sum\limits_{y_1\,(\mod q)} e_q\big(l_1y_1-sa_1y_1^{k_1}\big)\right| \nonumber\\
& \leq q^{-3}p^{-n+1}\cdot  \left(\sum\limits_{s\,(\mod q)}  \left|\sum\limits_{x_1\,(\mod q)}e_q\big(-m_1x_i+sa_1x_1^{k_1}\big)\right|^{2}\right)^{\frac{1}{2}}\left(\sum\limits_{s\,(\mod q)}  \left|\sum\limits_{y_1\,(\mod q)}e_q\big(-l_1x_1+sa_1y_1^{k_1}\big)\right|^{2}\right)^{\frac{1}{2}} \nonumber\\
&\leq q^{-3}p^{-n+1}\cdot q^{3-2/k_\ast}
= q^{-2/k_\ast}p^{-n+1}, \label{eq_improve_req1}
\end{align}
where we have applied \eqref{eq_onevariable_2n-2}, the Cauchy-Schwartz inequality and Lemma \ref{lemma_exponentialsum}.

It follows that
\begin{align*}
\Sigma_2&\ll q^{4n}\cdot q^{-2/k_\ast}p^{-n+1}\cdot \prod\limits_{i=1,2}\left(\sum\limits_{\bfm (\mod q)}|\widehat{1_{E_i}}(\bfm )|^2\right)^{1/2}\cdot \prod\limits_{i=1,2}\left(\sum\limits_{\bfl (\mod q)}|\widehat{1_{E_i}}(\bfl )|^2\right)^{1/2}\\
& \ll q^{4n}\cdot q^{-2/k_\ast}p^{-n+1}\cdot q^{-2n}|E_1||E_2| = q^{2n-2/k\ast} p^{-n+1} |E_1||E_2|.
\end{align*}
As a result, we have
\[
\sum\limits_{j\,(\mod q)}\cN_{r,j}^2 \ll q^{-1}|E_1|^2|E_2|^2 + q^{2n-2/k_\ast} p^{-n+1} |E_1||E_2|.
\]

Moreover, since
\[
|E_1 |^2|E_2 |^2 = \left(\sum\limits_{j\in \Delta_{n,r} (E_1 ,E_2 )}\cN_{r,j}\right)^2\leq |\Delta_{n,r} (E_1 ,E_2 )|\cdot \sum\limits_{j\,(\mod q)}\cN_{r,j}^2,
\]
we conclude that
\[
|\Delta_{n,r} (E_1 ,E_2 )| \geq \frac{|E_1 |^2|E_2 |^2}{\sum\limits_{j\,(\mod q)}\cN_{r,j}^2} \gg \min\{q,\,q^{-2n+2/k_\ast}p^{n-1}|E_1||E_2|\}.
\]
In particular, one has $|\Delta_{n,r} (E_1 ,E_2)|\gg q$ if
\[
\frac{\sqrt{|E_1 ||E_2 |}}{q^{n}} \gg q^{\frac{1}{2}-\frac{1}{k_\ast}}p^{-\frac{n-1}{2}}.
\]
This completes the proof.
\end{proof}

\begin{remark}
For the case $r=1$, we may use \eqref{eq_onevariable_req1} instead of \eqref{eq_onevariable_rgeq2} in the above proof, and obtain $\ll p^{-n}$ at the end of \eqref{eq_improve_req1}. As a consequence, the same result \eqref{eq_main2_r=1} as in \cite{KS} can be obtained.
\end{remark}

\section{Extension estimates associated to circles and orbits}
Let $V$ be a variety over $\mathbb{F}_{p}^n$. The $L^u\to L^{u'}$ Fourier extension problem for $V$ asks us to determine all exponents $1\le u, u'\le \infty$ such that the following inequality
\begin{equation}\label{extension}||(fd\sigma)^\vee||_{L^{u'}(\mathbb{F}_p^n, dm)}\le C ||f||_{L^u(V, d\sigma)},\end{equation}
holds for some constant $C>0$ and  all complex valued functions $f$ on $V$. Here  $dm$ is the counting measure on $\bF_p^n$, $d\sigma$ is the normalized surface measure on $V$, and
 \[(fd\sigma)^\vee(m)=\frac{1}{|V|}\sum_{x\in V}f(x)e_p(m\cdot x),\]
 \[||f||_{L^u(V, d\sigma)}=\left(\frac{1}{|V|}\sum_{x\in V}|f(x)|^u\right)^{1/u},\]
 and \[||g||_{L^{u'}(\mathbb{F}_p^n, dm)}=\left(\sum_{m\in \mathbb{F}_p^n}|g(m)|^{u'}\right)^{1/u'},\] for any functions $f$ and $g$. The $L^u\to L^{u'}$ Fourier extension problem has been studied intensively in the literature for paraboloids, spheres, cones, and homogeneous varieties with applications in several areas of Mathematics including \textit{Discrete Geometry} and \textit{Combinatorial Number Theory}. We refer the interested reader to a series of papers \cite{IK09, IK10, IKL, european, Le13, MT04} for more discussions.

When $V$ is a circle in the plane $\mathbb{F}_p^2$, it has been proved in \cite{chapman} that
\[||(fd\sigma)^\vee||_{L^4(\mathbb{F}_p^2, dm)}\ll ||f||_{L^2(V, d\sigma)}.\]
In this paper, to improve Corollary \ref{cor_diag_hom} in two dimensions, we need to study the finite p-adic ring analog of this estimate. The extension estimates are of independent interest and are expected to have more applications in other topics.

\subsection{Extension theorems}

In the rest of this section, we always assume that $n=2$. Recall that
\[
C_{r,j}=\{\bfx\in (\bZ/p^r\bZ)^2:\, \|\bfx\|\equiv j\,(\mod p^r)\}.
\]
We denote by $d\sigma_{r,j}$ the normalised surface measure on $C_{r,j}$. %Our results will be presented separately for two cases $p\equiv 3\, (\mod 4)$ and $p\equiv 1\, (\mod 4)$.

%\subsubsection{Extension theorems associated to circles}
\begin{theorem}\label{resitriction} \color{black}
Let {\color{black}$p$ be an odd prime,} and 
$r\ge 1$ be an integer. Let $j\in (\mathbb{Z}/p^r\mathbb{Z})^\ast$. Then
 \[
 \Big(\sum_{m\in (\mathbb{Z}/p^r\mathbb{Z})^2}|(fd\sigma_{r,j})^\vee(m)|^4\Big)^{1/2}\ll p^{-\frac{r+1}{2}}\sum_{x\in C_{r,j}}|f(x)|^2.
 \]
%\begin{enumerate}
%    \item {\color{black}If $j\not\equiv 0\,(\mod p)$}, then
%     \[\left(\sum_{m\in (\mathbb{Z}/p^r\mathbb{Z})^2}|(fd\sigma_{r,j})^\vee(m)|^4\right)^{1/2}\ll {\color{black}p^{-\frac{r+1}{2}}}\sum_{x\in C_{r,j}}|f(x)|^2.\]
%     \item {\color{black}If $j\equiv 0\,(\mod p)$ and $j\not\equiv 0\,(\mod p^r)$, then}
%      \[\left(\sum_{m\in (\mathbb{Z}/p^r\mathbb{Z})^2}|(fd\sigma_{r,j})^\vee(m)|^4\right)^{1/2}\ll {\color{black}p^{-1}}\sum_{x\in C_{r,j}}|f(x)|^2.\]
%      \item {\color{black}If $j\equiv 0\,(\mod p^r)$, we have
%      \[\left(\sum_{m\in (\mathbb{Z}/p^r\mathbb{Z})^2}|(fd\sigma_j)^\vee(m)|^4\right)^{1/2}\ll p^{(-1)^{r-1}}\sum_{x\in C_{r,j}}|f(x)|^2.\]
%      }
%\end{enumerate}
\end{theorem}
%\begin{theorem}\label{resitriction4} \color{black}
%Let $p\equiv 1\, (\mod 4)$ and $r\ge 1$ be an integer. Let $j\in (\mathbb{Z}/p^r\mathbb{Z})^\ast$.
% \[
% \Big(\sum_{m\in (\mathbb{Z}/p^r\mathbb{Z})^2}|(fd\sigma_{r,j})^\vee(m)|^4\Big)^{1/2}\ll p^{-\frac{r+1}{2}}\,\sum_{x\in C_{r,j}}|f(x)|^2.
% \]
%\begin{enumerate}
%    \item If $j\in (\mathbb{Z}/p^r\mathbb{Z})^*$, we have
%     \[\left(\sum_{m\in (\mathbb{Z}/p^r\mathbb{Z})^2}|(fd\sigma_{r,j})^\vee(m)|^4\right)^{1/2}\ll \frac{p^{(r-1)/2}}{p^r}\sum_{x\in C_{r,j}}|f(x)|^2.\]
%     \item If $j\notin (\mathbb{Z}/p^r\mathbb{Z})^\ast$ with $\ord_p(j)=v$, then
%      \[\left(\sum_{m\in (\mathbb{Z}/p^r\mathbb{Z})^2}|(fd\sigma_{r,j})^\vee(m)|^4\right)^{1/2}\ll \frac{1}{v_j^2p^{1/2}}\sum_{x\in C_{r,j}}|f(x)|^2.\]
%\end{enumerate}
%\end{theorem}

%\subsubsection{Extension theorems associated to orbits}
We now define a more general extension problem.
Denote
\[
G_r:=SO_2(\bZ/p^r\bZ)=\left\{\begin{bmatrix}
a &-b\\
b & a
\end{bmatrix}\,(\mod p^r):\,\, a^2+b^2\equiv 1\,(\mod p^r)\right\}
\]
for $r\geq 1$. The group $G_r$ acts on $(\bZ/p^r\bZ)^2$ naturally by $\bfx \mapsto \theta \bfx$, where $\theta \in G_r$ and $\bfx\in (\bZ/p^r\bZ)^2$. Write the orbit of $\bfx$ by
\[
\textsf{orb}_r(\bfx) = \big\{\theta \bfx:\, \theta \in G_r\big\},
\]
and the stabilizer of $\bfx$ is
\[
\textsf{stab}_r(\bfx) = \{\theta\in G_r:\, \theta \bfx \equiv \bfx\, (\mod p^r)\}.
\]
%For any $\mathbf{m}\in (\mathbb{Z}/p^r\mathbb{Z})^2$, let $d\sigma_{V_\mathbf{m}}$ be the corresponding surface measure on $\textsf{orb}_r(\bfm)$.

%\begin{theorem}\label{restriction2}
%    Let $p\equiv 3\, (\mod 4)$ and $r$ be an odd integer. For $j\in \mathbb{Z}/p^r\mathbb{Z}$. Assume that $j=p^{2u}t\not\equiv 0 \,(\mod p^r)$ with $t\in U_r$. Then we have
%      \[\left(\sum_{x\in (\mathbb{Z}/p^r\mathbb{Z})^2}|(fd\sigma_{V_{j, \mathbf{m}}})^\vee(x)|^4\right)^{1/2}\ll \frac{p^{2u+(r-u-1)/2}}{p^r}\sum_{x\in V}|f(x)|^2.\]
%\end{theorem}

{\color{black}For $\bfm=(m_1,m_2)\in (\bZ/p^r\bZ)^2$, recall that $v_\bfm=\min\{\ord_p(m_1),\ord_p(m_2)\}$.}

\begin{theorem}\label{restriction3}
    Let $p\equiv 3\, (\mod 4)$ and $r\ge 1$ be an integer. Let $\bfm\in (\bZ/p^r\bZ)^2$ be such that $\bfm\neq \bfo$. Denote by $d{\color{black}\sigma_r}$ the normalized surface measure on $\textsf{orb}_r(\bfm)$. Then
    \[
    \Big(\sum_{x\in (\mathbb{Z}/p^r\mathbb{Z})^2}|(fd\sigma_r)^\vee(x)|^4\Big)^{1/2}\ll  p^{-\frac{r-3v_\bfm+1}{2}}\sum_{x\in \textsf{orb}_r(\bfm)}|f(x)|^2.
    \]
%\begin{enumerate}
%    \item If $\bfm=p^{v_{\bfm}}\tilde{\bfm}\not\equiv \mathbf{0}\,(\mod p^r)$ with $v_{\tilde{\bfm}}=0$ and %$||\tilde{\bfm}||\not\equiv 0\, (\mod p)$, we have
%      \[\left(\sum_{x\in (\mathbb{Z}/p^r\mathbb{Z})^2}|(fd\sigma_{V_\mathbf{m}})^\vee(x)|^4\right)^{1/2}\ll  \frac{p^{(r+3v_{\bfm}-1)/2}}{p^r}\sum_{x\in V_\mathbf{m}}|f(x)|^2.\]
%\item If $\bfm=p^{v_{\bfm}}\tilde{\bfm}\not\equiv \mathbf{0}\,(\mod p^r)$ with $v_{\tilde{\bfm}}=0$ and $||\tilde{\bfm}||\equiv 0\, (\mod p)$, we have
%      \[\left(\sum_{x\in (\mathbb{Z}/p^r\mathbb{Z})^2}|(fd\sigma_{V_\mathbf{m}})^\vee(x)|^4\right)^{1/2}\ll  p^{v_{\bfm}-1}\sum_{x\in V_\mathbf{m}}|f(x)|^2.\]
%\end{enumerate}
\end{theorem}

\begin{theorem}\label{resitriction5}
    Let $p\equiv 1 \, (\mod 4)$ and $r \ge 1$ be an integer. Let $\bfm\in (\bZ/p^r\bZ)^2$ be such that $\bfm\neq \bfo$. Denote by $d\sigma_r$ the normalized surface measure on $\textsf{orb}_r(\bfm)$. Suppose that $\bfm=p^{v_m}\tilde{\bfm}$ with $\tilde{\bfm}\in (\bZ/p^{r-v_m}\bZ)^2$ and {\color{black}$\|\tilde{\bfm}\|\not\equiv 0\,(\mod p)$}. Then
    \[
    \Big(\sum_{x\in (\mathbb{Z}/p^r\mathbb{Z})^2}|(fd\sigma_r)^\vee(x)|^4\Big)^{1/2}\ll  p^{-\frac{\color{black}r-3v_\bfm+1}{2}}\sum_{x\in \textsf{orb}_r(\bfm)}|f(x)|^2.
    \]
%        \item  If $\bfm=p^{v_{\bfm}}\tilde{\bfm}\not\equiv \mathbf{0}\,(\mod p^r)$ with $v_{\tilde{\bfm}}=0$ and $||\tilde{\bfm}||\not\equiv 0\, (\mod p)$, we have
%        \[ \left( \sum_{x\in (\bZ /p^r\bZ)^2} \vert (fd\sigma_{V_\bfm})^{\vee}(x) \vert^4\right)^{1/2} \ll \frac{p^{(r+3v_{\bfm}-1)/2}}{p^{r}}\sum_{x\in V_\bfm} \vert f(x)\vert^2 .\]
%        \item If $\bfm=p^{v_{\bfm}}\tilde{\bfm}\not\equiv \mathbf{0}\,(\mod p^r)$ with $v_{\tilde{\bfm}}=0$ and $||\tilde{\bfm}||\equiv 0\, (\mod p)$, we have
%        \[  \left( \sum_{x\in (\bZ /p^r\bZ)^2} \vert (fd\sigma_{V_\bfm})^{\vee}(x) \vert^4\right)^{1/2} \ll \frac{p^{(r+3v_{\bfm})/2}}{p^{r}}\sum_{x\in V_\bfm} \vert f(x)\vert^2 .\]
%    \end{enumerate}
\end{theorem}

\subsection{Preliminary lemmas}
To prove these extension theorems, we first need to collect and prove a number of preliminary results. All detailed proofs in this subsection are written under the assumption that $p\equiv 3\, (\mod 4)$. In the case $p\equiv 1\, (\mod 4)$, {\color{black}the arguments are identical but more complicated, and will be omitted in this paper.}

Let $p$ be a prime with $p\equiv 3\,(\mod 4)$, $r\geq 1$ be an integer. For any $l\in \bZ$, one has $(l,p)=1$ if and only if $(l,p^r)=1$. For any $\theta\in G_r$ and $\bfx\in (\bZ/p^r\bZ)^2$, it is not hard to verify that $\|\theta \bfx\| \equiv \|\bfx\| \quad (\mod p^r)$. Moreover, we have $v_{\theta\bfx}\geq v_\bfx$. Since $\theta$ is invertible, one obtains that $v_{\theta\bfx}= v_{\bfx}$.

It is well-known (see \cite{HLR} for example) that
\[
|C_{1,j}|=|G_1| = p+1,\quad (j\in (\bZ/p\bZ)^\ast ),\qquad |C_{1,0}|=1
\]
when $p\equiv 3\,(\mod 4)$. In particular, $C_{1,0}=\{\bfo\}$. For any $\bfx\in \bZ/p^r\bZ$, it follows that $\|\bfx\|\equiv 0\, (\mod p)$, if and only if $\bfx\equiv \bfo\,(\mod p)$, if and only if $v_\bfx>0$. Or equivalently, {\color{black}we have $\|\bfx\|\in (\bZ/p^r\bZ)^\ast$ if and only if $\|\bfx\|\not\equiv 0\,(\mod p)$, if and only if $\bfx\not\equiv \bfo\,(\mod p)$, if and only if $v_\bfx=0$.}

%It follows from Hensel's lemma that, for any $\bfz\in (\bZ/p^r\bZ)^2$, one has $v_\bfz=0$ if and only if $\|\bfz\|\not\equiv 0\, (\mod p^r)$.

\begin{lemma} \label{lem_car_Gr}
{\color{black}Let $p\equiv 3\,(\mod 4)$ and $r\geq 1$.} Then $|G_r| = p^r(1+1/p)$.
\end{lemma}

\begin{proof}
For $r=1$, we have $|G_1|=p+1$. Now consider the circumstances that $r\geq 2$. For any $\theta=\begin{bmatrix}a & -b\\b &a\end{bmatrix} \in G_r$, there is  some $\theta_0 = \begin{bmatrix}a_0 & -b_0\\b_0 &a_0\end{bmatrix}\in G_1$ such that $\theta\equiv \theta_0\,(\mod p)$. Applying Lemma \ref{lem_Hensel} to $(a_0,b_0)$ and the polynomial $F(x,y)=x^2+y^2-1$. Then $(\nabla F)(a_0,b_0) = (2a_0, 2b_0) \not\equiv \bfo \,(\mod p)$, since $a_0^2+b_0^2\equiv 1\,(\mod p)$. Thus,
\[
\#\left\{(z_1,z_2)\,(\mod p^{r-1}):\, (a_0+pz_1)^2+(b_0+pz_2)^2 \equiv 1\,(\mod p^r)\right\} = p^{r-1}.
\]
It follows that
\[
|G_r|=p^{r-1}|G_1|=p^{r}(1+1/p). %,\qquad |G_r^I|= p^{r-1}.
\]
\end{proof}

\begin{lemma} \label{lem_staborb}
{\color{black}Let $p\equiv 3\,(\mod 4)$ and $r\geq 1$.} For any $\bfo\neq \bfx\in (\bZ/p^r\bZ)^2$, let $\bfx=p^{v_\bfx}\tilde{\bfx}$ for some $\tilde{\bfx}\in (\bZ/p^{r-v_\bfx}\bZ)^2$ with $v_{\tilde{\bfx}}=0$. Then
\[
\textsf{orb}_r(\bfx) =  \big\{p^{v_\bfx} \theta_0\tilde{\bfx}:\, \theta_0 \in G_{r-v_\bfx}\big\} = p^{v_\bfx}C_{r-v_\bfx,\|\tilde{\bfx}\|}.
\]
And
\[
|\textsf{stab}_r(\bfx)| = p^{v_\bfx},\qquad |\textsf{orb}_r(\bfx)| =  p^{r-v_\bfx}(1+1/p).
\]
\end{lemma}

\begin{proof}
Write $v=v_\bfx$ for simplicity. Here $0\leq v\leq r-1$. Let $\tilde{\bfx}=(\tilde{x_1},\tilde{x_2})$. The equation $\theta \bfx \equiv \bfx\, (\mod p^r)$ is equivalent to
\begin{equation} \label{eq_rot1}
\begin{bmatrix}
a-1 &-b\\
b & a-1
\end{bmatrix}
\begin{bmatrix}
p^v\tilde{x_1}\\
p^v\tilde{x_2}
\end{bmatrix}\equiv
\begin{bmatrix}
0\\
0
\end{bmatrix}\quad (\mod p^r),
\end{equation}
or equivalently,
\begin{equation} \label{eq_rot2}
\begin{bmatrix}
\tilde{x_1} &-\tilde{x_2}\\
\tilde{x_2} & \tilde{x_1}
\end{bmatrix}
\begin{bmatrix}
a-1\\
b
\end{bmatrix}\equiv
\begin{bmatrix}
0\\
0
\end{bmatrix}\quad (\mod p^{r-v}).
\end{equation}
Since $v_{\tilde{\bfx}}=0$, one has $\tilde{x_1}^2+\tilde{x_2}^2 \not\equiv 0\,(\mod p)$ (recalling that $p\equiv 3\,(\mod 4)$). The coefficient matrix is invertible and
\[
\begin{bmatrix}
a\\
b
\end{bmatrix}\equiv
\begin{bmatrix}
1\\
0
\end{bmatrix},\quad (\mod p^{r-v}).
\]
By Lemma \ref{lem_Hensel}, the number of $(a,b)\in (\bZ/p^r \bZ)^2$ satisfying the above equivalence  and $a^2+b^2\equiv 1\,(\mod p^r)$ is exactly $p^v$. So
\[
\big|\textsf{stab}_r(\bfx)\big| = p^v.
\]
It then follows that $|\textsf{orb}_r(\bfx)| = |G_r|/|\textsf{stab}_r(\bfx)| = p^{r-v}(1+1/p)$.

Moreover, for any $\theta\in G_r$, there is some $\theta_0\in G_{r-v}$ such that $\theta \equiv \theta_0\,(\mod p^{r-v})$. It can be verified that $\theta \tilde{\bfx} \equiv \theta_0 \tilde{\bfx}\, (\mod p^{r-v})$. So
\[
\textsf{orb}_r(\bfx) = \big\{\theta (p^v \tilde{\bfx}):\, \theta \in G_r\big\} = \big\{p^v \theta_0\tilde{\bfx}:\, \theta_0 \in G_{r-v}\big\}.
\]
Note that $|G_{r-v}|=p^{r-v}(1+1/p)$ by Lemma \ref{lem_car_Gr}, the elements on the right-hand side of above formula give different members of the orbit. Furthermore, one has $\|\tilde{\bfx}\|\in (\bZ/p^{r-v}\bZ)^\ast$, since $v_{\tilde{\bfx}}=0$ and $p\equiv 3\,(\mod 4)$. We have $\{\theta_0\tilde{\bfx}:\, \theta_0\in G_{r-v}\}=C_{r-v,\|\tilde{\bfx}\|}$. The proof is completed.
\end{proof}

\begin{lemma}\label{circle-size}
{\color{black}Let $p\equiv 3\,(\mod 4)$ and $r\geq 1$.} Let $j$ be an integer with $\text{ord}_p(j)=v$. % , $0\leq v\leq r$, and $\tilde{j}$ be a unit in $\mathbb{Z}/p^{r-v}\bZ$.
Then
\[
|C_{r,j}| =
\begin{cases}
p^{r}(1+1/p),\quad &\text{if }0\leq v <r \text{ and }v\text{ is even},\\
p^{2\lfloor r/2 \rfloor},\quad &\text{if }v=r,\\
0,\quad &\text{otherwise}.
\end{cases}
\]
More concretely, we have
\[
C_{r,j} =
\begin{cases}
\left\{p^u \bfz+p^{r-u}\bfw:\,\bfz\in C_{r-2u,\tilde{j}},\, \bfw \in \bZ/p^u\bZ\right\},\quad &\text{if }j=p^{2u}\tilde{j},\, 0\leq u\leq \lfloor \frac{r-1}{2} \rfloor,\, \tilde{j}\in (\bZ/p^{r-2u}\bZ)^\ast\\
\left\{p^{\lceil r/2 \rceil} \bfw:\, \bfw \in \bZ/p^{\lfloor r/2 \rfloor}\bZ\right\},&\text{if }j=0,\\
\emptyset, &\text{otherwise}.
\end{cases}
\]
\end{lemma}
\begin{proof}

We first consider the case $j\in (\bZ/p^r\bZ)^\ast$. By applying Lemma \ref{lem_Hensel}, similar arguments as previous show that % to any point $\bfx_0\in C_{1,j}$ and the polynomial $F(\bfx)=\|\bfx\|-j$. Then $(\nabla F)(\bfx_0) = 2\bfx_0\not\equiv \bfo \, (\mod p)$ whenever $\|\bfx_0\|=j$. It follows that
%\[
%\#\big\{\bfx\,(\mod p^r):\, \bfx\equiv \bfx_0\,(\mod p),\, \|\bfx\|\equiv j\,(\mod p^r)\big\}\leq p^{r-1}.
%\]
%This is to say,
\[
|C_{r,j}| = p^{r-1}|C_{1,j}|=p^r(1+1/p).
\]
%Combining Lemma \ref{lem_staborb}, we have $|C_{r,j}|=$
%Applying Lemma \ref{lem_hensel}, we obtain that
%\[
%\#\left\{(z_1,z_2)\,(\mod p^{r-1}):\, (a_0+pz_1)^2+(b_0+pz_2)^2 \equiv 1\,(\mod p^r)\right\} = p^{r-1}.
%\]
%It follows that
%\[
%|G_r|=p^{r-1}|G_1|=p^{r}(1+1/p). %,\qquad |G_r^I|= p^{r-1}.
%\]
%Similarly by Hensel's lemma, we have

%We have already proved in the above that, for $j\in U_r$, given any $\bfx,\bfy\in C_{r,j}$, there is a unique $\theta\in G_r$ such that $\theta x\equiv y \,(\mod q)$.

Next, consider the circle $C_{r,0}$ $(r\geq 2)$. Since $C_{1,0}=\{\bfo\}$, we can write
\[
C_{r,0} = \{p\bfy:\, \bfy\in C'_{r,0}\}, \quad C'_{r,0}=\big\{\bfy\,(\mod p^{r-1}):\, \|p\bfy\|\equiv 0\,(\mod p^r)\big\}.
\]
When $r=2$, one sees that $C_{2,0} = \{p\bfy:\, \bfy \in (\bZ/pZ)^2\}$ and $|C_{2,0}|=p^2$. When $r\geq 3$, we further write $\bfy = \bfz+p^{r-2}\bfw$ with $\bfz\in (\bZ/p^{r-2}\bZ)^2$ and $\bfw\in (\bZ/p\bZ)^2$. The condition $\|p\bfy\|\equiv \bfo\,(\mod p^r)$ holds if and only if $\|\bfz+p^{r-2}\bfw\|=\|\bfy\|\equiv \bfo\,(\mod p^{r-2})$, if and only if $\|\bfz\|\equiv \bfo\,(\mod p^{r-2})$.
It follows that
\[
C_{r,0} = \big\{p(\bfz+p^{r-2}\bfw):\, \bfz\in C_{r-2,0},\, \bfw\in (\bZ/p\bZ)^2\big\},
\]
and $|C_{r,0}| = p^2 |C_{r-2,0}|$.

By induction, we conclude that
\[
C_{r,0} = \left\{p^{\lceil r/2 \rceil} \bfw:\, \bfw \in \bZ/p^{\lfloor r/2 \rfloor}\bZ\right\}
\]
and $|C_{r,0}| = p^{2\lfloor r/2 \rfloor}$.

Now, let us consider the case $j=p^v \tilde{j}$, where $1\leq v\leq r-1$ and $\tilde{j}\in (\bZ/p^{r-v}\bZ)^\ast$. For any $\bfx\in C_{r,j}$, it satisfies that $\|\bfx\|\equiv 0\,(\mod p)$. Recalling that $C_{1,0}=\{\bfo\}$, one has $\bfx\equiv \bfo\,(\mod p)$. So
\[
C_{r,j} = \{p\bfy:\, \bfy\in C'_{r,j}\}, \quad C'_{r,j}=\big\{\bfy\,(\mod p^{r-1}):\, \|p\bfy\|\equiv p^{v} \tilde{j}\,(\mod p^r)\big\}.
\]
When $v=1$, it is easy to see that $C'_{r,j}=\emptyset$. When $v\geq 2$, which means $r\geq 3$, similar arguments as previous show that
%we have
%\begin{eqnarray*}
%C'_{r,j} &&= \{\bfy\,(\mod p^{r-1}):\, \|\bfy\|\equiv p^{v-2} \tilde{j}\,(\mod p^{r-2})\}\\
%&&= \{\bfz+p^{r-2}\bfw:\, \bfz\in \bZ/p^{r-2}\bZ,\,\bfw\in \bZ/p\bZ:\, \|\bfz\|\equiv p^{v-2} \tilde{j}\,(\mod p^{r-2})\}.
%\end{eqnarray*}
%It follows that
\[
C_{r,p^v\tilde{j}} = \big\{p(\bfz+p^{r-2}\bfw):\, \bfz\in C_{r-2,p^{v-2}\tilde{j}},\, \bfw\in (\bZ/p\bZ)^2\big\},
\]
and $|C_{r,p^v\tilde{j}}| = p^2 \cdot |C_{r-2,p^{v-2}\tilde{j}}|$. By induction, it follows that
\[
|C_{r,p^v\tilde{j}}| =
\begin{cases}
p^r(1+1/p),\quad &\text{if } v \text{ is even},\\
0,\quad &\text{if } v \text{ is odd}
\end{cases}
\]
when $1\leq v \leq r-1$.  Indeed,
\begin{equation} \label{eq_concreteform}
C_{r,p^{2u}\tilde{j}} = \left\{p^u \bfz+p^{r-u}\bfw:\,\bfz\in C_{r-2u,\tilde{j}},\, \bfw \in (\bZ/p^u\bZ)^2\right\}
\end{equation}
when $r\geq 3$ and $1\leq u\leq (r-1)/2$.
\end{proof}

\paragraph{Remark:} When $p\equiv 3\,(\mod 4)$, we have the disjoint union
\[
(\bZ/p^r\bZ)^2 = \left(\bigcup\limits_{u=0}^{\lfloor (r-1)/2\rfloor} \bigcup\limits_{\tilde{j}\in (\bZ/p^{r-2u}\bZ)^\ast} C_{r,p^{2u}\tilde{j}}\right)\,\cup \, C_{r,0},
\]
whose cardinalities give
\begin{align*}
p^{2r} &= \sum\limits_{u=0}^{\lfloor (r-1)/2\rfloor} p^{r}(1+1/p)\cdot p^{r-2u}(1-1/p) + p^{2\lfloor r/2 \rfloor}\\
& = \sum\limits_{u=0}^{\lfloor (r-1)/2\rfloor} (p^{2r-2u}-p^{2r-2u-2}) + p^{2\lfloor r/2 \rfloor}.
\end{align*}

\begin{lemma}\label{energy1} \color{black}
{\color{black}Let $p\equiv 3\,(\mod 4)$ and $r\geq 1$.} Let $j\in (\bZ/p^r\bZ)^\ast$. For any $\bfo\neq \bfz\in (\bZ/p^r\bZ)^2$, we have
\[
\#\big\{(\bfx,\bfy)\in C_{r,j}^2:\, \bfx-\bfy\equiv \bfz\,(\mod p^r)\big\}\, \leq \,
2p^{r-1}.
\]
\end{lemma}

%Remark: The set on the left-hand side is empty if $\ord_p(j)>2v_\bfz$.

\begin{proof}
%First consider $t\neq 0$. Let $t=p^v\tilde{j}$ with $v=\ord_p(t)$ and $\tilde{j}\in (\bZ/p^{r-v}\bZ)^\ast$. Here $0\leq v<r$ and $v=2u$ is even. Then
%\[
%C_{r,j}=\left\{p^u \tilde{\bfx}+p^{r-u}\bfw:\,\tilde{\bfx}\in C_{r-2u,\tilde{j}},\, \bfw \in \bZ/p^u\bZ\right\}.
%\]
%Its easy to see that
%\[
%\#\big\{(\bfx,\bfy)\in C_{r,j}^2:\, \bfx-\bfy\equiv \bfz\,(\mod p^r)\big\}=0
%\]
%when $v_\bfz<p$. When $v_\bfz\geq p$, write $\bfz=p^u\tilde{\bfz}$. Then
%\[
%\#\big\{(\bfx,\bfy)\in C_{r,j}^2:\, \bfx-\bfy\equiv \bfz\,(\mod p^r)\big\} = \#\big\{(\bfx,\bfy)\in C_{r,j}^2:\, \tilde{\bfx}-\tilde{\bfy}\equiv \tilde{\bfz}\,(\mod p^{r-u})\big\}.
%\]

The set on the left-hand side involves system of congruences
\begin{equation} \label{eq_sys}
\begin{cases}
x_1^2+x_2^2 \equiv j,\\
y_1^2+y_2^2 \equiv j,\\
x_1-y_1\equiv z_1,\\
x_2-y_2\equiv z_2.
\end{cases}
\end{equation}
Considering the system \eqref{eq_sys} modulo $p$, one sees that
\[
2y_1z_1+2y_2z_2+z_1^2+z_2^2\equiv 0\,(\mod p).
\]

First, we consider the case $\bfz\not\equiv \bfo\,(\mod p)$. Assume without loss of generality that $z_1\not\equiv 0\, (\mod p)$. Inserting $y_1\equiv -(2z_1)^{-1}(2y_2z_2+z_1^2+z_2^2)\, (\mod p)$ into $y_1^2+y_2^2\equiv j\,(\mod p)$, one obtains
\[
z_1^{-2}(z_1^2+z_2^2)y_2^2+z_1^{-2}z_2(z_1^2+z_2^2)y_2+4^{-1}z_1^{-2}(z_1^2+z_2^2)^2\equiv j\quad (\mod p).
\]
When $p\equiv 3\,(\mod 4)$, $\bfz\not\equiv \bfo\,(\mod p)$ implies that $z_1^2+z_2^2\not\equiv 0\, (\mod p)$. The left-hand side of above congruence has degree $2$. So there are at most two solutions in $y_2$. Moreover, each choice of $y_2$ exactly determines the choices of $x_1$, $y_1$ and $x_2$.

%When $\bfz\equiv \bfo\,(\mod p)$, the number of the system \eqref{eq_sys} modulo $p$ is $|C_{1,j}|$, which equals $p+1$ when $j\not\equiv 0\,(\mod p)$ and equals $1$ otherwise.

Now we apply Hensel's lemma. The Jacobian matrix, in $(x_1,y_1,x_2,y_2)$, is given by
\[
\begin{bmatrix}
2x_1 & 0 & 2x_2 & 0\\
0 &2y_1 & 0 & 2y_2\\
1 & -1 & 0 & 0\\
0 & 0 &1 &-1
\end{bmatrix}.
\]
By elementary operations, we obtain
\[
\begin{bmatrix}
1 & 0 & -1 & 0\\
0 &1 & 0 & -1\\
0 & 0 & x_1 & x_2\\
0 & 0 &y_1 &y_2
\end{bmatrix}.
\]
Since $j\not\equiv 0\,(\mod p)$, it satisfies that $\bfx,\bfy\not\equiv \bfo\,(\mod p)$ and the rank modulo $p$ is at least $3$. Then the conclusion follows from Lemma \ref{lem_Hensel}.

Second, let us deal with the case $\bfz\equiv \bfo$ $(\mod p)$. Assume that $\bfz=p^k\tilde{\bfz}$ with $k=v_\bfz$ and $\tilde{\bfz}\in (\bZ/p^{r-k}\bZ)^2$. Here $1\leq k<r$, since $\bfz\neq \bfo$. %{\color{black}And $\|\tilde{\bfz}\|\not\equiv 0\,(\mod p)$, since $p\equiv 3\,(\mod 4)$.}
Note that, for any solution $(\bfx,\bfy)$ to the system \eqref{eq_sys} modulo $p^r$, it also satisfies \eqref{eq_sys} modulo $p^k$ or $p^{k+1}$. For the former situation, one deduces that $\bfx\equiv \bfy\,(\mod p^k)$ since $\bfz\equiv \bfo\,(\mod p^k)$. For the latter situation, similar arguments as previous shows that
\[
(y_1+p^k\tilde{z_1})^2+(y_2+p^k\tilde{z_2})^2 \equiv j\equiv y_1^2+y_2^2,\,(\mod p^{k+1}),
\]
i.e., $2y_1\tilde{z_1}+2y_2\tilde{z_2}\equiv 0\,(\mod p)$. Noting that {\color{black}$v_{\tilde{\bfz}}=0$, we assume without loss of generality that $\tilde{z_1}\not\equiv 0\,(\mod p)$.} Then
\[
x_1\equiv y_1\equiv y_2\tilde{z_2}\tilde{z_1}^{-1}\equiv x_2\tilde{z_2}\tilde{z_1}^{-1}, \quad (\mod p).
\]
Now any solution $(\bfx,\bfy)$ to the system \eqref{eq_sys} modulo $p^r$ satisfies that
\[
\begin{cases}
x_1^2+x_2^2\equiv j,\\
x_1\equiv x_2\tilde{z_2}\tilde{z_1}^{-1},\\
x_1\equiv y_1,\\
x_2\equiv y_2,
\end{cases}\quad (\mod p).
\]
It is not hard to see that there are at most $2$ solutions modulo $p$. Now we lift these solutions modulo $p$ to solutions to \eqref{eq_sys} modulo $p^r$. By applying Hensel's lemma with the rank of Jacobian matrix being $3$, we conclude that
\[
\#\big\{(\bfx,\bfy)\in C_{r,j}^2:\, \bfx-\bfy\equiv \bfz\,(\mod p^r)\big\} \leq 2 p^{r-1}.
\]
\end{proof}

\begin{lemma}\label{energy2}
{\color{black}Let $p\equiv 3\,(\mod 4)$ and $r\geq 1$.} Let $\bfm,\bfz\in (\bZ/p^r\bZ)^2$ be such that $\bfm,\bfz\neq \bfo$. %write $\bfm=p^{v_\bfm}\tilde{\bfm}$ for some $\tilde{\bfm}\in (\bZ/p^{r-v_\bfm}\bZ)^2$ with $v_{\tilde{\bfm}}=0$. %$v\ge \frac{r+1}{2}$ and $||\tilde{m}||\in U_p$. The number of pairs $(\mathbf{x}, \mathbf{y})\in V_{j, \mathbf{m}}\times V_{j, \mathbf{m}}$ such that $\mathbf{x}-\mathbf{y}\equiv \mathbf{z}\,\mod(p^r)$ is at most $2p^{r-v-1}$.
%Let $F(x)=x_1^2+x_2^2-j$ with $j\in U_p$.
%Let $j\in \bZ$ and $\bfz\in \bZ^2$. Suppose that $\bfz\not\equiv 0\,(\mod q)$. Then
Then
\[
\#\big\{(\bfx,\bfy)\in (\textsf{orb}_r(\bfm))^2:\, \bfx-\bfy\equiv \bfz\,(\mod p^r)\big\}\, \leq\, 2p^{r-v_\bfm-1}.
\]
\end{lemma}

\begin{proof}
Write $v=v_{\bfm}$ and $\bfm=p^{v}\tilde{\bfm}$. Here $0\leq v\leq r-1$, $\tilde{\bfm}\in (\bZ/p^{r-v_\bfm}\bZ)^2$ and $v_{\tilde{\bfm}}=0$. Write $\|\tilde{\bfm}\|=\tilde{j}$ for simplicity. Then $\tilde{j}\in (\bZ/p^{r-v}\bZ)^\ast$. By Lemma \ref{lem_staborb}, we have $\textsf{orb}(\bfm)= p^v C_{r-v,\tilde{j}}$. Write $\bfx=p^v \tilde{\bfx}$ and $\bfy=p^v \tilde{\bfy}$ with $\tilde{\bfx},\tilde{\bfy}\in C_{r-v,\tilde{j}}$. When $\bfz\not\equiv \bfo (\mod p^v)$, it is obvious that
\[
\#\big\{(\bfx,\bfy)\in (\textsf{orb}_r(\bfm))^2:\, \bfx-\bfy\equiv \bfz\,(\mod p^r)\big\}=0.
\]
When $\bfz\equiv \bfo (\mod p^v)$, we write $\bfz=p^v\tilde{\bfz}$ for some $\bfo\neq \tilde{\bfz}\in (\bZ/p^{r-v}\bZ)^2$. It follows that
\[
\#\big\{(\bfx,\bfy)\in (\textsf{orb}_r(\bfm))^2:\, \bfx-\bfy\equiv \bfz\,(\mod p^r)\big\} = \#\big\{(\tilde{\bfx},\tilde{\bfy})\in (C_{\tilde{j},r-v})^2:\, \tilde{\bfx}-\tilde{\bfy}\equiv \tilde{\bfz}\,(\mod p^{r-v})\big\}.
\]
By Lemma \ref{energy1}, the above quantity can be bounded by $2p^{r-v-1}$.
\end{proof}

With the same argument, we obtain similar results for the case $p\equiv 1\, (\mod 4)$.

\begin{lemma} \label{lem_staborb1}
{\color{black}Let $p\equiv 1\,(\mod 4)$ and $r\geq 1$.} For any $0 \ne \bfx\in (\bZ/p^r\bZ)^2$, let $\bfx=p^{v_\bfx}\tilde{\bfx}$ for some $\tilde{\bfx}\in (\bZ/p^{r-v_\bfx}\bZ)^2$ with $v_{\tilde{\bfx}}=0$. Then
\[
|\textsf{stab}_r(\bfx)| = p^{v_\bfx},\qquad |\textsf{orb}_r(\bfx)| =  p^{r-v_\bfx}(1-1/p).
\]
And $\textsf{orb}_r(\bfx) =  p^{v_\bfx}\textsf{orb}_{r-v_\bfx}(\tilde{\bfx})$.
\end{lemma}
\begin{lemma}\label{circle-size1}
{\color{black}Let $p\equiv 1\,(\mod 4)$ and $r\geq 1$.} Let $j$ be an integer with $\text{ord}_p(j)=v$. % , $0\leq v\leq r$, and $\tilde{j}$ be a unit in $\mathbb{Z}/p^{r-v}\bZ$.
Then
\[
|C_{r,j}| =
\begin{cases}
(v+1)p^r(1-1/p), \quad &\text{if }0\leq v <r,\\
(rp+p-r)p^{r-1},\quad &\text{if  $v=r$}
\end{cases}
\]

\end{lemma}

\begin{lemma}\label{energy3}\color{black}
{\color{black}Let $p\equiv 1\,(\mod 4)$ and $r\geq 1$.} Let $j \in (\bZ /p^r\bZ)^\ast$. For any $\bfo \ne \bfz \in (\bZ /p^r \bZ)^2$, we have
\[
\#\big\{(\bfx,\bfy)\in C_{r,j}^2:\, \bfx-\bfy\equiv \bfz\,(\mod p^r)\big\}\, \leq \,
2p^{r-1}.
\]
\end{lemma}

\begin{lemma}\label{energy4}
    For $p\equiv 1\, (\mod 4)$ be a prime, $r \ge 1$ be an integer. Let $\bfm =p^{v_{\bfm}}\tilde{\bfm} \in (\bZ /p^r \bZ)^2$ with $\|\tilde{\bfm}\|\not\equiv 0\,(\mod p)$. For any $ \bfz \in (\bZ /p^r\bZ)^2$, we have
    \[ \# \{ (\bfx ,\bfy )\in (\textsf{orb}_r(\bfm))^2 \colon \bfx -\bfy \equiv \bfz (\mod p^r) \} \ll  p^{r-v_{\bfm}-1}.
     \]
\end{lemma}

\subsection{Proof of extension theorems}
Theorems \ref{resitriction} and \ref{restriction3} are proved in the same way. Let $V$ be a variety with $V\subseteq (\bZ/p^r\bZ)^2$, which can be $C_{r,j}$ or $\textsf{orb}_r(\mathbf{m})$, on which the normalized surface measure is denoted by $d\sigma$. For simplicity, we denote $q=p^r$ in this subsection.

We have
\begin{align*}
&\sum_{\bfm\,(\mod q)}|(fd\sigma)^\vee(\bfm)|^4=\sum_{\bfm\,(\mod q)}\left\vert \frac{1}{|V|}\sum_{\bfx\in V}f(\bfx)e_q(\bfm\cdot \bfx)\right\vert^4\\
&\qquad=\frac{q^{2}}{|V|^4}\sum_{\xi,\xi',\eta,\eta'\in V\atop \xi-\eta\equiv \xi'-\eta'\,(\mod q)}f(\xi)f(\xi')\overline{f(\eta)f(\eta')}
=\frac{q^{2}}{|V|^4}\sum_{\zeta\, (\mod q)} \left|\sum_{\xi,\eta\in V\atop \xi-\eta\equiv \zeta\,(\mod q)}f(\xi)\overline{f(\eta)}\right|^2.
\end{align*}

%{\color{teal}When $p\equiv -1(\mod 4)$, $q=p^r$ with $r$ odd and $V=C_0$, the right-hand side is
%\[
%\sum\limits_{\bfw\,(\mod p^{\frac{r-1}{2}})} \left|\sum_{\bfw'-\bfw''\equiv \bfw\,(\mod p^{\frac{r-1}{2}})}f(p^{\frac{r+1}%{2}}\bfw')\overline{f(p^{\frac{r+1}{2}}\bfw')} \right|^2.
%\]
%}

For $\zeta\equiv \mathbf{0}\,(\mod q)$, we have
\[\Big|\sum_{\xi,\eta\in V\atop \xi\equiv \eta\, (\mod q)}f(\xi)\overline{f(\eta)}\Big|^2 = \Big(\sum_{\xi\in V}|f(\xi)|^2\Big)^2.\]

For $\zeta\not\equiv \mathbf{0}\,(\mod p^r)$, the Cauchy-Schwarz inequality implies
\begin{align*}
\sum_{\zeta\not\equiv \mathbf{0}\,(\mod q)} \Big| \sum_{\xi,\eta\in V\atop \xi-\eta\equiv\zeta\,(\mod q)}f(\xi)\overline{f(\eta)}\Big|^2
\leq \sum_{\zeta\not\equiv \mathbf{0}\,(\mod q)} \Big( \sum_{\xi,\eta\in V\atop \xi-\eta\equiv \zeta\,(\mod q)}1^2\Big)\Big( \sum_{\xi,\eta\in V\atop\xi-\eta\equiv \zeta\,(\mod q)}|f(\xi)|^2|f(\eta)|^2\Big).\\
\end{align*}
Assuming that $\sum\nolimits_{\xi,\eta\in V\,\xi-\eta\equiv \zeta\,(\mod q)}1 \ll U$ {\color{black}for all $\zeta\not\equiv \bfo\,(\mod q)$}, then
\[
\sum_{\zeta\not\equiv \mathbf{0}\,(\mod q)} \Big| \sum_{\xi,\eta\in V\atop \xi-\eta\equiv\zeta\,(\mod q)}f(\xi)\overline{f(\eta)}\Big|^2 \ll U\, \sum_{\zeta \,(\mod p^r)} \Big( \sum_{\xi,\eta\in V\atop\xi-\eta\equiv \zeta\,(\mod q)}|f(\xi)|^2|f(\eta)|^2\Big)=U\,\Big(\sum_{\xi\in V}|f(\xi)|^2\Big)^2.
\]
It follows that
\[\left(\sum_{m\, (\mod q)}|(fd\sigma_{r,j})^\vee(m)|^4\right)^{1/2}\ll \frac{q {\color{black}(1+U)}^{1/2}}{|V|^2}\sum_{x\in V}|f(x)|^2.\]
On the one hand, to get the bound $U$, we use Lemmas \ref{energy1} and \ref{energy2}. On the other hand, Lemma \ref{circle-size} gives estimates on the size of $V$. Hence, the theorems follow.

%\section{Proof of Theorem \ref{restriction3}}

%Note that
%\[
%C_{r,0} = \left\{p^{\lceil r/2 \rceil} \bfw:\, \bfw \in \bZ/p^{\lfloor r/2 \rfloor}\bZ\right\}.
%\]
%Let $\bfx=p^{v} \tilde{\bfx}$ with $\lceil r/2 \rceil\leq v\leq r-1$, $\tilde{\bfx}\in (\bZ/p^{v-r}\bZ)^2$ and $v_{\tilde{\bfx}}=0$. Then
%\[
%|\textsf{stab}_{\bfx}|  = p^{v}, \quad |\textsf{orb}(\bfx)| = p^{r-v}(1+1/p).
%\]
%For $\bfx=\bfo$, $|\textsf{stab}_{\bfx}|  = p^r(1+1/p), \,|\textsf{orb}(\bfx)| = 1$.

\section{Proof of Theorems \ref{un-conditional} and \ref{conditional}}
These two theorems are proved by the same argument.

By Lemma \ref{circle-size}, it is clear that that the number of pairs $(\mathbf{x}, \mathbf{y})\in E\times E$ such that $||\mathbf{x}-\mathbf{y}||\not\in (\mathbb{Z}/p^r\mathbb{Z})^*$ is much smaller than $|E|^2$, since $|E|\gg p^{2r-1}$.
By the Cauchy-Schwarz inequality, we have
\begin{align*}
|E|^4 &\ll \Big(\sum\limits_{\bfx,\bfy\in E\atop \|\bfx-\bfy\|\in (\bZ/p^r\bZ)^\ast}1\Big)^2=\big(\sum\limits_{j\in \Delta_{2,r}(E)\cap (\bZ/p^r\bZ)^\ast} \sum\limits_{\bfx\in E}\sum\limits_{\bfy\in E\atop \|\bfx-\bfy\|\equiv j\,(\mod p^r)}1\Big)^2\\
&\leq |\Delta_{2,r}(E)|\cdot |E|\cdot  \sum\limits_{j\not\equiv 0\, (\mod p)} \sum\limits_{\bfx\in E}\Big(\sum\limits_{\bfy\in E\atop \|\bfx-\bfy\|\equiv j\,(\mod p^r)}1\Big)^2.
\end{align*}
So
\[|\Delta_{2, r}(E)|\gg \frac{|E|^3}{\mathcal{N}},\]
where
\[
\cN := \sum\limits_{j\not\equiv 0\, (\mod p)} \sum\limits_{\bfx\in E}\Big(\sum\limits_{\bfy\in E\atop \|\bfx-\bfy\|\equiv j\,(\mod p^r)}1\Big)^2= \sum\limits_{x\in E}\, \sum\limits_{j\not\equiv 0\,(\mod p)}\,\Big( \sum_{\mathbf{z}\in C_{r,j}}1_E(\mathbf{x}-\mathbf{z}) \Big)^2.
\]

For any given $\mathbf{x}\in E$, we have
\begin{align*}
 &\sum_{j\not\equiv 0\,(\mod p)} \Big( \sum_{\mathbf{z}\in C_{r,j}}1_E(\mathbf{x}-\mathbf{z}) \Big)^2=\sum_{j\not\equiv 0\,(\mod p)}\sum_{\mathbf{y}, \mathbf{z}\in C_{r,j}}1_E(\mathbf{x}-\mathbf{y})1_E(\mathbf{x}-\mathbf{z})\\
&\le \sum_{j\not\equiv 0\,(\mod p)}\sum_{\mathbf{z}\in C_{r,j}}1_E(\mathbf{x}-\mathbf{z})\sum_{\theta\in G_r}1_E(\mathbf{x}-\theta \mathbf{z})\le\sum_{{\color{black}\bfz\, (\mod p^r)}}\sum_{\theta\in G_r}1_E(\mathbf{x}-\mathbf{z})1_E(\mathbf{x}-\theta\mathbf{z}).
\end{align*}
%Here the terms with $\|\bfz\|\not\equiv 0\,(\mod p^r)$ but $\|\bfz\|\equiv 0\,(\mod p)$ can also be involved, since the summands are non-negative.
In the next step, we write
\begin{align*}
& \sum_{\mathbf{z}\,(\mod p^r)} \sum_{\theta\in G_r}1_E(\mathbf{x}-\mathbf{z})1_E(\mathbf{x}-\theta \mathbf{z})\\
&\quad =\sum_{\mathbf{z}\,(\mod p^r)} \sum_{\theta\in G_r}\sum\limits_{\bfm\,(\mod p^r)}\widehat {1_E}(\mathbf{m})e_{p^r}(\bfm\cdot (\bfx-\bfz))\overline{\sum\limits_{\bfm'\,(\mod p^r)}\widehat{1_E}(\bfm')e_{p^r}(\mathbf{m}'\cdot (\bfx-\theta \bfz))}\\
&\quad =\sum_{\theta\in G_r}\sum\limits_{\bfm\,(\mod p^r)}\sum\limits_{\bfm'\,(\mod p^r)}\widehat {1_E}(\mathbf{m})\overline{\widehat{1_E}(\bfm')}e_{p^r}\big((\bfm-\bfm')\cdot \bfx\big)\sum_{\mathbf{z}\,(\mod p^r)}e_{p^r}\big(-(\bfm\cdot \bfz-\bfm'\cdot (\theta \bfz)\big).
\end{align*}
Note that
\[
\sum_{\mathbf{z}\,(\mod p^r)}e_{p^r}\big(-(\bfm\cdot \bfz-\bfm'\cdot (\theta \bfz)\big) = \sum_{\mathbf{z}\,(\mod p^r)}e_{p^r}(-(\bfm-\theta^{-1}\bfm')\cdot  \bfz\big) =
\begin{cases}
p^{2r},\quad &\text{if }\bfm'=\theta \bfm,\\
0,&\text{otherwise}.
\end{cases}
\]
One deduces that
{\color{black}
\begin{align*}
\cN\ll & \sum\limits_{x\in E}\sum_{\mathbf{z}\,(\mod p^r)} \sum_{\theta\in G_r}1_E(\mathbf{x}-\mathbf{z})1_E(\mathbf{x}-\theta \mathbf{z})\\
=& p^{2r}\sum\limits_{x\in E}\sum_{\mathbf{m}\, (\mod p^r)}\sum_{\theta\in G_r} \widehat{1_E}(\mathbf{m})\overline{\widehat{1_E}(\theta\mathbf{m})} e_{p^r}((\bfm-\theta\mathbf{m})\cdot \mathbf{x})\\
\ll&  \sum\limits_{x\in E}\left(\frac{|E|^2}{p^r}+\mathcal{E}(\mathbf{x})\right) = \frac{|E|^3}{p^r}+\sum\limits_{x\in E}\mathcal{E}(\mathbf{x}),
\end{align*}
}
where
%\[\mathcal{E}_1(\mathbf{x}):=p^r^2\sum_{||\mathbf{m}||\in (\mathbb{Z}/p^r\mathbb{Z})^*}\sum_{\theta\in SO_2(\mathbb{Z}/p^r\mathbb{Z})} e_p^r(-\mathbf{m}\cdot \mathbf{x})e_p^r((\theta\mathbf{m})\cdot \mathbf{x})\hat A(-\mathbf{m})\overline{\hat A(-\theta\mathbf{m})},\]
%and
\[
\mathcal{E}(\mathbf{x}):=p^{2r}\sum\limits_{v=0}^{r-1} \cE_{v}(\bfx),\qquad \cE_{v}(\bfx)=\sum_{\bfm\,(\mod p^r)\atop v_\bfm=v}\sum_{\theta\in G_r} \widehat{1_E}(\mathbf{m})\overline{\widehat{1_E}(\theta\mathbf{m})} e_{p^r}((\bfm-\theta\mathbf{m})\cdot \mathbf{x}).
\]
Here, the first term $|E|^2/p^{r}$ comes from the summands with $\bfm\equiv \bfo\,(\mod p^r)$, in view of $|\widehat{1_E}(\bfo)|=|E|/p^{2r}$ and $|G_r|\ll p^r$.

Write $\bfm=p^v\tilde{\bfm}$ and $\tilde{\theta}\equiv \theta \,(\mod p^{r-v})$. Then $\tilde{\theta}\tilde{\bfm}\equiv \theta{\color{black}\tilde{\bfm}}\,(\mod p^{r-v})$. Noting that $|G_r/G_{r-v}|=p^v$, we have
\begin{align}
\cE_{v}(\bfx)&=p^v\sum\limits_{\tilde{\bfm}\, (\mod p^{r-v})\atop v_{\tilde{\bfm}=0}}\sum_{\tilde{\theta}\in G_{r-v}} \widehat{1_E}(p^v\tilde{\mathbf{m}})\overline{\widehat{1_E}(p^v\tilde{\theta}\tilde{\mathbf{m}})}e_{p^{r-v}}\big((\tilde{\bfm}-\tilde{\theta}\tilde{\mathbf{m}})\cdot \mathbf{x})\big)\label{eq_xandxtilde}\\
&=p^v\sum\limits_{\tilde{j}\in (\bZ/p^{r-v}\bZ)^\ast\,}\sum\limits_{\tilde{\bfm}\in C_{r-v,\tilde{j}}} \sum\limits_{\bfm\in C_{r-v,\tilde{j}}} \widehat{1_E}(p^v\tilde{\mathbf{m}})\overline{\widehat{1_E}(p^v \mathbf{m})}e_{p^{r-v}}\big((\tilde{\bfm}-\bfm)\cdot \mathbf{x})\big)\nonumber\\
&= p^v\sum\limits_{\tilde{j}\in (\bZ/p^{r-v}\bZ)^\ast\,}\left|\sum\limits_{\tilde{m}\in C_{r-v,\tilde{j}}}\widehat{1_E}(p^v\tilde{\bfm})e_{p^{r-v}}(\tilde{\bfm}\cdot \bfx)\right|^2. \nonumber
\end{align}
{\color{black}When $1\leq v\leq r-1$,} we rewrite
\begin{align*}
\widehat{1_E}(p^v\tilde{\bfm}) &= \frac{1}{p^{2r}} \sum\limits_{\bfy\,(\mod p^r)} 1_E(\bfy) e_{p^r}(-p^v \tilde{\bfm}\cdot \bfy) \\
&= {\color{black}\frac{1}{p^{2(r-v)}}} \sum\limits_{\bfy_1\,(\mod p^{r-v})} {\color{black}\frac{1}{p^{2v}}}\sum\limits_{\bfy_2\,(\mod p^{v})} 1_E(\bfy_1+p^{r-v}\bfy_2)e_{p^{r-v}}(-\tilde{\bfm}\cdot\bfy_1)
:= \widehat{g_{E,r-v}}(\tilde{\bfm}).
\end{align*}
Here, the Fourier transformation on the left-hand side is over $(\bZ/p^r\bZ)^2$, while that on the right-hand side is over $(\bZ/p^{r-v}\bZ)^2$. And, for $\bfy_1\,(\mod p^{r-v})$,
\[
g_{E,r-v}(\bfy_1)  =\frac{1}{p^{2v}} \sum\limits_{\bfy_2\,(\mod p^v)} 1_E(\bfy_1+p^{r-v}\bfy_2) = \frac{1}{p^{2v}}\#\{\bfy\in E:\, \bfy\equiv \bfy_1\,(\mod p^{r-v})\}.
\]
We note that for all $\bfy_1\,(\mod p^{r-v})$, one has $g_{E,r-v}(\bfy_1)\le 1$. {\color{black}When $v=0$, the above notation gives $g_{E,r}=1_E$. Now , for all $0\leq v\leq r-1$, we have
\[
\cE_{v}(\bfx)=p^v\sum\limits_{\tilde{j}\in (\bZ/p^{r-v}\bZ)^\ast\,}\left|\sum\limits_{\tilde{m}\in C_{r-v,\tilde{j}}}\widehat{g_{E,r-v}}(\tilde{\bfm})e_{p^{r-v}}(\tilde{\bfm}\cdot \bfx)\right|^2
\]}
Write $f=\widehat{g_{E,r-v}}\big|_{C_{r-v,\tilde{j}}}$. {\color{black}Denote
\[
(fd\sigma_{r-v,\tilde{j}})^\vee (x) = \frac{1}{|C_{r-v,\tilde{j}}|}\sum\limits_{\bfy\in C_{r-v,\tilde{j}}}f(y)e_{p^{r-v}}(\bfy\cdot \bfx).
\]
}It follows that
\[
\cE_{v}(\bfx)\ll p^{2r-v}\sum\limits_{\tilde{j}\in (\bZ/p^{r-v}\bZ)^\ast\,}\left|(f d\sigma_{\color{black}r-v,\tilde{j}})^\vee(\bfx)\right|^2.
\]
Moreover, {\color{black}from \eqref{eq_xandxtilde}} we have $\cE_{v}(\bfx)=\cE_{v}(\tilde{\bfx})$ if $\tilde{\bfx}\equiv \bfx\,(\mod p^{r-v})$. Now we write
\[\color{black}
E(\mod p^{\gamma}):=\{\bfx\,(\mod p^\gamma):\, \bfx\in E\}= \{\tilde{\bfx}\,(\mod p^\gamma):\, \tilde{\bfx} \equiv \bfx\,(\mod p^\gamma)\text{ for some }\bfx\in E\}
\]
for $1\leq \gamma\leq r$. Then, by Cauchy-Schwartz inequality, Theorem \ref{resitriction} {\color{black}and Parseval's identity,} one deduces that
\begin{align*}
p^{2r}\sum\limits_{\bfx\in E}\cE_{v}(\bfx) &\ll p^{4r-v} \sum\limits_{\tilde{\bfx}\in E(\mod p^{r-v})} p^{2v} g_{E,r-v} (\tilde{\bfx}) \sum\limits_{\tilde{j}\in (\bZ/p^{r-v}\bZ)^\ast} |(f d\sigma_{r-v,\tilde{j}})^\vee(\tilde{\bfx})|^2\\
&\ll p^{4r+v} \sum\limits_{\tilde{j}\in (\bZ/p^{r-v}\bZ)^\ast}\left(\sum\limits_{\tilde{\bfx}\in E(\mod p^{r-v})} |g_{E,r-v}(\tilde{x})|^2\right)^{1/2}\left(\sum\limits_{\tilde{\bfx}\, (\mod p^{r-v})} |( f d\sigma_{r-v,\tilde{j}})^\vee(\tilde{\bfx})|^4\right)^{1/2}\\
&\ll p^{4r+v}\left(\sum\limits_{\tilde{\bfx}\in E(\mod p^{r-v})} |g_{E,r-v}(\tilde{\bfx})|^2\right)^{1/2}\cdot p^{-\frac{r-v+1}{2}}\sum\limits_{\tilde{j}\in (\bZ/p^{r-v}\bZ)^\ast}\sum\limits_{\tilde{\bfy}\in C_{r-v,\tilde{j}}} |\widehat{g_{E,r-v}} (\tilde{\bfy})|^2\\
&{\color{black}\ll p^{\frac{7r+3v-1}{2}}\left(\sum\limits_{\tilde{\bfx}\, (\mod p^{r-v})} |g_{E,r-v}(\tilde{\bfx})|^2\right)^{1/2}\cdot\sum\limits_{\tilde{\bfy}\,(\mod p^{r-v})} |\widehat{g_{E,r-v}} (\tilde{\bfy})|^2 }\\
&\ll  p^{\frac{3r+7v-1}{2}} \left(\sum\limits_{\tilde{\bfx}\,(\mod p^{r-v})}|g_{E,r-v}(\tilde{\bfx})|^2\right)^{3/2}.
\end{align*}
The last sum can be bounded trivially by
\[\sum\limits_{\tilde{\bfx}\,(\mod p^{r-v})}|g_{E,r-v}(\tilde{\bfx})|^2\le \sum\limits_{\tilde{\bfx}\,(\mod p^{r-v})}|g_{E,r-v}(\tilde{\bfx})|\le \frac{|E|}{p^{2v}}.\]
In other words,
\[p^{2r}\sum\limits_{\bfx\in E}\cE_{v}(\bfx)\ll p^{\frac{3r+v-1}{2}}|E|^{3/2}.\]
With this bound in hand, one has
\[\mathcal{N}\ll \frac{|E|^3}{p^r}+p^{\frac{3r-1}{2}}|E|^{3/2}\sum_{v=0}^{r-1}p^{v/2}\ll \frac{|E|^3}{p^r}+ p^{2r-1}|E|^{3/2}\ll \frac{|E|^3}{p^r},\]
whenever $|E|\gg p^{2r-\frac{2}{3}}$. This proves Theorem \ref{un-conditional}.

Indeed,
\[
p^{2r}\sum\limits_{\bfx\in E}\sum\limits_{v=0}^{r-2}\cE_v(\bfx) \ll p^{\frac{3r-1}{2}}|E|^{3/2}\sum_{v=0}^{r-2}p^{v/2} \ll p^{2r-\frac{3}{2}}|E|^{3/2}\ll \frac{|E|^3}{p^r}
\]
whenever $|E|\gg p^{2r-1}$. In the statement of Theorem \ref{conditional}, we know that
\[
\#\{(\bfx_1,\bfx_2)\in E^2:\, \bfx_1\equiv \bfx_2\,(\mod p)\}\ll p^{2r-\frac{7}{3}}|E|.
\]
Then
\[\sum\limits_{\tilde{\bfx}\,(\mod p)}|g_{E,1}(\tilde{\bfx})|^2 = \frac{1}{p^{4r-4}}\#\{(\bfx_1,\bfx_2)\in E^2:\, \bfx_1\equiv \bfx_2\,(\mod p)\} \ll  p^{-2r+\frac{5}{3}}|E| .\]
%\[\#\{\bfx'\in A\colon \bfx'\equiv\bfx\,(\mod p)\}\ll p^{2r-\frac{7}{3}}.\]
%This implies that
%\[\sum\limits_{\tilde{\bfy}\,(\mod p)}|g_{A,1}(\tilde{\bfy})|^2\le \frac{p^{2r-\frac{7}{3}}}{p^{2r-2}}\cdot\sum\limits_{\tilde{\bfy}\,(\mod p^{r-v})}|g_{A,r-v}(\tilde{\bfy})|\le p^{-\frac{1}{3}}\cdot \frac{|A|}{p^{2r-2}}=p^{-2r+\frac{5}{3}}|A|.\]
It follows that
\[
p^{2r}\sum\limits_{x\in E}\cE_{r-1}(\bfx) \ll  p^{5r-4}\cdot \left(p^{-2r+\frac{5}{3}}|E|\right)^{3/2} \ll p^{2r-\frac{3}{2}}|E|^{\frac{3}{2}} \ll\frac{|E|^3}{p^r},
\]
and then $\cN\ll \frac{|E|^3}{p^r}$. Plugging this bound to the above argument gives us Theorem \ref{conditional}.

{\color{black}
\begin{proof} [Proof of Corollary \ref{conditiona2}]
Under the condition of Corollary \ref{conditiona2}, we have
\[
\#\{(\bfx_1,\bfx_2)\in E^2:\, \bfx_1\equiv \bfx_2\,(\mod p)\} = \sum\limits_{\bfx_1\in E} \#\{\bfx_2\in E:\, \bfx_2\equiv \bfx_1 \,(\mod p)\} \ll p^{2r-7/3}|E|.
\]
Then the conclusion follows from Theorem \ref{conditional}.
\end{proof}
}

\section{An $\mathbb{F}_p$--to--$\mathbb{Z}/p^r\mathbb{Z}$ transfer principle in arbitrary dimensions}
\label{sec:ffr-general}

The purpose of this section is to formulate a  prime field-to-ring
transfer principle in arbitrary dimensions.  We work with the standard
quadratic form
\[
\|\bfx\|:=x_1^2+\cdots+x_n^2.
\]
The proof separates the finite field combinatorial input from the
higher-conductor analysis.  The dependence on the dimension enters
through the finite field isosceles energy, the zero-sphere estimate,
and the Fourier transform of a multiplicative character composed with
the quadratic form.

For an odd prime $p$ and an integer $j\geq1$, write
\[
R_j:=\bZ/p^j\bZ,
\qquad
U_j:=R_j^\times,
\]
and let $\widehat{U_j}$ denote the multiplicative character group of
$U_j$.  If $1\leq s\leq r$, let
\[
\rho_{r,s}:U_r\longrightarrow U_s
\]
be reduction modulo $p^s$.  We say that
$\chi\in\widehat{U_r}$ \emph{factors through $U_s$} if
\[
\chi=\chi_s\circ\rho_{r,s}
\]
for some $\chi_s\in\widehat{U_s}$.  Since $\rho_{r,s}$ is
surjective, the character $\chi_s$ is unique.  Characters that factor
through $U_1=\bF_p^\times$ will be called \emph{low-conductor
characters}.  For $2\leq s\leq r$, we say that $\chi$ has
\emph{exact conductor $p^s$} if it factors through $U_s$ but not
through $U_{s-1}$.  Equivalently, the induced character $\chi_s$ is
nontrivial on
\[
\ker(U_s\longrightarrow U_{s-1})=1+p^{s-1}R_s.
\]
Throughout this section, every multiplicative character of $U_j$ is
extended by zero to $R_j$.

There are two points that must be made precise.  First, the input is a
full isosceles-energy estimate, rather than only a lower bound for the
cardinality of a finite field distance set.  Second, in dimensions
$n\geq3$ this energy cannot be replaced by the count obtained by
requiring the base to have nonzero squared length: two distinct
endpoints may have zero squared distance.

For a weight $h:\bF_p^n\to[0,1]$, put
\[
H:=\sum_{\mathbf a\in\bF_p^n}h(\mathbf a)
\]
and define
\begin{equation}\label{eq:ffr-general-full-energy}
\mathcal T_{p,n}^{\times}(h)
:=
\sum_{\mathbf a,\bfb,\bfc\in\bF_p^n}
h(\mathbf a)h(\bfb)h(\bfc)
\mathbf 1_{\{
\|\mathbf a-\bfb\|=\|\mathbf a-\bfc\|\neq0
\}}.
\end{equation}
For a set $A\subset\bF_p^n$, we write
\[
\mathcal T_{p,n}^{\times}(A)
:=\mathcal T_{p,n}^{\times}(1_A).
\]
Thus, the superscript $\times$ records that the common squared length
is nonzero.  This is the full nonzero isosceles energy, including the
terms with $\bfb=\bfc$.

\begin{theorem}\label{thm:ffr-general-transfer}
Fix $n\geq2$ and
\[
\left\lfloor\frac n2\right\rfloor\leq\beta<n.
\]
Let $\mathfrak P$ be a collection of odd primes.  Suppose that there
are constants $C_0,C_1>0$, independent of $p\in\mathfrak P$, such
that, for every $p\in\mathfrak P$ and every set
$A\subset\bF_p^n$ with $|A|\geq C_0p^\beta$, one has
\begin{equation}\label{eq:ffr-general-set-input}
\mathcal T_{p,n}^{\times}(A)
\leq C_1\frac{|A|^3}{p}.
\end{equation}
Then there are constants
$C_2=C_2(n,C_0,C_1)>0$ and
$c_2=c_2(n,C_0,C_1)>0$ with the following property.  For every
$p\in\mathfrak P$, every $r\geq1$, and every
$E\subset(\bZ/p^r\bZ)^n$ satisfying
\[
|E|\geq C_2p^{n(r-1)+\beta},
\qquad\text{equivalently}\qquad
\delta_E\geq C_2p^{\beta-n},
\]
one has
\[
|\Delta_{n,r}(E)\cap U_r|\geq c_2p^r.
\]
In particular, $|\Delta_{n,r}(E)|\gg p^r$, and all implied constants
are independent of both $p$ and $r$.
\end{theorem}

Taking $\mathfrak P=\{p\}$ gives the fixed-prime formulation.  The
family formulation above merely records when the constants are
uniform as $p$ varies.

In two dimensions, it follows from the paper \cite{Murphy} that for $A\subset \mathbb{F}_p^2$, if $|A|\gg p^{\frac{5}{4}}$, then one has
\[
\mathcal T_{p,n}^{\times}(A)
\ll \frac{|A|^3}{p}.\]
In four dimensions \cite{TPXUE}, we obtained the same conclusion under the assumption that $|A|\gg p^{4-\frac{8}{r+2}}$ for any $r>23/7$.

Therefore, Theorems \ref{un-conditional-26} and \ref{un-conditional-27} follow immediately from Theorem \ref{thm:ffr-general-transfer}. 
\paragraph{Sketch of ideas.}
The transfer has three transparent steps.  First, reduce $E$ modulo
$p$ and record, at each point of $\bF_p^n$, the proportion of its
$p^{n(r-1)}$ lifts that lie in $E$.  This replaces the possibly very
irregular fibres of $E$ by a weight $h_E$ on the finite field.  Second,
expand the unit-distance second moment by multiplicative characters.
The characters that see only reduction modulo $p$ reproduce the
weighted finite field isosceles energy exactly, apart from the
predictable factor $p^{(3n-1)(r-1)}$.  Third, the characters that see
higher $p$-adic digits are controlled at once by the
$n$-dimensional quadratic Gauss transform; their Fourier multipliers
have size at most $p^{-sn/2}$ at conductor $p^s$. Thus, all of the
combinatorial information is supplied at the first residue level, and
the remaining digits contribute only a uniform analytic error.  A
short estimate for the zero sphere in $\bF_p^n$ guarantees that, when
$H_E\gg p^{\lfloor n/2\rfloor}$, a positive proportion of the pairs
have nonzero distance and hence lift to unit distances.

For
$E\subset(\bZ/p^r\bZ)^n$, define
\begin{equation}\label{eq:ffr-general-weight}
h_E(\mathbf a)
:=
p^{-n(r-1)}
\#\{\bfx\in E:\ \bfx\equiv\mathbf a\,(\mod p)\},
\qquad \mathbf a\in\bF_p^n.
\end{equation}
Then, $0\leq h_E\leq1$, and
\begin{equation}\label{eq:ffr-general-mass}
H_E:=\sum_{\mathbf a\in\bF_p^n}h_E(\mathbf a)
=\frac{|E|}{p^{n(r-1)}}.
\end{equation}
For $\bfx\in E$ and $j\in U_r$, let
\[
\nu_{\bfx}(j)
:=\#\{\bfy\in E:\ \|\bfx-\bfy\|=j\},
\]
and define
\begin{equation}\label{eq:ffr-general-ring-energy}
\cN_n^\times(E)
:=\sum_{\bfx\in E}\sum_{j\in U_r}\nu_{\bfx}(j)^2.
\end{equation}

The following two lemmas separate the combinatorial input from the
ring-theoretic argument.

\begin{lemma}\label{lem:ffr-general-weighted-transfer}
Let $n\geq2$ and
$\lfloor n/2\rfloor\leq\beta<n$, and let $p$ be an odd prime.
Suppose that $C_0',C_1'>0$ and that, whenever
$h:\bF_p^n\to[0,1]$ has mass $H\geq C_0'p^\beta$, one has
\[
\mathcal T_{p,n}^{\times}(h)
\leq C_1'\frac{H^3}{p}.
\]
Then, there are constants
$C_4=C_4(n,C_0',C_1')>0$ and
$c_4=c_4(n,C_0',C_1')>0$, independent of $p$ and $r$, such that
for every $r\geq1$ and every $E\subset(\bZ/p^r\bZ)^n$ with
$|E|\geq C_4p^{n(r-1)+\beta}$, one has
\[
|\Delta_{n,r}(E)\cap U_r|\geq c_4p^r.
\]
\end{lemma}

\begin{lemma}\label{lem:ffr-general-randomization}
Let $n\geq2$ and $1\leq\beta<n$.  Suppose that there are constants
$C_0,C_1>0$ such that every $A\subset\bF_p^n$ with
$|A|\geq C_0p^\beta$ satisfies
\[
\mathcal T_{p,n}^{\times}(A)
\leq C_1\frac{|A|^3}{p}.
\]
Put $C_3:=\max\{2C_0,3\}$. Then, every weight
$h:\bF_p^n\to[0,1]$ of mass $H\geq C_3p^\beta$ satisfies
\begin{equation}\label{eq:ffr-general-weighted-input}
\mathcal T_{p,n}^{\times}(h)
\leq(5C_1+1)\frac{H^3}{p}.
\end{equation}
\end{lemma}
Theorem \ref{thm:ffr-general-transfer} follows directly from these two lemmas.

\subsection{Proof of Lemma \ref{lem:ffr-general-weighted-transfer}}

We begin with the estimate that ensures the existence of sufficiently
many unit-distance pairs.

\begin{lemma}\label{lem:ffr-general-zero-cone}
Let $h:\bF_p^n\to[0,1]$ have mass $H$.  Then
\begin{equation}\label{eq:ffr-general-zero-cone}
\sum_{\mathbf a,\bfb\in\bF_p^n}
h(\mathbf a)h(\bfb)
\mathbf 1_{\{\|\mathbf a-\bfb\|=0\}}
\leq
\bigl(p^{-1}+p^{-n/2}\bigr)H^2
+p^{\lfloor n/2\rfloor}H.
\end{equation}
Consequently, if
$H\geq6p^{\lfloor n/2\rfloor}$, then
\begin{equation}\label{eq:ffr-general-nonzero-pairs}
\sum_{\mathbf a,\bfb\in\bF_p^n}
h(\mathbf a)h(\bfb)
\mathbf 1_{\{\|\mathbf a-\bfb\|\neq0\}}
\geq\frac{H^2}{6}.
\end{equation}
\end{lemma}

\begin{proof}
Let
\[
S_0:=\{\bfz\in\bF_p^n:\ \|\bfz\|=0\},
\]
and let $\mathcal A$ be the convolution matrix
$\mathcal A_{\mathbf a,\bfb}=1_{S_0}(\mathbf a-\bfb)$.
The additive characters of $\bF_p^n$ are eigenvectors of
$\mathcal A$.  The constant eigenvector has eigenvalue $|S_0|$, and
the eigenvalue at $\bfm\neq\bfo$ is
\[
\lambda(\bfm)
=\sum_{\bfz\in S_0}e_p(-\bfm\cdot\bfz).
\]
Let $\eta$ be the quadratic character of $\bF_p^*$, and let
\[
\mathfrak g_p:=\sum_{x\in\bF_p}e_p(x^2),
\qquad |\mathfrak g_p|=p^{1/2}.
\]
Additive orthogonality followed by completing the square gives, for
$\bfm\neq\bfo$,
\begin{equation}\label{eq:ffr-general-cone-eigenvalue}
\lambda(\bfm)
=\frac{\mathfrak g_p^n}{p}
\sum_{t\in\bF_p^*}
\eta(t)^n
e_p\left(-\frac{\|\bfm\|}{4t}\right).
\end{equation}
If $n$ is even, the last sum has modulus $p-1$ when
$\|\bfm\|=0$ and modulus $1$ otherwise.  If $n$ is odd, it is zero
when $\|\bfm\|=0$ and is a quadratic Gauss sum of modulus $p^{1/2}$
otherwise. Hence
\begin{equation}\label{eq:ffr-general-cone-spectrum}
\max_{\bfm\neq\bfo}|\lambda(\bfm)|
\leq p^{\lfloor n/2\rfloor}.
\end{equation}
The same calculation at $\bfm=\bfo$, now including the term $t=0$,
also gives
\begin{equation}\label{eq:ffr-general-cone-size}
\frac{|S_0|}{p^n}\leq p^{-1}+p^{-n/2}.
\end{equation}

For completeness, the evaluations just used follow after the change
of variables $u=t^{-1}$.  When $n$ is even, the sum in
\eqref{eq:ffr-general-cone-eigenvalue} becomes
$\sum_{u\neq0}e_p(-\|\bfm\|u/4)$, which equals $p-1$ if
$\|\bfm\|=0$ and $-1$ otherwise.  When $n$ is odd, it becomes a
quadratic Gauss sum if $\|\bfm\|\neq0$ and vanishes if
$\|\bfm\|=0$.

Write $h=H/p^n+h_0$, where $\sum_{\mathbf a}h_0(\mathbf a)=0$.
By \eqref{eq:ffr-general-cone-spectrum}, the spectral theorem, and
$0\leq h\leq1$,
\begin{align*}
\sum_{\mathbf a,\bfb}
h(\mathbf a)h(\bfb)1_{S_0}(\mathbf a-\bfb)
&=
\frac{|S_0|}{p^n}H^2
+\langle\mathcal A h_0,h_0\rangle\\
&\leq
\frac{|S_0|}{p^n}H^2
+|\langle\mathcal A h_0,h_0\rangle|\\
&\leq
\bigl(p^{-1}+p^{-n/2}\bigr)H^2
+p^{\lfloor n/2\rfloor}\|h_0\|_2^2\\
&\leq
\bigl(p^{-1}+p^{-n/2}\bigr)H^2
+p^{\lfloor n/2\rfloor}\sum_{\mathbf a}h(\mathbf a)^2\\
&\leq
\bigl(p^{-1}+p^{-n/2}\bigr)H^2
+p^{\lfloor n/2\rfloor}H.
\end{align*}
Here we used
$\|h_0\|_2^2=\sum_{\mathbf a}h(\mathbf a)^2-H^2/p^n$.
This proves \eqref{eq:ffr-general-zero-cone}.  Since $p\geq3$ and
$n\geq2$, one has $p^{-1}+p^{-n/2}\leq2/3$.  If
$H\geq6p^{\lfloor n/2\rfloor}$, the last term in
\eqref{eq:ffr-general-zero-cone} is at most $H^2/6$. Hence, the
zero-distance sum is at most $5H^2/6$, and subtracting it from $H^2$
proves \eqref{eq:ffr-general-nonzero-pairs}.
\end{proof}

\begin{proof}[Proof of Lemma \ref{lem:ffr-general-weighted-transfer}]
We first establish the following estimate
\begin{equation}\label{eq:ffr-general-energy-bridge}
\cN_n^\times(E)
\leq
p^{(3n-1)(r-1)}\mathcal T_{p,n}^{\times}(h_E)
+|E|q^{2n-1}p^{2-2n}.
\end{equation}
Let us see first how this comparison proves the lemma.

Assume
\[
|E|\geq C_4p^{n(r-1)+\beta},
\qquad
C_4:=\max\{C_0',6\}.
\]
Then
\[
H_E\geq C_4p^\beta
\geq\max\{C_0'p^\beta,\,6p^{\lfloor n/2\rfloor},\,p\}.
\]
Thus, the hypothesis gives
\[
\mathcal T_{p,n}^{\times}(h_E)
\leq C_1'\frac{H_E^3}{p}.
\]
Therefore
\begin{align}
p^{(3n-1)(r-1)}
\mathcal T_{p,n}^{\times}(h_E)
&\leq
C_1'p^{(3n-1)(r-1)}\frac{H_E^3}{p}\notag=C_1'\frac{|E|^3}{q}.
\label{eq:ffr-general-low-main}
\end{align}
Moreover,
\[
\frac{|E|q^{2n-1}p^{2-2n}}{|E|^3/q}
=\frac{p^2}{H_E^2}\leq1.
\]
It follows from \eqref{eq:ffr-general-energy-bridge} that
\begin{equation}\label{eq:ffr-general-final-energy}
\cN_n^\times(E)\leq(C_1'+1)\frac{|E|^3}{q}.
\end{equation}

Arguing as the previous section, we obtain
\[
|\Delta_{n,r}(E)\cap U_r|
\geq\frac{q}{36(C_1'+1)}.
\]

It remains to prove \eqref{eq:ffr-general-energy-bridge}.  We use the
conductor decomposition introduced above.  For
$\chi\in\widehat{U_r}$, define
\[
S_\chi(\bfx)
:=\sum_{\bfy\in E}\chi(\|\bfx-\bfy\|).
\]
For $u,v\in\bZ/p^r\bZ$, multiplicative character orthogonality,
with every character extended by zero away from $U_r$, reads
\[
\frac{1}{\varphi(p^r)}
\sum_{\chi\in\widehat{U_r}}
\chi(u)\overline{\chi(v)}
=\mathbf 1_{\{u=v\in U_r\}}.
\]
Expanding $\nu_{\bfx}(j)^2$ and applying this identity gives
\begin{equation}\label{eq:ffr-general-character-energy}
\cN_n^\times(E)
=
\frac{1}{\varphi(p^r)}
\sum_{\chi\in\widehat{U_r}}
\sum_{\bfx\in E}|S_\chi(\bfx)|^2.
\end{equation}
Let $\cN_{n,\mathrm{low}}^\times(E)$ denote the part of
\eqref{eq:ffr-general-character-energy} coming from characters that
factor through $U_1=\bF_p^*$, and let
$\cN_{n,\mathrm{high}}^\times(E)$ denote the sum over characters of
exact conductor $p^s$, $2\leq s\leq r$.  Thus
\[
\cN_n^\times(E)
=\cN_{n,\mathrm{low}}^\times(E)
+\cN_{n,\mathrm{high}}^\times(E).
\]

We first compute the low-conductor term.  If
$\bfx\equiv\mathbf a\,(\mod p)$ and $\chi$ is induced by
$\chi_1\in\widehat{\bF_p^*}$, grouping $\bfy$ by its residue class
gives
\begin{equation}\label{eq:ffr-general-low-S}
S_\chi(\bfx)
=p^{n(r-1)}
\sum_{\bfb\in\bF_p^n}
h_E(\bfb)\chi_1(\|\mathbf a-\bfb\|).
\end{equation}
There are $p^{n(r-1)}h_E(\mathbf a)$ choices of $\bfx\in E$ above
$\mathbf a$.  Substituting \eqref{eq:ffr-general-low-S} and using
orthogonality on $\bF_p^*$ in the form
\[
\sum_{\chi_1\in\widehat{\bF_p^*}}
\chi_1(u)\overline{\chi_1(v)}
=(p-1)\mathbf 1_{\{u=v\neq0\}},
\qquad u,v\in\bF_p,
\]
we obtain
\begin{align}
\cN_{n,\mathrm{low}}^\times(E)
&=
\frac{p^{3n(r-1)}}{\varphi(p^r)}
\sum_{\chi_1\in\widehat{\bF_p^*}}
\sum_{\mathbf a\in\bF_p^n}
h_E(\mathbf a)
\left|
\sum_{\bfb\in\bF_p^n}
h_E(\bfb)\chi_1(\|\mathbf a-\bfb\|)
\right|^2\notag\\
&=
\frac{p^{3n(r-1)}(p-1)}{\varphi(p^r)}
\mathcal T_{p,n}^{\times}(h_E)\notag\\
&=
p^{(3n-1)(r-1)}
\mathcal T_{p,n}^{\times}(h_E).
\label{eq:ffr-general-low-exact}
\end{align}
This identity is the precise point at which the finite field energy
enters the ring argument.

We next treat the high-conductor term.  Fix $2\leq s\leq r$ and a
character $\chi$ of exact conductor $p^s$.  Put
\[
f:=1_E,
\qquad
K_\chi(\bfz):=\chi(\|\bfz\|).
\]
With the normalized convolution from Section~2,
\[
S_\chi=q^n(f\ast K_\chi),
\qquad
\widehat{S_\chi}(\bfm)
=q^n\widehat f(\bfm)\widehat{K_\chi}(\bfm).
\]
Parseval's identity and
Lemma~\ref{lem:ffr-general-multiplier} below give
\begin{align}
\sum_{\bfx\in E}|S_\chi(\bfx)|^2
&\leq
\sum_{\bfx\,(\mod p^r)}|S_\chi(\bfx)|^2\notag\\
&=
q^{3n}\sum_{\bfm\,(\mod p^r)}
|\widehat f(\bfm)|^2|\widehat{K_\chi}(\bfm)|^2\notag\\
&\leq
q^{3n}p^{-sn}\frac{|E|}{q^n}
=q^{2n}p^{-sn}|E|.
\label{eq:ffr-general-one-high-character}
\end{align}
The number of characters that factor through $U_s$ is
$|\widehat{U_s}|=\varphi(p^s)$. Hence, the number of characters of
exact conductor $p^s$ is
\[
\varphi(p^s)-\varphi(p^{s-1})
=(p-1)^2p^{s-2}.
\]
Consequently,
\begin{align}
\cN_{n,\mathrm{high}}^\times(E)
&\leq
\frac{q^{2n}|E|}{\varphi(p^r)}
\sum_{s=2}^r(p-1)^2p^{s-2-sn}\notag\\
&=
|E|q^{2n-1}p^{2-2n}
\frac{(1-p^{-1})
\bigl(1-p^{-(n-1)(r-1)}\bigr)}
{1-p^{-(n-1)}}\notag\\
&\leq
|E|q^{2n-1}p^{2-2n}.
\label{eq:ffr-general-high-bound}
\end{align}
Here we used
\[
\sum_{s=2}^r p^{s-2-sn}
=p^{-2n}
\frac{1-p^{-(n-1)(r-1)}}{1-p^{-(n-1)}},
\]
and the factor multiplying
$|E|q^{2n-1}p^{2-2n}$ in the middle line of
\eqref{eq:ffr-general-high-bound} is at most one because $n\geq2$.
When $r=1$, the high-conductor sum is empty and the same conclusion
holds.

Combining \eqref{eq:ffr-general-low-exact} and
\eqref{eq:ffr-general-high-bound} proves
\eqref{eq:ffr-general-energy-bridge}, and hence the lemma with
$C_4=\max\{C_0',6\}$ and
$c_4=1/(36(C_1'+1))$.
\end{proof}

\subsection*{Quadratic Gauss transforms in arbitrary dimensions}

We now prove the multiplier estimate used above.  The argument is
uniform in the dimension: additive Fourier inversion in the value of
the quadratic form separates the $n$ coordinates, and each coordinate
contributes a one-dimensional quadratic Gauss sum.

We first record the scalar multiplicative Gauss sum needed in the
argument.  The modulus at a unit frequency is the specialization of
\cite[Lemma~5.1(ii)]{Nica} to $R=\bZ/p^s\bZ$; we include the
elementary proof because it also gives the required vanishing at
nonunit frequencies.

\begin{lemma}
\label{lem:ffr-general-primitive-gauss}
Let $s\geq2$, and let $\theta\in\widehat{U_s}$ have exact conductor
$p^s$.  For $c\in R_s$, define
\[
\mathfrak G_s(\theta;c)
:=
\sum_{u\in U_s}\theta(u)e_{p^s}(cu).
\]
Then
\[
\mathfrak G_s(\theta;c)
=
\overline{\theta(c)}\,\mathfrak G_s(\theta;1)
\qquad(c\in U_s),
\]
whereas
\[
\mathfrak G_s(\theta;c)=0
\qquad(c\in pR_s).
\]
Moreover,
\[
\left|\mathfrak G_s(\theta;1)\right|=p^{s/2}.
\]
Consequently,
\begin{equation}\label{eq:ffr-general-primitive-gauss-size}
\left|\mathfrak G_s(\theta;c)\right|
=
\begin{cases}
p^{s/2},&c\in U_s,\\
0,&c\in pR_s.
\end{cases}
\end{equation}
\end{lemma}

\begin{proof}
Put
\[
V_s:=\ker(U_s\longrightarrow U_{s-1})
=1+p^{s-1}R_s.
\]
Since $\theta$ has exact conductor $p^s$, its restriction to $V_s$ is
nontrivial.

Suppose first that $c\in pR_s$.  Choose $\varepsilon\in V_s$ with
$\theta(\varepsilon)\neq1$.  For every $u\in U_s$, one has
\[
c(\varepsilon-1)u=0
\qquad\text{in }R_s.
\]
Multiplication by $\varepsilon$ permutes $U_s$, and hence
\[
\mathfrak G_s(\theta;c)
=\theta(\varepsilon)\mathfrak G_s(\theta;c).
\]
Thus, $\mathfrak G_s(\theta;c)=0$.

If $c\in U_s$, the substitution $v=cu$ gives
\[
\mathfrak G_s(\theta;c)
=\overline{\theta(c)}\,\mathfrak G_s(\theta;1).
\]
It remains to evaluate the modulus at $c=1$.  For $a\in R_s$, put
\[
C_s(a):=\sum_{v\in U_s}e_{p^s}(av).
\]
Writing the unit sum as the sum over all residue classes minus the sum
over the multiples of $p$ gives
\[
C_s(a)
=p^s\mathbf 1_{\{a=0\text{ in }R_s\}}
-p^{s-1}\mathbf 1_{\{a\in p^{s-1}R_s\}}.
\]
Writing $u=tv$, we therefore obtain
\begin{align*}
\left|\mathfrak G_s(\theta;1)\right|^2
&=\sum_{t\in U_s}\theta(t)C_s(t-1)\\
&=\varphi(p^s)
-p^{s-1}\sum_{\substack{t\in V_s\\t\neq1}}\theta(t).
\end{align*}
Since $\theta|_{V_s}$ is nontrivial,
\[
\sum_{t\in V_s}\theta(t)=0.
\]
The last restricted sum is therefore $-1$, and hence
\[
\left|\mathfrak G_s(\theta;1)\right|^2
=\varphi(p^s)+p^{s-1}
=p^s.
\]
This proves the lemma.
\end{proof}

\begin{lemma}
\label{lem:ffr-general-multiplier}
Let $p$ be odd, let $n\geq2$, let $2\leq s\leq r$, and let
$\chi\in\widehat{U_r}$ have exact conductor $p^s$.  Define
\[
K_\chi(\bfz):=\chi(\|\bfz\|),
\qquad \bfz\in R_r^n,
\]
where $\chi$ is extended by zero to $R_r$.
Then $\widehat{K_\chi}(\bfm)=0$ unless
\[
\bfm=p^{r-s}\widetilde{\bfm}
\quad\text{for some}\quad
\widetilde{\bfm}\in(\bZ/p^s\bZ)^n.
\]
At such a frequency,
\begin{equation}\label{eq:ffr-general-multiplier-exact}
\left|
\widehat{K_\chi}
\bigl(p^{r-s}\widetilde{\bfm}\bigr)
\right|
=
\begin{cases}
p^{-sn/2},
&\|\widetilde{\bfm}\|\in U_s,\\
0,
&\|\widetilde{\bfm}\|\in pR_s.
\end{cases}
\end{equation}
In particular,
\[
\|\widehat{K_\chi}\|_{\ell^\infty((\bZ/p^r\bZ)^n)}
\leq p^{-sn/2}.
\]
\end{lemma}

\begin{proof}
Since $\chi$ factors through $U_s$, there is a unique
$\chi_s\in\widehat{U_s}$ such that
\[
\chi=\chi_s\circ\rho_{r,s}.
\]
The exact-conductor hypothesis says that $\chi_s$ has exact conductor
$p^s$.  In particular, $K_\chi$ depends only on
$\bfz\,(\mod p^s)$.  Summing over the lifts of each residue class
modulo $p^s$ shows that its Fourier transform vanishes unless
$\bfm=p^{r-s}\widetilde{\bfm}$.  At a frequency of this form, the
normalized coefficient at level $r$ is exactly the normalized
coefficient at level $s$ formed with $\chi_s$.  We may therefore
assume that $r=s$ and, to simplify notation, write $\chi$ for
$\chi_s$.

Indeed, writing $\bfz=\bfz_0+p^s\bfw$, the sum over
$\bfw\,(\mod p^{r-s})$ contains the factor
\[
\prod_{i=1}^n
\sum_{w_i\,(\mod p^{r-s})}
e_{p^{r-s}}(-m_iw_i),
\]
which vanishes unless every $m_i$ is divisible by $p^{r-s}$.  If
$\bfm=p^{r-s}\widetilde{\bfm}$, the lift multiplicity
$p^{n(r-s)}$ cancels the corresponding part of the normalization
$p^{-rn}$, leaving precisely the normalized level-$s$ coefficient.

Let $\eta$ denote the quadratic character of $\bF_p^*$, and also its
lift $\eta\circ\rho_{s,1}$ to $U_s$.  For $a\in U_s$, put
\[
g_s(a):=\sum_{x\,(\mod p^s)}e_{p^s}(ax^2).
\]
We first record the elementary evaluation
\begin{equation}\label{eq:ffr-general-additive-quadratic}
g_s(a)=\eta(a)^s g_s(1),
\qquad
|g_s(1)|=p^{s/2}.
\end{equation}
Indeed, writing $x=y+p^{s-1}z$, with
$y\,(\mod p^{s-1})$ and $z\,(\mod p)$, and summing first in $z$ gives
\[
\sum_{z\,(\mod p)}
e_p(2ayz)
=
\begin{cases}
p,&p\mid y,\\
0,&p\nmid y.
\end{cases}
\]
Writing $y=py'$ in the surviving terms therefore gives
\[
g_s(a)=p\,g_{s-2}(a\bmod p^{s-2})
\qquad(s\geq2),
\]
where $g_0=1$.  At level one,
\[
g_1(a)=\eta(a)g_1(1),
\qquad |g_1(1)|=p^{1/2},
\]
by the finite field quadratic Gauss-sum calculation.  For completeness,
extend $\eta$ by $\eta(0)=0$.  Since the number of solutions of
$x^2=t$ is $1+\eta(t)$ and
$\sum_{t\in\bF_p}e_p(at)=0$ for $a\neq0$, one has
\[
g_1(a)
=\sum_{t\in\bF_p}(1+\eta(t))e_p(at)
=\sum_{t\in\bF_p}\eta(t)e_p(at)
=\eta(a)\sum_{t\in\bF_p}\eta(t)e_p(t).
\]
If $\tau(\eta):=\sum_t\eta(t)e_p(t)$, then
\begin{align*}
|\tau(\eta)|^2
&=
\sum_{u\in\bF_p^*}\eta(u)
\sum_{y\in\bF_p^*}e_p((u-1)y)\\
&=(p-1)-\sum_{\substack{u\in\bF_p^*\\u\neq1}}\eta(u)
=p.
\end{align*}
Thus, $|g_1(1)|=p^{1/2}$.  Iterating the recurrence proves
\eqref{eq:ffr-general-additive-quadratic}.

Define the primitive multiplicative Gauss sum with positive additive
sign by
\[
\tau_+(\overline\chi)
:=
\sum_{a\in U_s}\overline{\chi(a)}e_{p^s}(a)
=\mathfrak G_s(\overline\chi;1).
\]
The character $\overline\chi$ also has exact conductor $p^s$, and
Lemma~\ref{lem:ffr-general-primitive-gauss} gives
\[
|\tau_+(\overline\chi)|=p^{s/2}.
\]
If $t\in U_s$, the change of variables $w=at$ gives
\[
\sum_{a\in U_s}\overline{\chi(a)}e_{p^s}(at)
=\chi(t)\tau_+(\overline\chi).
\]
If $t\in pR_s$, the sum on the left is
$\mathfrak G_s(\overline\chi;t)$, which vanishes by the same lemma.
Hence, under the zero-extension convention, one has for every
$t\in R_s$ the inversion formula
\begin{equation}\label{eq:ffr-general-multiplicative-inversion}
\chi(t)
=
\frac{1}{\tau_+(\overline\chi)}
\sum_{a\in U_s}
\overline{\chi(a)}e_{p^s}(at).
\end{equation}

Let $\mathcal K_\chi(\bfm)$ denote the unnormalized Fourier transform
of $K_\chi$ at level $s$.  Inserting
\eqref{eq:ffr-general-multiplicative-inversion}, separating the
$n$ coordinates, and completing the square in each coordinate, we
obtain
\begin{align}
\mathcal K_\chi(\bfm)
&:=
\sum_{\bfz\,(\mod p^s)}
\chi(\|\bfz\|)e_{p^s}(-\bfm\cdot\bfz)\notag\\
&=
\frac{1}{\tau_+(\overline\chi)}
\sum_{a\in U_s}\overline{\chi(a)}
\prod_{i=1}^n
\sum_{z_i\,(\mod p^s)}
e_{p^s}(az_i^2-m_i z_i)\notag\\
&=
\frac{g_s(1)^n}{\tau_+(\overline\chi)}
\sum_{a\in U_s}
\overline{\chi(a)}\eta(a)^{sn}
e_{p^s}\left(-\frac{\|\bfm\|}{4a}\right).
\label{eq:ffr-general-completed-square}
\end{align}
After the substitution $b=a^{-1}$, the final sum in
\eqref{eq:ffr-general-completed-square} becomes
\[
\mathfrak G_s
\left(\chi\eta^{sn};-\frac{\|\bfm\|}{4}\right).
\]
The lifted character $\eta$ is trivial on
$V_s=\ker(U_s\longrightarrow U_{s-1})$.  Consequently,
$\chi\eta^{sn}$ has the same nontrivial restriction to $V_s$ as
$\chi$, and hence also has exact conductor $p^s$.
Lemma~\ref{lem:ffr-general-primitive-gauss} therefore shows that the
last sum vanishes when $\|\bfm\|\in pR_s$ and has modulus $p^{s/2}$
when $\|\bfm\|\in U_s$.  In the latter case,
\[
|\mathcal K_\chi(\bfm)|
=
\frac{p^{sn/2}}{p^{s/2}}\,p^{s/2}
=p^{sn/2}.
\]
Dividing by $p^{sn}$, as required by the normalized Fourier
transform, proves \eqref{eq:ffr-general-multiplier-exact}.
\end{proof}

\begin{remark}
The restriction $s\geq2$ is essential: at level $s=1$, the twist by
$\eta^{sn}$ can destroy primitivity.  In the transfer argument, the
characters factoring through $U_1$ are therefore kept in the
low-conductor contribution.
\end{remark}

\subsection{Proof of Lemma \ref{lem:ffr-general-randomization}}
The reduction modulo \(p\) naturally produces a weight \(h_E\), whereas
the finite field hypothesis is stated for subsets of \(\bF_p^n\).  We
bridge this gap by independently sampling a set \(A\), using \(h(\mathbf
a)\) as the probability that \(\mathbf a\) belongs to \(A\).  Since the
full energy includes the terms with \(\bfb=\bfc\), the expected sampled
energy is slightly larger than \(\mathcal T_{p,n}^{\times}(h)\), which
gives precisely the one-sided inequality needed below.  A lower-tail
estimate then controls the probability that the sampled set is too
small for the assumed set estimate to apply.
\begin{proof}
Choose a random subset $A\subset\bF_p^n$ by including each point
$\mathbf a$ independently with probability $h(\mathbf a)$, and put
$X:=|A|$.  For a triple with $\bfb\neq\bfc$, the nonzero common
squared length ensures that its three vertices are distinct, so its
expected indicator is $h(\mathbf a)h(\bfb)h(\bfc)$.  When
$\bfb=\bfc$, the apex and endpoint are distinct, and the expected
indicator is $h(\mathbf a)h(\bfb)$, whereas the corresponding term in
$\mathcal T_{p,n}^{\times}(h)$ is
$h(\mathbf a)h(\bfb)^2$.  Consequently,
\begin{align}
\mathbb E\mathcal T_{p,n}^{\times}(A)
-\mathcal T_{p,n}^{\times}(h)
&=
\sum_{\mathbf a,\bfb\in\bF_p^n}
h(\mathbf a)h(\bfb)(1-h(\bfb))
\mathbf 1_{\{\|\mathbf a-\bfb\|\neq0\}}
\geq0.
\label{eq:ffr-general-randomization-comparison}
\end{align}
In particular,
\begin{equation}\label{eq:ffr-general-randomization-upper}
\mathcal T_{p,n}^{\times}(h)
\leq\mathbb E\mathcal T_{p,n}^{\times}(A).
\end{equation}

We use the elementary identity
\[
X^3=X(X-1)(X-2)+3X(X-1)+X.
\]
Taking expectations and using independence yields
\[
\mathbb E X^3\leq H^3+3H^2+H\leq5H^3,
\]
since $H\geq1$.

Let $L:=C_0p^\beta$, and take
\[
H\geq C_3p^\beta,
\qquad
C_3=\max\{2C_0,3\}.
\]
Then $L\leq H/2$.  We shall use the elementary lower-tail estimate
\begin{equation}\label{eq:ffr-general-lower-tail}
\mathbb P(X<H/2)\leq e^{-H/8}.
\end{equation}
For completeness, Markov's inequality and independence give
\begin{align*}
\mathbb P(X<H/2)
&\leq
e^{H/2}\mathbb E e^{-X}=
e^{H/2}
\prod_{\mathbf a\in\bF_p^n}
\bigl(1-h(\mathbf a)+h(\mathbf a)e^{-1}\bigr)\\
&\leq
\exp\bigl(H/2-(1-e^{-1})H\bigr)
\leq e^{-H/8},
\end{align*}
which proves \eqref{eq:ffr-general-lower-tail}.

Since $L\leq H/2$, the event $\{X<L\}$ is contained in
$\{X<H/2\}$. Hence
\[
\mathbb P(X<L)\leq e^{-H/8}.
\]
On the event $X\geq L$, the set estimate
assumed in the lemma applies.  On its complement, the
trivial estimate
$\mathcal T_{p,n}^{\times}(A)\leq X^3\leq L^3$ applies. Hence
\begin{align*}
\mathcal T_{p,n}^{\times}(h) \leq
\mathbb E\mathcal T_{p,n}^{\times}(A)\leq
\frac{C_1}{p}\mathbb E X^3
+L^3\mathbb P(X<L)\leq
5C_1\frac{H^3}{p}
+H^3e^{-H/8}.
\end{align*}
Since $\beta\geq1$, our choice $H\geq3p^\beta$ implies
$e^{-H/8}\leq p^{-1}$ for every odd prime $p$; indeed,
$H/8\geq3p/8\geq\log p$ for $p\geq3$.  Therefore
\[
\mathcal T_{p,n}^{\times}(h)
\leq(5C_1+1)\frac{H^3}{p},
\]
which proves
\eqref{eq:ffr-general-weighted-input}.
\end{proof}

\section{Proof of Theorems \ref{rectangle}, \ref{cycle}, \ref{chain}, and \ref{tree}}

In this section, we write $q=p^r$ for simplicity. To prove Theorems \ref{rectangle} and \ref{chain}, we use a general framework developed for pseudo-random graphs in which Lemma \ref{lemma_fourier} plays a crucial role.

\paragraph{Proof of Theorem \ref{rectangle}:} Let $G=(V, E)$ be a graph with the vertex set $V$ and the edge set $E$. Let $M$ be its adjacency matrix, i.e., $M_{ij}=1$ if and only if there is an edge between $i$ and $j$ and zero otherwise.
Let $\lambda_1\ge \lambda_2\ge \ldots\ge \lambda_{|V|} $ be its eigenvalues. Assume that $|V|=N$, the graph $G$ is called a $(N, d, \lambda)$-graph if each vertex is of degree $d$ and the second eigenvalue, determined by $\max\{\lambda_2, -\lambda_N\}$, is at most $\lambda$.

Let $G$ and $H$ be two graphs. The Cartesian product of $G$ and $H$, denoted by $G\square H$, is defined as follows. The vertex set $V(G\square H)=V(G)\times V(H)$, and there is an edge between $(u_1, v_1)$ and $(u_2, v_2)$ if and only if either $u_1=u_2$ and $(v_1, v_2) \in E(H)$ or $v_1=v_2$ and $(u_1, u_2) \in V(G)$.

We recall the following result on the number of rectangles in a Cartesian product of graphs in \cite{TP}.

\begin{theorem}\label{thm:CS}
Let $G_i$ be $(N_i ,d_i, \lambda_i)$-graphs with $1\le i\le 2$. Set $G=G_1\square G_2$.
For any $0<\delta'<\delta<1$, there exists %{\color{red}$n_0=n_0(\delta)$ such that if $|V(G_1\square G_2)|\ge n_0$}
$\epsilon>0$ such that for any $A\subset V(G_1\square G_2)$ with $|A|\ge \delta |V(G_1\square G_2)|$, if $\max\left\{ \frac{\lambda_1}{d_1}, \frac{\lambda_2}{d_2}\right\} < \epsilon$, then
\[N=\sum_{(u_1, u_2)\in E(G_1), (v_1,v_2)\in E(G_2)}A(u_1, v_1)A(u_1, v_2)A(u_2, v_1)A(u_2, v_2)>\delta'^4 N_1N_2d_1d_2.\]
\end{theorem}

Given $j\in (\mathbb{Z}/p^r\mathbb{Z})^*$. To prove Theorem \ref{rectangle}, we will apply the above theorem with appropriate graphs $G$ and $H$. More precisely, define $G=H$ with $V(G)=V(H)=(\mathbb{Z}/p^r\mathbb{Z})^n$, and there is an edge between two vertices $\bfx$ and $\bfy$ if and only if $F(\bfx-\bfy)\equiv j\,(\mod q)$. So $G$ and $H$ have $q^n$ vertices and are regular of degree $(1+o(1))q^{n-1}$. On the other hand, since $G$ and $H$ are Cayley graphs with the generating set $S=\{\bfx\in (\mathbb{Z}/p^r\mathbb{Z})^n\colon F(\bfx)\equiv j\,(\mod q)\}$. Thus, all non-trivial eigenvalues of $G$ and $H$ are bounded by $q^n\cdot \max_{\mathbf{m}\not\equiv\bfo\,(\mod q)}|\widehat{1_S}(\mathbf{m)}|\ll q^{n-1}p^{-\frac{n-1}{2}}$ by Lemma \ref{lemma_fourier}. In other words, $G$ and $H$ are $\big(q^n, (1+o(1))q^{n-1}, q^{n-1}p^{-\frac{n-1}{2}}\big)$-graphs.

Thus, Theorem \ref{rectangle} follows directly from Theorem \ref{thm:CS} with $A=E$.

\paragraph{Proof of Theorem \ref{chain}:} To prove Theorem \ref{chain}, we follow the same approach as in the proof of Theorem \ref{rectangle}. More precisely, we make use of the following result on $(N, d, \lambda)$-graphs, which is also taken from \cite{TP}.

\begin{theorem}\label{paths}
Let $G$ be a $(N, d, \lambda)$-graph, $\ell\ge 1$ an integer,  and $A$ be a vertex set with $|A|\gg_\ell \lambda \cdot \frac{N}{d} $. Let $P_\ell(U)$ denotes the number of paths of length $\ell$ in $A$. Then we have
\[P_\ell(E)\sim \frac{|A|^{\ell+1}d^\ell}{N^\ell}.\]
\end{theorem}
With the graph $G$ defined as above, Theorem \ref{chain} follows from Theorem \ref{paths} immediately.

\paragraph{Proof of Theorem \ref{cycle}:}
We first recall the statement of Theorem \ref{main}: for $E\subset (\mathbb{Z}/p^r\mathbb{Z})^n$ if $|E|\gg q^np^{-\frac{n-1}{2}}$, then for any $j\in (\mathbb{Z}/p^r\mathbb{Z})^*$, the number of pairs $(\bfx, \bfy)\in E\times E$ such that $F(\bfx-\bfy)\equiv j\,(\mod q)$ is at least $\gg q^{-1}|E|^2$.

 By the pigeon-hole principle, there exists $\mathbf{u}\in (\mathbb{Z}/p^r\mathbb{Z})^n$ such that $F(\mathbf{u})\equiv j\,(\mod q)$ and the number of pairs $(\bfx, \bfy)\in E\times E$ such that $\bfx=\bfy+\mathbf{u}$ is at least $\gg \frac{|E|^2}{q^n}$. Let $E_1\subset E$ be the set of $\bfy\in E$ such that $\bfy+\bfu\in E$. Without loss of generality, we assume that $|E_1|=|E|^2/q^n$. By Theorem \ref{main} again, the number of pairs $(\bfy_1, \bfy_2)\in E_1\times E_2$ such that $F(\bfy_1-\bfy_2)\equiv j\,(\mod q)$ is at least $|E_1|^2/q=|E|^4/q^{2n+1}$. Among these pairs, the number of pairs with $\bfy_1=\bfy_2+\mathbf{u}$ is at most $|E_1|$, which is smaller than $|E_1|^2/q$. Therefore, by the pigeon-hole principle again, there exists $\mathbf{u}'\in (\mathbb{Z}/p^r\mathbb{Z})^n$ such that $\mathbf{u}\ne \mathbf{u}'$, $F(\mathbf{u}')\equiv j\,(\mod q)$, and the number of pairs $(\bfy_1, \bfy_2)\in E_1\times E_1$ such that $\bfy_1=\bfy_2+\mathbf{u}'$ is at least $\gg \frac{|E|^2}{q^n}>0$.

 With each pair $(\bfy_1, \bfy_2)$ satisfying the above property, we have a cycle of the from $(\bfy_1, \bfy_2, \bfy_2+\mathbf{u}, \bfy_1+\mathbf{u})$. This completes the proof.

\paragraph{Proof of Theorem \ref{tree}:}
To prove Theorem \ref{tree}, we need a stronger version of Theorem \ref{main}, namely, the pinned distance version.
\begin{theorem}\label{pin}
    Let $E_1, E_2\subset (\mathbb{Z}/p^r\mathbb{Z})^n$ with $|E_1|, |E_2|\gg q^np^{-\frac{n-1}{2}}$, then there exists $E_1'\subset E_1$ with $|E_1'|\gg |E_1|$ and $|\Delta_{n.r,\bfx}(E_2)|\gg q$ for all $\bfx\in E_1'$. Here
    \[\Delta_{n.r,\bfx}(E_2):=\{F(\bfx-\bfy)\colon \bfy\in E_2\}.\]
\end{theorem}
To prove this strong version, we introduce an incidence problem between points and $F$-spheres in $(\mathbb{Z}/p^r\mathbb{Z})^n$.

A $F$-sphere centered at $x$ of radius $j$ is the set
\[\{y\in (\mathbb{Z}/p^r\mathbb{Z})^n\colon F(x-y)\equiv j\mod q\}.\]

Let $\cP$ be a set of points in $(\mathbb{Z}/p^r\mathbb{Z})^n$ and $\cS$ be a set of $F$-spheres. The number of incidences between $\cP$ and $\cS$, denoted by $I(\cP, \cS)$, is defined by
\[I(\cP, \cS)=\#\{(p, s)\in \cP\times\cS\colon p\in s\}.\]
\begin{theorem}\label{thm-incidence}
    Let $\cP$ be a set of points in $(\mathbb{Z}/p^r\mathbb{Z})^n$ and $\cS$ be a set of $F$-spheres with radii in $(\mathbb{Z}/p^r\mathbb{Z})^*$. Then the number of incidences between $\cP$ and $\cS$ satisfies
    \[I(\cP, \cS)\le \frac{|\cP||\cS|}{q}+Cq^{n-\frac{1}{2}}p^{-\frac{n-1}{2}}|\cP|^{1/2}|cS|^{1/2}.\]
\end{theorem}
\begin{proof}
    We partition $\cS$ into $\cS_i$, $i\in (\mathbb{Z}/p^r\mathbb{Z})^*$, such that spheres in $\cS_i$ are of radius $i$. By abuse of notation, we denote the set of centers in $\cS_i$ by $\cS_i$. By Theorem \ref{main}, we have
    \[I(\cP, \cS_i)\le \frac{|\cP||\cS_i|}{q}+q^{n-1}p^{-\frac{n-1}{2}}|\cP|^{1/2}|\cS_i|^{1/2}.\]
Taking the sum over all $i$ and by using the Cauchy-Schwarz inequality, we have the conclusion.
\end{proof}

 \begin{proof}[Proof of Theorem \ref{pin}]
 For $\bfx\in E_1$, let $\Delta_{n,r,{\bfx}}^*(E_2)$ be the set of distances in $(\mathbb{Z}/p^r\mathbb{Z})^*$ between $\bfx$ and points in $E_2$. Let $\cS(\bfx)$ be the set of spheres centered at $\bfx$ of radii in $\Delta_{n.r,\bfx}^*(E_2)$, and set $\cS=\cup_{\bfx\in E_1}\cS(\bfx)$. We have $|\cS|=\sum_{\bfx\in E_1}|\Delta_{n,r,{\bfx}}^*(E_2)|$. Under the conditions on the sizes of $E_1$ and $E_2$, for any $j\in (\mathbb{Z}/p^r\mathbb{Z})^*$, it follows from Theorem \ref{main} that the number of pairs $(\bfx, \bfy)\in E_1\times E_2$ such that $F(\bfx-\bfy)\equiv j\,(\mod q)$ is at least $|E_1||E_2|/2q$.

Thus, $I(E_2, \cS)\ge |E_1||E_2|/4$. Applying Theorem \ref{thm-incidence} gives
\[I(E_2, \cS)\le \frac{|E_2|\sum_{\bfx\in E_1}|\Delta_{\bfx}(E_2)^*|}{q}+Cq^{n-\frac{1}{2}}p^{-\frac{n-1}{2}}|E_2|^{1/2}\left(\sum_{\bfx\in E_1}|\Delta_{n.r,\bfx}(E_2)^*|\right)^{1/2}.\]
If $\sum_{\bfx\in E_1}|\Delta_{n.r,\bfx}(E_2)^*|\le |E_1|q/8$, then $|E_1||E_2|\le 8C^2q^{2n}p^{-(n-1)}$, so we reach a contradiction as long as $|E_1||E_2|> 8C^2q^{2n}p^{-(n-1)}$.

Thus, if $|E_1||E_2|> 8C^2q^{2n}p^{-(n-1)}$, then we have
\begin{equation}\label{eq:contradiction}\sum_{\bfx\in E_1}|\Delta_{n.r,\bfx}(E_2)^*|>\frac{|E_1|q}{8}.\end{equation}

Set $E_1':=\{\bfx\in E_1\colon |\Delta_{\bfx}(E_2)^*|\ge q/32\}$. We are going to show that $|E_1'|\ge |E_1|/32$. Indeed, otherwise, we have
\[\sum_{\bfx\in E_1}|\Delta_{\bfx}(E_2)^*|=\sum_{\bfx\in E_1'}+\sum_{\bfx\in E_1\setminus E_1'}\le q|E_1'|+\frac{q|E_1|}{32}\le \frac{q|E_1|}{16},\]
which contradicts (\ref{eq:contradiction}). Therefore, we have proved that there exists $E_1'\subset E_1$ with $|E_1'|\ge |E_1|/32$ and for all $\bfx\in E_1'$, we have $|\Delta_{n.r,\bfx}(E_2)^*|\ge q/32$. The proof is complete.
\end{proof}

With Theorem \ref{pin} in hand, Theorem \ref{tree} now follows from a combinatorial argument which is identical to \cite[Proof of Theorem 3.1]{DM} with Theorem \ref{pin} in place of \cite[Lemma 2.2]{DM}. So we omit the details.

\section{Sharpness examples}

In this section, we provide some sharpness examples for results stated in the introduction.

\begin{example} \label{ex1}
Let $p$ be a large prime. Suppose that  $n$ is odd and $-1$ is a $k$-th power modulo $p$, i.e., $\xi^k\equiv -1\, (\mod p)$. We consider the polynomial $F(\bfx)=x_1^k+x_2^k+\ldots+x_n^k$.

For $r=1$, let us take
\[
\Pi_0=\{(u, \xi u):\, u\in \bZ/p\bZ\}.
\]
Then, for any $(u,\xi u)\in \Pi_0$, one has $\|(u,\xi u)\|=u^k+(\xi u)^k \equiv 0\, (\mod p)$.  This is a one-dimensional subspace of $(\bZ/p\bZ)^2$. We have $\|(u,\xi u) - (u',\xi u')\| = 0$ for any $(u,\xi u)$, $(u',\xi u')\in \Pi_0$. Let $1\leq l<p/2$ be a parameter to be determined later.  Take
\[
E_1=\big\{(x_1,x_2,\ldots,x_{n})\in \bZ/p\bZ:\, x_1\in \{1,2,\ldots,l\},\, (x_2,x_3),\ldots,(x_{n-1},x_n)\in \Pi_0\big\}.
\]
Then $|E_1|=lp^{\frac{n-1}{2}}$, and
\[
|\Delta_{n,1}(E_1)| = \#\left\{(-l+1)^k,(-l+2)^k,\ldots,(l-1)^k\right\} < 2l.
\]

For $r\geq 2$, let us take
\[
E_r=\big\{\bfx+p\bfy:\,  \bfx\in E_1,\,\bfy\in (\bZ/p^{r-1}\bZ)^n\big\}
\]
as a subset of $(\bZ/p^r\bZ)^n$. It is not hard to see that $|E_r|= lp^{nr-\frac{n+1}{2}}$. Note that, for any two elements $\bfz=\bfx+p\bfy,\, \bfz'=\bfx'+p\bfy'\in E_r$, we have  $\|\bfz-\bfz'\| \equiv \|\bfx-\bfx'\| \, (\mod p)$. So
\[
|\Delta_{n,r}(E_r)| \leq p^{r-1} \cdot |\Delta_{n,1}(E_1)|<2l p^{r-1}.
\]

To sum up, if we take $l = \lceil p \cdot w(p) \rceil$ with $w(p)=\textit{o}(1)$ $(p\rightarrow \infty)$ being a positive-valued function tending to $0$ arbitrarily slowly, then there are sets $E_r$ with density $p^{-\frac{n-1}{2}}w(p)$ such that $|\Delta_{n,r}(E_r)| < 2p^r w(p) = \textit{o}(p^r)$. So Corollary  \ref{cor_diag_hom} is optimal in general.
\end{example}

\begin{example} \label{ex2}
Given a positive integer $C$, we first claim that there exists a set $E_1\subset (\mathbb{Z}/p\mathbb{Z})^n$ with {\color{black}$\delta_{E_1} \geq C^{-1}p^{-n/2}$} such that $\Delta_{n, 1}(E_1)=\{0\}$ or $|\Delta_{n, 1}(E_1)|\le (p+1)/C$.

If $-1$ is a square or $n\equiv 0\,(\mod 4)$, as above, we can find such a set $E_1\subset (\mathbb{Z}/p\mathbb{Z})^n$ of size $p^{n/2}$ such that $\Delta_{n, 1}(E_1)=\{0\}$.

If $-1$ is a non-square and $n\equiv 2\,(\mod 4)$, then let $\theta$ be a generator of the group $G_1=SO_2(\bF_p)$, which is cyclic of order $p+1$, and let $v_0\in (\mathbb{Z}/p\mathbb{Z})^2$ such that $||v_0||=1$. By the Dirichlet's theorem on arithmetic progressions, one can choose $p$ large enough such that $C$ divides $p+1$. Let $G$ be the subgroup spanned by $\theta^C$. Then we have $|G|=(p+1)/C$. Set $X=\textsf{orb}_1(\bfv_0)$. A direct computation shows that $X$ satisfies the desired property, namely, $|\Delta_{2, 1}(X)|\le |X|=(p+1)/C$. Let $X'\subset (\mathbb{Z}/p^r\mathbb{Z})^{n-2}$ of size $p^{(n-2)/2}$ such that $\Delta_{n, 1}(X')=\{0\}$. Set $E_1=X\times X'$. Then we have $|E_1|=p^{\frac{n-2}{2}}\cdot \frac{p+1}{C}$. By the definition of $E_1$, we obtain $|\Delta_{n, 1}(E_1)|\le (p+1)/C$.

For $r\ge 2$, we set
\[E_r=\{\mathbf{x}+((\mathbb{Z}/p^r\mathbb{Z})\setminus U_r)^n\colon \mathbf{x}\in E_1\}.\]
So {\color{black}
\[
\delta_{E_r}=p^{-rn}|E|\geq p^{-rn}p^{\frac{n-2}{2}}\cdot \frac{p+1}{C}\cdot p^{(r-1)n}\geq C^{-1}p^{-n/2},
\]
and
\[
|\Delta_{n, r}(E_r)|\le p^{r-1}|\Delta_{n, 1}(E_1)|\ll p^r/C.
\]
}

\end{example}

\section{Acknowledgements}
Thang Pham would like to thank the Alfr\'{e}d R\'{e}nyi Institute of Mathematics for the hospitality and excellent working
conditions. Thang Pham was partially supported by ERC grant ``GeoScape'', no. 882971, under Prof. J\'{a}nos Pach. We thank Le Quang Hung for checking all proofs in Section 4 for the case $p\equiv 1\, (\mod 4)$.

\end{document}